\documentclass[11pt]{amsart}
\usepackage[dvipsnames]{xcolor}
\usepackage{geometry}
\usepackage[linesnumbered,ruled,vlined]{algorithm2e}
\usepackage{color}
\usepackage[colorlinks=true,linkcolor=red,urlcolor=blue]{hyperref}
\usepackage{amsthm}
\usepackage{graphicx}
\usepackage{times}
\usepackage{xcolor}
\graphicspath{ {./images/} }
\usepackage{tikz}
\usepackage{multirow}
\usepackage{caption}
\usepackage{subcaption}
\usepackage[square,numbers]{natbib}
\usepackage{hyperref}
\usepackage{comment}
\usepackage[colorinlistoftodos]{todonotes}
\usepackage{verbatim}
\hypersetup{
    colorlinks=true,
    linkcolor=blue,
    filecolor=magenta,      
    urlcolor=cyan,
    pdftitle={Overleaf Example},
    pdfpagemode=FullScreen,
}

\usepackage[title]{appendix}
\usepackage{fancyhdr}	
\usepackage[foot]{amsaddr}
\usepackage[dvipsnames]{xcolor}
\usepackage[colorlinks=true,linkcolor=red,urlcolor=blue]{hyperref}
\numberwithin{equation}{section}
\newtheorem{theorem}{Theorem}[section]
\newtheorem{corollary}{Corollary}[section]
\newtheorem{lemma}{Lemma}[section]
\newtheorem{proposition}{Proposition}[section]
\newtheorem{remark}{Remark}[section]

\newtheorem{assumption}{Assumption}[section]
\newtheorem{definition}{Definition}[section]
\newtheorem{assumptions}{Assumptions}[section]

\title[Stationary Singular Control MFGs under Uncertainty]{Stationary Mean-Field Games of Singular Control under Knightian Uncertainty}
\date{\today}
\author[Ferrari]{Giorgio Ferrari $^{\ast,1}$}
\email[Ferrari]{\href{mailto:giorgio.ferrari@uni-bielefeld.de}{$^{1}$giorgio.ferrari@uni-bielefeld.de}}
\author[Tzouanas]{Ioannis Tzouanas $^{\ast,2}$}
\email{\href{mailto:ioannis.tzouanas@uni-bielefeld.de}{$^{2}$ioannis.tzouanas@uni-bielefeld.de}}
\address{$^{\ast}$Center for Mathematical Economics (IMW), Bielefeld University, Universit\"atsstrasse 25, 33615, Bielefeld, Germany.}
\begin{document}

\begin{abstract}
 In this work, we study a class of stationary mean-field games of singular stochastic control under model uncertainty. The representative agent adjusts the dynamics of an It\^o diffusion via one-sided singular stochastic control, aiming to maximize a long-term average reward criterion. The mean-field interaction is of scalar type through the stationary distribution of the population. Due to the presence of uncertainty, the problem involves the study of a stochastic zero-sum game, where the decision maker chooses the best singular control policy, while the adversarial player selects the worst probability measure. Using a constructive approach, we prove existence and uniqueness of a stationary mean-field equilibrium. Finally, we provide a stylized numerical benchmark of dirty-capacity reduction under ambiguity and analyze the impact of uncertainty on the mean-field equilibrium.
\end{abstract}

\maketitle

\subsection*{Keywords:} stationary mean-field games; singular control; model uncertainty; ergodic criterion; free-boundary problem; shooting method.
\subsection*{MSC subject classification:} 49L20, 91A07, 91A15, 91A16, 91A26, 9J15, 60J70, 35R35.


\section{Introduction}
Mean-field games (MFGs in short) were independently introduced in 2006 by Lasry and Lions \cite{Lions}, and by Caines et al.\ \cite{Caines2006LargePS}, as asymptotic models for symmetric \emph{N}-player differential games. In these settings, each player's dynamics and decisions are influenced by the collective behavior of the population, typically represented by the empirical distribution of the states (and potentially the actions) of all players (extended MFGs). The central idea of MFGs is to replace the complex many-player interaction with the problem of a single representative agent who optimizes her strategy in response to a given flow of probability measures, reflecting the statistical distribution of the other, indistinguishable agents.\ The equilibrium concept in MFGs emerges as a consistency condition: The law of the optimally controlled state process of the representative agent must coincide with the prescribed flow of distributions.\ In this way, the classical Nash equilibrium from the \emph{N}-player game is replaced by a fixed-point requirement on the evolution of distributions.\ Since their introduction, MFGs have garnered substantial attention in both the mathematical and applied communities.\ This is due to their analytical tractability, their deep connections with the theory of propagation of chaos and forward-backward stochastic systems, and their capacity to approximate $\varepsilon_{N}$-Nash equilibria in large, symmetric \emph{N}-player games. For a comprehensive treatment of the theory, methodologies, and main results, we refer the reader to the two-volume monograph by Carmona and Delarue \cite{carmona2018probabilistic}. A detailed overview of applications of MFGs in fields such as Economics, Finance, and Engineering can be found in the survey by Carmona \cite{carmona2020applications}.

In stationary MFGs, the representative player interacts with the long-run distribution of the population. Such a concept has a long tradition in economic theory: Stationary equilibria appeared already in the 1980s in the context of games with a continuum of players (see \cite{Hopeynan} and \cite{JOVANOVIC198877}), and also play an important role in the analysis of competitive market models with heterogeneous agents (see, e.g., \cite{PDEinMacro} and \cite{Luttmer}, amongst many others). Closely connected is also the concept of stationary oblivious equilibria, introduced by Adlakha et al.\ in  \cite{ADLAKHA2015269}.
Within the mathematical literature, stationary MFGs have been approached both via analytic and probabilistic methods. Among those papers adopting a partial differential equations (PDE) approach, we refer to the works of Bardi and Feleqi \cite{bardi2016nonlinear} for the study of the forward-backward system arising in stationary MFGs with regular controls, Gomes et al.\ \cite{GOMES201449} for extended stationary MFGs, Cardaliaguet and Porretta \cite{cardaliaguetprotera} for the study of the long-term behavior of the master equation arising in MFG theory, and to Bertucci \cite{BERTUCCI2018165} for the study of stationary mean-field optimal stopping games. On the other hand, a probabilistic approach is followed in a series of recent contributions dealing with stationary MFGs with singular and impulsive controls, see A\"id et al.\ \cite{Ferrari-Basei}, Cannerozzi and Ferrari \cite{cannerozzi2024cooperation}, Cao and Guo \cite{CAO2022995}, Cao et al.\ \cite{CaoFerrariDianetti} and Dianetti et. al. \ \cite{dianetti2023ergodic}.

The aforementioned problems are based on the assumption that agents possess complete certainty regarding the occurrence of system events-that is, the real-world probability measure is perfectly known to them. However, this assumption is unrealistic, as economic and financial models often involve complex mechanisms and multiple sources of uncertainty. A well-known concept that addresses this issue is Knightian uncertainty \cite{knight2002risk} (also referred to as model uncertainty), which describes situations in which the decision maker has incomplete knowledge about the probabilities of various outcomes. To account for this, the concept of ambiguity has been introduced, wherein the decision maker evaluates her objective function by minimizing it over a set of plausible probability measures, commonly referred to as the set of priors. In this context, Gilboa and Schmeidler \cite{GilboaSchmeidler} proposed the max-min expected utility framework, in which the agent maximizes expected utility with respect to the worstcase prior within a suitable set. This approach was further extended by Hansen and Sargent \cite{HansenSargent}, who developed a continuous-time version of the max-min expected utility framework and explored its connection to robust control theory. Over the past two decades, optimization problems under model uncertainty have played a significant role in economics and finance.\ A detailed literature review falls outside the scope of this paper.\ We only would like to mention papers that deal with optimal timing and singular control problems under uncertainty. Among these, optimal stopping problems under ambiguity have been studied by Nutz and Zhang \cite{NutzZhang}, Riedel \cite{RiedelOptimalStoppingMultiplePriors}, and Riedel and Cheng \cite{RiedelChengOptimalStoppingUnderAmbiguity}, among others. For singular control problems, we refer to the works of Chakraborty et al.\ \cite{ChakrabortyCohenYoung}, Cohen \cite{CohenQueueingModelUncertainty}, Cohen et al.\ \cite{CohenHening} Ferrari et. al.\ \cite{FerrariLiRiedel} and Ferrari et al.\ \cite{Ferrari_Li_Riedel_2022}, while Perninge \cite{perninge2023non} examines impulse control problems under model uncertainty.

\subsection{Our Results} In this paper, we study a class of stationary MFGs under model uncertainty, where the underlying state process is a general singularly controlled one-dimensional diffusion. More precisely, the representative agent optimally controls a real-valued It\^o-diffusion through a one-sided singular control in order to maximize an ergodic reward functional while is uncertain about the \textit{real-world model}. To take into account \textit{model uncertainty}, agent maximizes a long-time-average of the time-integral of a running profit function, net of the proportional reward of actions under the worst-case scenario probability measure. The latter can be addressed as a zero-sum game between the agent and an \textit{adverse player} (see, for instance, Cohen et. al.\ \cite{CohenHening}). The mean-field interaction is of scalar type and comes through a real-valued parameter denoted by $\theta$, which, at equilibrium, has to identify with a suitable generalized moment of the stationary distribution of the optimally controlled state process. From the economic point of view, $\theta$ can be interpreted as an aggregate index computed from the stationary distribution of the optimally controlled state. In reduced-form examples, this aggregate index may affect individual profitability through a decreasing profit multiplier; see Remarks \ref{remark: interpretation of benchmark remark} and \ref{remark: general reduced form profit}.

Our first contribution is the solution of the representative player's optimal control problem. In this context, we extend the result of Cohen et al.\ \cite{CohenHening} to a setting that includes running profit and state-dependent proportional reward of control (see Theorem \ref{theorem: Verification result} below). This is achieved by applying the \textit{shooting method}, following an approach similar to that used in \cite{CohenHening}. In our setting, we face challenges similar to those encountered in \cite{CohenHening} (see also \cite{LiangLiuZervos} for the risk-neutral case). Specifically, the potential function $V^{\epsilon}$ solving~\eqref{free-boundary problem} is unbounded from below as $x\downarrow 0$ (cf. Proposition \ref{prop: uniform boundness of phi}). To overcome this obstacle, we introduce a truncated version of our problem, in the same spirit as Section 5 of \cite{LiangLiuZervos} adapted to our setting with model-uncertainty. Using a verification argument (see Proposition \ref{E&U of free boundary problem} and Theorem \ref{theorem: Verification result}), we demonstrate that, for a fixed mean-field parameter, $\theta$, the optimal control is of barrier-type. That is, the optimal control uniquely solves a Skorokhod reflection problem (see e.g., Tanaka \cite{Tanaka}) at endogenously determined barrier, which depend on the level of ambiguity and on the given and fixed mean-field parameter $\theta$.

The next step deals with the construction of the MFG equilibrium and with the proof of its uniqueness.\ To that end, we first show that the process constituted by the optimally controlled diffusion process under the worstcase scenario admits a stationary distribution (cf.\ Proposition \ref{Prop: existence of stationary distribution} below) and the stationary density function has an explicit form (cf. (\ref{eq: stationary distribution})). Clearly, the stationary distribution and its density function depend on the level of ambiguity and fixed mean-field parameter $\theta$, since the optimally controlled state does. In order to proceed with the equilibrium analysis, we thus study the stability of the stationary distribution with respect to $\theta$ and actually prove its continuity with respect to such a parameter (cf.\ Proposition \ref{Prop: existence of stationary distribution}). Further exploiting the connection to the auxiliary boundary value problem, we are then able to show the monotonicity and the boundedness from below of the free-boundary with respect to mean-field parameter (cf. Lemmata \ref{lemma: monotonicity of fb wrt theta} and \ref{lemma: Robust lower bound of free boundary} below) and to determine an invariant compact set where any equilibrium value of $\theta$ (if one exists) should lie. Combining those continuity and compactness results, an application of the Schauder-Tychonof fixed point theorem allows us to prove that there exists a stationary equilibrium (cf.\ Theorem \ref{theorem: existence and uniqueness of MFE}), which is then also proved to be unique.

Finally, we complement our theoretical analysis with a stylized numerical benchmark of dirty-capacity reduction under ambiguity. In this example, the representative firm's state is interpreted as an index of emission-intensive productive capacity or residual fossil-based exposure, evolving as an affine diffusion with mean-reverting drift. The running profit is specified in reduced form and is of power type, with the mean-field parameter capturing the negative effect of aggregate dirty capacity on individual profitability. In this setting, we study the sensitivity of the equilibrium reduction boundary, the stationary distribution, and the aggregate index with respect to the ambiguity parameter and the volatility level. Due to the presence of a quadratic term in the variational inequality (cf.\ \eqref{free-boundary problem}), the equilibrium cannot be solved explicitly. Therefore, we develop a policy iteration algorithm (PIA) to approximate it.

\subsection{Related Literature}
Ergodic singular stochastic control problems for one-dimensional diffusions have been treated in general settings, including state-dependent costs of actions, and with different applications; see \cite{AlvarezHenning}, \cite{LokkaZervosI}, \cite{Menaldi} and \cite{ZervosErgodic}, \cite{Kunwai}, among others. However, in all those papers, model uncertainty is not considered.

Our paper is placed within the recent bunch of literature dealing with MFGs with singular controls by following a probabilistic approach; see A\"id et al.\ \cite{Ferrari-Basei}, Cao and Guo \cite{CAO2022995}, Cao et al.\ \cite{CaoFerrariDianetti}, Campi et al.\ \cite{Campi_et.al}, Cohen and Sun \cite{cohen2024existenceoptimalstationarysingular}, Dianetti et al.\ \cite{Unifying_Sub_MFG}, Dianetti et. al.\ \cite{dianetti2023ergodic}, Denkert and Horst \cite{DenkertHorst}, Fu and Horst \cite{Fu-Horst}, and Guo and Xu \cite{GuoXu}. Amongst those, the work that most relates to ours is by Cao et al.\ \cite{CaoFerrariDianetti}. Cao et al.\ consider in \cite{CaoFerrariDianetti} ergodic MFGs involving a one-dimensional singularly controlled It\^o-diffusion that can be increased via a monotone control process. In contrast to our work, \cite{CaoFerrariDianetti} does not consider model uncertainty. Our paper is also closely related to the works of Chakraborty et al.\ \cite{ChakrabortyCohenYoung}, Cohen \cite{CohenQueueingModelUncertainty}, and Cohen et al.\ \cite{CohenHening}. In particular, like these studies, we consider worstcase scenarios modeled via Kullback-Leibler divergence and employ the \textit{shooting method} to solve the stochastic singular control problem, following the approach in \cite{CohenHening}. However, these papers do not incorporate a mean-field game framework.

We also clearly relate to those works dealing with MFGs involving model uncertainty. Huang and Huang \cite{huang2013mean} consider mean-field linear-quadratic-Gaussian control under model uncertainty. Bauso et. al.\ in \cite{bauso2016robust} focus on robust MFGs and risk-sensitive type MFGs. Furthermore, Langner et. al.\ \cite{langner2024markovnashequilibriameanfieldgames} study Markov-Nash equilibrium (discrete time and state space) under model uncertainty. 

\subsection{Organization of the Paper}
The rest of the paper is organized as follows.\ In Section \ref{section: problem formulation}, we introduce the probabilistic setting and the MFG under study. Next, in Section \ref{section: solution to ergodic zero-sum game}, for a given and fixed mean-field parameter, we solve the ergodic stochastic control problem faced by the representative player. In Section \ref{section: MFE part} we then prove the existence and uniqueness of the mean-field equilibrium, while in Section 5 we present a stylized numerical benchmark of dirty-capacity reduction under ambiguity. Finally, technical proofs are collected in Appendix \ref{section: appendix}.

\section{Problem Formulation}
\label{section: problem formulation}
\subsection{Probabilistic setting} Let $(\Omega,\mathcal{F},\mathbb{P})$ be a probability space which satisfies the usual conditions, on which it is defined a one-dimensional Brownian motion $\{W_{t}\}_{t\geq 0}$ and denote by $\mathbb{F}:=\{\mathcal{F}^{W}_{t}\}_{t\geq 0}$ the filtration which is generated by $W$, as usual augmented by $\mathbb{P}$-null sets of $\mathcal{F}$. Let
\begin{equation}
    \label{admissible controls}
    \mathcal{A}:=\{\;\{\xi_{t}\}_{t\geq 0},\;\mathbb{F}\text{-adapted, nondecreasing, right continuous and such that }\xi_{0^{-}}=0,\text{ a.s.}\;\},
\end{equation}
and set $\mathbb{R}:=(-\infty,\infty)$ and $\mathbb{R}_{+}:=(0,\infty)$. Then, for given $\xi\in\mathcal{A}$ and Borel-measurable functions $b:\mathbb{R}_{+}\to\mathbb{R}$, $\sigma:\mathbb{R}_{+}\to\mathbb{R}_{+}$, we introduce the $\mathbb{R}_{+}$-valued process $X^{\xi}$ with dynamics under $\mathbb{P}$
\begin{equation}
\label{SDE}
dX^{\xi}_{t}=b(X^{\xi}_{t})dt+\sigma(X^{\xi}_{t})dW^{\mathbb{P}}_{t}-d\xi_{t},\quad X^{\xi}_{0^{-}}=x\in\mathbb{R}_{+}.
\end{equation}

The next assumption ensures, in particular, that there exists a unique strong solution to (\ref{SDE}) for every $\xi\in\mathcal{A}$ and $x\in\mathbb{R}_{+}$ (see Theorem 7 Chapter V in \cite{protter2005stochastic}). In the following, we shall denote such a strong solution by $X^{x,\xi}$, when needed.

\begin{assumption}
\label{Assumption for dynamics of SDE}
The following hold:
        \begin{enumerate}
            \item \label{Cont drift and vol} The functions $b$ and $\sigma$ are continuously differentiable.
            \item \label{Linear growth of drift and vol} There exists $C>0$ and $\zeta>0$, such that
                    \[
                        |b(x)|+|\sigma(x)|\leq C(1+|x|^{\zeta}),\quad \text{for any } x\in\mathbb{R}_{+}.
                    \]
        \end{enumerate}
\end{assumption}
Denoting by $X^{0}$ the unique strong solution to (\ref{SDE}) with $\xi\equiv 0$, Conditions (\ref{Cont drift and vol}) and (\ref{Linear growth of drift and vol}) in Assumption \ref{Assumption for dynamics of SDE} imply that $X^{0}$ is regular, meaning that, for any $x,y\in\mathbb{R}_{+}$, $\mathbb{P}_{x}(\beta_{y}<\infty)>0$, where $\beta_{y}:=\inf\{t\geq 0: X^{0}_{t}=y\},\;\mathbb{P}_{x}$-a.s. For further details see Section II-1 in \cite{BorodinSalminen2012}.

 For arbitrary $x_{0}>0$, we define the scale function of $X^{0}$ under $\mathbb{P}$ as
\begin{equation}
    \label{eq: Scale function under P}
    S^{\mathbb{P}}(x):=\int_{x_{0}}^{x}S^\mathbb{P}_{x}(y)dy\quad\text{with}\quad S^{\mathbb{P}}_{x}(x):=\exp\bigg( -\int_{x_{0}}^{x}\frac{2b(x)}{\sigma^{2}(y)}dy \bigg),\quad x\in\mathbb{R}_{+},
\end{equation}
and the speed measure under $\mathbb{P}$
\begin{equation}
    \label{eq: Speed measure under P}
    m^{\mathbb{P}}((x_{0},x)):=\int_{x_{0}}^{x}\frac{2}{\sigma^{2}(y)S^{\mathbb{P}}_{x}(y)}dy,\quad x\in \mathbb{R}_{+}.
\end{equation}

We make the following standing assumption, which ensures that $0$ is not attainable for $X^{0}$.
\begin{assumption}
   \label{Ass: Non explosion}
   The scale function $S^{\mathbb{P}}$ under $\mathbb{P}$ of $X^{0}$ satisfies
   \begin{equation}
       \lim_{x\downarrow 0}S^{\mathbb{P}}(x)=-\infty,\quad \lim_{x\uparrow \infty}S^{\mathbb{P}}(x)=\infty.
   \end{equation}
   Moreover, we assume that
   \begin{equation}
       \label{eq: integrability condition}
       \int_{0}^{\infty}m^{\mathbb{P}}_{x}(x)dx<\infty.
   \end{equation}
\end{assumption}

Assumption \ref{Ass: Non explosion} ensures that $0$ is unattainable in finite time and that $\infty$ is a natural boundary. Consequently, the uncontrolled process does not explode, which implies that it admits a unique pathwise solution. Additionally, the integrability condition guarantees the existence of an invariant measure for the uncontrolled process. The latter plays a crucial role in the verification result as well as in the existence of the mean-field equilibrium (cf. Theorems \ref{theorem: Verification result} and \ref{theorem: existence and uniqueness of MFE} below).
\begin{remark}
    \label{remark: benchmark SDE}
    A mean-reverting uncontrolled process $X^{0}$ with $b(x)=\kappa-\alpha x$ and $\sigma(x)=\sigma x$, for positive $\kappa,\alpha$ and $\sigma$, satisfies Assumption \ref{Assumption for dynamics of SDE} and \ref{Ass: Non explosion}. 
\end{remark}
We also define $\mathcal{Q}$ to be the set of probability measures on $(\Omega,\mathcal{F})$, which are equivalent with respect to $\mathbb{P}$, i.e.\ $\mathcal{Q}:=\{\mathbb{Q}\in \mathcal{P}(\Omega,\mathcal{F}):\mathbb{Q}\sim \mathbb{P}\}$, where $\mathcal{P}(\Omega,\mathcal{F})$ is the set of probability measures on $(\Omega,\mathcal{F})$ endowed with the weak topology. 
In the rest of the paper, we adopt the following notation: $\mathbb{P}_{x}[\cdot ]:=\mathbb{P}[\cdot|X^{\xi}_{0}=x]$ and $\mathbb{E}^{\mathbb{P}}_{x}[\cdot]:=\mathbb{E}^{\mathbb{P}}[\cdot|X^{\xi}_{0}=x]$. Also, for $\mathbb{Q}\in\mathcal{Q}$, we set $\mathbb{Q}_{x}[\cdot ]:=\mathbb{Q}[\cdot|X^{\xi}_{0}=x]$ and $\mathbb{E}^{\mathbb{Q}}_{x}[\cdot]:=\mathbb{E}^{\mathbb{Q}}_{x}[\cdot|X^{\xi}_{0}=x]$.
\subsection{The ergodic mean-field game} Within the previous probabilistic setting, we now introduce
the ergodic mean-field game (ergodic MFG for short) which will be the main object of our study. We start with the definition of admissible controls.

\begin{definition}[Admissible controls]
    \label{def: set of admissible controls}
    Let $x\in\mathbb{R}_{+}$. We say that $\xi\in\mathcal{A}$ is admissible if it is such that $X^{\xi}_{t}\geq 0,\; \mathbb{P}$-a.s. for any $t\geq 0$ and also
    \begin{equation}
        \sup_{t\in [0,T]}\mathbb{E}^{\mathbb{P}}[|X_{t}^{\xi}|^{2\zeta}]<\infty,\quad \text{for any }T<\infty.
    \end{equation}
   We denote the set of admissible controls by $\mathcal{A}_{e}(x)$.
\end{definition}



For a Borel-measurable function $g:\mathbb{R}\to\mathbb{R}$ (such that the subsequent quantities are well-posed), introduce the integral (cf. \cite{ZervosErgodic}, \cite{Zhu}, among others)\begin{equation}
    \label{propotional costs integral}
    \int_{0}^{t}g(X^{\xi}_{s})\circ d\xi_{s}:=\int_{0}^{t}g(X_{s}^{\xi})d\xi^{c}_{s}+\sum_{0\leq s \leq t}\int_{0}^{\Delta\xi_{s}}g(X^{\xi}_{s}-r)dr,\quad t\geq 0,
\end{equation}
where $\xi^{c}$ denotes the continuous part of $\xi$. 
Then, for any $(\xi,\mathbb{Q})\in\mathcal{A}_{e}(x)\times\mathcal{Q}$, the profit functional to be optimized is given by

\begin{align}
    \label{Ergodic criterion}
    J^{\epsilon}(x;\xi,\mathbb{Q},\theta):&=\liminf_{T\to\infty}\frac{1}{T}\mathbb{E}^{\mathbb{Q}}_{x}\bigg[\int_{0}^{T}\pi(X^{\xi}_{t},\theta)dt+\int_{0}^{T}c(X^{\xi}_{t})\circ d\xi_{t}+\frac{1}{\epsilon}\log\bigg(\frac{d\mathbb{Q}}{d\mathbb{P}}\bigg|_{\mathcal{F}_{T}}\bigg) \bigg],
\end{align}
for $\epsilon>0$ and $\theta\in\mathbb{R}_{+}$. For $\epsilon=0$, the (risk-neutral) profit functional is given by
\begin{equation}
    \label{eq: Risk-neutral ergodic criterion}
    J^{0}(x;\xi,\theta):=\liminf_{T\to\infty}\frac{1}{T}\mathbb{E}^{\mathbb{P}}_{x}\bigg[\int_{0}^{T}\pi(X^{\xi}_{t},\theta)dt+\int_{0}^{T}c(X^{\xi}_{t})\circ d\xi_{t}\bigg].
\end{equation}

In (\ref{Ergodic criterion}) and (\ref{eq: Risk-neutral ergodic criterion}), $\pi$ is the instantaneous profit function and $c$ the proportional reward function satisfying Assumption \ref{Assumptions for profit fun} below.

\begin{assumption}
    \label{Assumptions for profit fun}
    The function $\pi:\mathbb{R}^{2}_{+}\mapsto \mathbb{R}_{+}$ and $c:\mathbb{R}_{+}\to\mathbb{R}_{+}$ are such that:
    \begin{enumerate}
        \item \label{Regularity of profit fun} $\pi(\cdot,\theta)\in C^{2}(\mathbb{R}_{+})$, for any $\theta\in\mathbb{R}_{+}$; 
        \item \label{Concativity} $\pi(\cdot,\theta)$ is non-decreasing and concave, for any $\theta\in\mathbb{R}_{+}$;
        
        \item \label{proportional cost ass} $c$ is continuously differentiable and bounded, i.e. there exists $0<\underline{c}\leq\overline{c}<\infty$ such that $c(x)\in [\underline{c},\overline{c}]$ for any $x\in\mathbb{R}_{+}$.\ Moreover, $\lim_{x\downarrow 0}c(x)$ and $\lim_{x\uparrow\infty}c(x)$ exist.
    \end{enumerate}
\end{assumption}



The parameter $\epsilon\geq 0$ in (\ref{Ergodic criterion}) measures the level of ambiguity that the decision maker has towards the probabilistic model $\mathbb{Q}$ with respect to the reference probabilistic setting associated to $\mathbb{P}$. The following admissibility conditions clarify the structure of the Radon-Nikodym derivative $\frac{d\mathbb{Q}}{d\mathbb{P}}\big|_{\mathcal{F}_{t}}$ in (\ref{Ergodic criterion}) (see also Definition 1 in \cite{CohenHening}).

\begin{definition}[Admissible measures]
	    \label{Admissible strategies}
		Given $x\in\mathbb{R}_{+}$, we say that $\mathbb{Q}\in\mathcal{Q}$ is admissible if
		\begin{equation*}
			\frac{d\mathbb{Q}}{d\mathbb{P}}(t):=\frac{d\mathbb{Q}}{d\mathbb{P}}\bigg|_{\mathcal{F}_{t}}=\exp{\bigg(\int_{0}^{t}\psi(X^{x,\xi}_{s})dW_{s}-\frac{1}{2}\int_{0}^{t}\psi^{2}(X^{x,\xi}_{s})ds \bigg)},\quad \xi\in\mathcal{A}_{e}(x),
		\end{equation*}   
		where $\psi:\mathbb{R}\to\mathbb{R}$ is a bounded Borel-measurable locally Lipschitz function such that
		\begin{equation}
            \label{eq: Novikov condition}
			\mathbb{E}^{\mathbb{P}}_{x}\bigg[ \exp{\bigg(\frac{1}{2}\int_{0}^{T}\psi^{2}(X^{\xi}_{s})ds \bigg)} \bigg]<\infty,\quad\xi\in\mathcal{A}_{e}(x),\;\text{ for any }T<\infty,
		\end{equation}
        the scale function of $X^{0}$ with respect to $\mathbb{Q}$, i.e.
        \begin{equation}
            \label{eq: speed measure wrt Q}
            S^{\mathbb{Q}}_{x}(x):=\exp{\bigg( -\int_{c}^{x}\frac{2(b(y)+\psi(y)\sigma(y))}{\sigma^{2}(y)}dy \bigg)},\quad x\in\mathbb{R}_{+},
        \end{equation}
        is well-defined. 
\end{definition}

In the following analysis, we shall refer to $\psi_{t}=\psi(X^{\xi}_{t}),\; \xi\in\mathcal{A}_{e}(x)$, as the Girsanov kernel of $\mathbb{Q}$, and will denote the set of admissible $\mathbb{Q}$ as $\widehat{\mathcal{Q}}(x)\subseteq \mathcal{Q}$. Let then $\mathcal{S}(x):=\mathcal{A}_{e}(x)\times\widehat{\mathcal{Q}}(x)$.
By choosing an admissible $\mathbb{Q}\in\widehat{\mathcal{Q}}(x)$, with Girsanov kernel $\psi_{t}:=\psi(X^{\xi}_{t}),\; \xi\in\mathcal{A}_{e}(x),\; x\in\mathbb{R}_{+}$, we have
\begin{align*}
    \mathbb{E}^{\mathbb{Q}}_{x}\bigg[ \log\bigg(\frac{d\mathbb{Q}}{d\mathbb{P}}(T)\bigg)\bigg]&=\mathbb{E}^{\mathbb{Q}}_{x}\bigg[ \log\bigg(\exp{\bigg(\int_{0}^{T}\psi(X_{s}^{\xi})}dW_{s}-\frac{1}{2}\int_{0}^{T}\psi(X_{s}^{\xi})^{2}ds \bigg)\bigg)\bigg] \\
    &=\mathbb{E}^{\mathbb{Q}}_{x}\bigg[\int_{0}^{T}\psi(X_{s}^{\xi})dW_{s}-\frac{1}{2}\int_{0}^{T}\psi(X_{s}^{\xi})^{2}ds \bigg] \\
    &=\mathbb{E}^{\mathbb{Q}}_{x}\bigg[\bigg(\int_{0}^{T}\psi(X_{s}^{\xi})\underbrace{(dW_{s}-\psi(X_{s}^{\xi})ds}_{=dW^{\mathbb{Q}}})+\frac{1}{2}\int_{0}^{T}\psi(X_{s}^{\xi})^{2}ds\bigg) \bigg] \\
    &=\mathbb{E}^{\mathbb{Q}}_{x}\bigg[\frac{1}{2}\int_{0}^{T}\psi(X_{t}^{\xi})^{2}dt\bigg],
\end{align*}
so that for $(\xi,\mathbb{Q})\in\mathcal{S}(x)$ the payoff functional (\ref{Ergodic criterion}) rewrites as
\begin{equation}
    \label{second version of functional}
    J^{\epsilon}(x;\xi,\mathbb{Q},\theta):=\liminf_{T\to\infty}\frac{1}{T}\mathbb{E}^{\mathbb{Q}}_{x}\bigg[\int_{0}^{T}\pi(X^{\xi}_{t},\theta)dt+\int_{0}^{T}c(X^{\xi}_{t})\circ d\xi_{t}+\frac{1}{2\epsilon}\int_{0}^{T}\psi(X_{t}^{\xi})^{2}dt \bigg].
\end{equation}
The positive sign of the entropy term reflects the max-min structure of the problem. The representative agent maximizes the criterion over singular controls, while Nature minimizes it over equivalent probability measures. Hence, the entropy term penalizes Nature for choosing probability measures far from the reference model. As $\epsilon\downarrow0$, this penalty becomes increasingly severe, forcing the minimizing measure to remain close to the reference measure $\mathbb{P}$; in this formal sense, the risk-neutral criterion is recovered.

The parameter $\theta$ drives the interaction of the representative player with the population (see Definition \ref{Def Mean-Field equilibrium} below).\ Its representation at equilibrium involves the functions $F:\mathbb{R}_{+}\to\mathbb{R}_{+}$ and $f:\mathbb{R}_{+}\to\mathbb{R}_{+}$ that satisfy the following requirements.
\begin{assumption}
    \label{Ass: Regularity of Mean-Field term}
    $F:\mathbb{R}_{+}\to\mathbb{R}_{+}$, $f:\mathbb{R}_{+}\to\mathbb{R}_{+}$ are such that:
    \begin{enumerate}
        \item \label{Ass: Regularity of Mean-Field term-1} $F$ and $f$ are strictly increasing and continuously differentiable functions;
        \item  for $\delta\in(0,1)$, there exists $C>0$ such that:
        \begin{enumerate}
            \item \label{Ass: Regularity of Mean-Field term-2-a}
            $
                |f(x)|\leq C(1+|x|^{\delta}),\quad |F(x)|\leq C(1+|x|^{\frac{1}{\delta}}),
            $
            \item \label{Ass: Regularity of Mean-Field term-2-b}
            $
                \big| F(x)-F(y) \big|\leq C(1+|x|+|y|)^{\frac{1}{\delta}-1}|x-y|;
            $
            
        \end{enumerate}
        \item \label{Ass: Regularity of Mean-Field term-3} $\lim_{x\uparrow\infty}F(x)=\lim_{x\uparrow\infty}f(x)=\infty$.

    \end{enumerate}
\end{assumption}
\begin{remark}
    \label{remark: benchmark example}
    As a benchmark example (see Remark \ref{remark: interpretation of benchmark remark} below) we may take 
    \begin{equation}
        \label{eq: benchmark example}
        \pi(x,\theta)=x^{\delta}\theta^{-(1+\delta)},\quad c(x)=c,\quad f(x)=x^{\delta},\quad\text{and,}\quad F(x)=x^{1/\delta}
    \end{equation}
    for $\delta\in (0,1)$, $c>0$. In this case, Assumptions \ref{Assumptions for profit fun} and \ref{Ass: Regularity of Mean-Field term} hold. 
\end{remark}

The following definition finally introduces the notion of optimality for the considered MFG.
\begin{definition}[Ergodic MFG Equilibrium]
    \label{Def Mean-Field equilibrium}
        For $x\in\mathbb{R}_{+}$ and $\epsilon\geq 0$, a tuple $(\xi^{*}(\theta^{*}),\mathbb{Q}^{*}(\theta^{*}),\theta^{*})\in\mathcal{S}(x)\times\mathbb{R}_{+}$ is said to be an \textbf{equilibrium of the ergodic MFG} for the initial condition $x\in\mathbb{R}_{+}$ if
        \begin{enumerate}
            \item \label{cond: optimality} \begin{enumerate}
                \item \label{eq:Def 1-1} $J^{\epsilon}(x;\xi^{*}(\theta^{*}),\mathbb{Q}^{*}(\theta^{*}),\theta^{*})\geq J^{\epsilon}(x;\xi,\mathbb{Q}^{*}(\theta^{*}),\theta^{*}),\text{ for any }\xi\in\mathcal{A}_{e}(x)$;
                \item \label{eq:Def 1-1-0} $J^{\epsilon}(x;\xi^{*}(\theta^{*}),\mathbb{Q}^{*}(\theta^{*}),\theta^{*})\leq J^{\epsilon}(x;\xi^{*}(\theta^{*}),\mathbb{Q},\theta^{*}),\text{ for any }\mathbb{Q}\in\widehat{\mathcal{Q}}(x)$.
        \end{enumerate}
        \item \label{cond: consistency} The optimally controlled state process $X^{\xi^{*}(\theta^{*})}$ under $\mathbb{Q}^{*}(\theta^{*})$ admits a stationary distribution $\nu^{\theta^{*}}$ and the consistency condition $\theta^{*}=F\big( \int_{\mathbb{R}_{+}}f(x)\nu^{\theta^{*}}(dx) \big)$ holds true.
        \end{enumerate}
\end{definition}

In order to solve the ergodic MFG we follow a three-step approach:
\begin{enumerate}
    \item For $\theta\in\mathbb{R}_{+}$, we find $\xi^{*}(\theta)$ and $\mathbb{Q}^{*}(\theta)$ satisfying Definition \ref{Def Mean-Field equilibrium}-(\ref{cond: optimality}) (with $\theta^{*}$ replaced by $\theta$).
    \item We determine the stationary distribution $\nu^{\theta}$ of $X^{\xi^{*}(\theta)}$ under $\mathbb{Q}^{*}(\theta)$.
    \item We solve for the fixed-point problem deriving from the consistency condition as in Definition \ref{Def Mean-Field equilibrium}-(\ref{cond: consistency}).
\end{enumerate}

\begin{remark}
    \label{remark: interpretation of benchmark remark}
The benchmark example in Remark \ref{remark: benchmark example} can be interpreted as a stylized mean-field game of irreversible reductions of emission-intensive productive capacity; see Section \ref{Section: Case study} for details and the numerical illustration. In this interpretation, $X^{0}$ denotes a firm-level index of dirty productive capacity, residual fossil-based exposure, or emission-intensive activity in the absence of active reductions. The dynamics
    \[
        dX^{0}_{t}=(\kappa-\alpha X^{0}_{t})dt+\sigma X^{0}_{t}dW^{0}_{t},
        \quad X^{0}_{0}=x,
    \]
are used as a tractable mean-reverting benchmark specification. Here, $\kappa>0$ represents a baseline maintenance or recovery rate of the capacity index, $\alpha>0$ is the mean-reversion speed, and $\sigma>0$ measures exogenous fluctuations in operating conditions. In the absence of noise and control, the state is pulled toward the level $\kappa/\alpha$.

The singular control represents cumulative irreversible reductions of this dirty-capacity index. The running payoff $\pi$ is specified in reduced form. For instance, the specification
\[
  \pi(x,\theta)=x^{\delta}\theta^{-(1+\delta)}
\]
corresponds to the reduced-form inverse demand $p(\theta)=\theta^{-(1+\delta)}$ applied to effective output $x^\delta$. Thus, the term depending on the mean-field parameter captures the negative effect of aggregate dirty capacity on individual profitability. Finally, $c$ denotes the marginal reward from reducing dirty capacity, such as a retirement value, subsidy, avoided compliance cost, or salvage benefit. This interpretation is intended as a stylized benchmark for the theoretical results, not as a calibrated structural model of the energy sector.
\end{remark}

\begin{remark}
    \label{remark: general reduced form profit}
    The benchmark specification in Remark \ref{remark: benchmark example} can be seen as a particular case of a broader class of reduced-form separable interactions. Given a stationary distribution $\mu$ of the controlled state, define the aggregate mean-field parameter by
    \[
        \theta(\mu)
        :=
        F\left(\int_{\mathbb{R}_{+}} f(y)\mu(dy)\right),
    \]
    where $F$ and $f$ are the increasing functions appearing in Assumption \ref{Ass: Regularity of Mean-Field term}. One may then consider profit functions of the form
    \[
        \pi(x,\mu)
        =
        \Pi(x)G(\theta(\mu)),
    \]
    or, equivalently, writing the interaction through the scalar parameter $\theta$,
    \[
        \pi(x,\theta)=\Pi(x)G(\theta).
    \]
    Here $\Pi:\mathbb{R}_{+}\to\mathbb{R}_{+}$ is increasing and concave, while $G:\mathbb{R}_{+}\to\mathbb{R}_{+}$ is decreasing. The function $F$ determines how the stationary distribution is aggregated into the mean-field parameter, whereas $G$ describes how this aggregate index affects individual profitability.

    For example, as in Section \ref{Section: Case study}, the benchmark choice
    \[
        \Pi(x)=x^{\delta}, \qquad
        G(\theta)=\theta^{-(1+\delta)}, \qquad
        f(x)=x^{\delta}, \qquad
        F(x)=x^{1/\delta},
    \]
    with $\delta\in(0,1)$, yields
    \[
        \pi(x,\theta)=x^{\delta}\theta^{-(1+\delta)}
    \]
    and
    \[
        \theta
        =
        \left(\int_{\mathbb{R}_{+}}x^{\delta}\mu(dx)\right)^{1/\delta}.
    \]
\end{remark}

\section{Solving the Ergodic Zero-Sum Game}
\label{section: solution to ergodic zero-sum game}

Here, for $\theta\in\mathbb{R}_{+}$ given and fixed, we consider the zero-sum game between a singular controller acting on $\xi\in\mathcal{A}_{e}(x)$ and an adverse player (Nature) choosing $\mathbb{Q}\in\widehat{\mathcal{Q}}(x)$. Then, the aim is to determine

\begin{equation}
    \label{value function}
    \lambda^{\epsilon}(\theta):=\sup_{\xi\in\mathcal{A}_{e}(x)}\inf_{\mathbb{Q}\in\widehat{\mathcal{Q}}(x)}J^{\epsilon}(x;\xi,\mathbb{Q},\theta).
\end{equation}
Due to the ergodic setting, $\lambda^{\epsilon}(\theta)$ is expected to be independent of $x$. To tackle \eqref{value function}, we employ the dynamic programming formulation and introduce a potential function $V^{\epsilon}(\cdot,\theta):\mathbb{R}_{+}\to \mathbb{R}$. This needs to be determined such that $V^{\epsilon}(\cdot,\theta)\in C^{2}(\mathbb{R}_{+})$ and such that the pair $(V^{\epsilon}(\cdot,\theta),\lambda^{\epsilon}(\theta))$ solves the variational inequality
\begin{equation}
    \label{initial free boundary problem}
    \max\bigg\{\inf_{p\in\mathbb{R}}\{\mathcal{L}^{p,\epsilon}V^{\epsilon}(x,\theta)+\frac{1}{2\epsilon}p^{2}\}+\pi(x,\theta)-\lambda^{\epsilon}(\theta),-V^{\epsilon}_{x}(x,\theta)+c(x)\bigg\}=0,\quad x\in\mathbb{R}_{+},
\end{equation}
where, for $p\in\mathbb{R}$ and $f\in C^{2}(\mathbb{R}_{+})$,
\begin{equation}
    \label{infinit generator}
    \mathcal{L}^{p,\epsilon}f(x):=\frac{1}{2}\sigma^{2}(x)f_{xx}(x)+(b(x)+\sigma(x)p)f_{x}(x).
\end{equation}

It is clear that the infimum with respect to $p\in\mathbb{R}$ appearing in (\ref{initial free boundary problem}) is attained. Hence, we let 
\begin{align}
    \label{generator}
    (\mathcal{L}^{\epsilon}f)(x):=\inf_{p\in\mathbb{R}}\{\mathcal{L}^{p,\epsilon}f(x)+\frac{1}{2\epsilon}p^{2}\}&=\frac{1}{2}\sigma^{2}(x)f_{xx}(x)+b(x)f_{x}(x)-\frac{\epsilon}{2}\sigma^{2}(x)(f_{x}(x))^{2},
\end{align}
and (\ref{initial free boundary problem}) rewrites as
\begin{equation}
    \label{free-boundary problem}
    \max\{\mathcal{L}^{\epsilon}V^{\epsilon}(x,\theta)+\pi(x,\theta)-\lambda^{\epsilon}(\theta),-V^{\epsilon}_{x}(x,\theta)+c(x)\}=0,\quad x\in\mathbb{R}_{+}.
\end{equation}
Equation (\ref{free-boundary problem}) constitutes a variational inequality subject to a gradient constraint.\ In the standard context of singular control, such problems are typically resolved by introducing an auxiliary optimal stopping problem, whose value function satisfies a variational inequality with a value constraint; the integral of this auxiliary value function then provides the solution to the initial problem (see, for instance, \cite{KaratzasShreveMonotoneFollower}).\ However, in the present framework, this direct reduction is not feasible due to the non-linear quadratic term in (\ref{free-boundary problem}). Instead, we hinge on the study of an auxiliary problem specifically for the gradient of the potential function.
]To that end, it is important to introduce the function $\ell^{\epsilon}:\mathbb{R}_{+}\to \mathbb{R}$, defined as
\begin{equation}
    \label{function for eigenvalue}
    \ell^{\epsilon}(x,\theta):=b(x)c(x)+\pi(x,\theta)-\frac{1}{2}\sigma^{2}(x)(\epsilon c^{2}(x)-c_{x}(x)),
\end{equation}
which satisfies the following assumption.
\begin{assumption}The following hold:
    \label{Assumption for l}

    \begin{enumerate}
        \item For any $\epsilon\geq 0$, there exist $\widehat{x}_{\epsilon}(\theta)\in\mathbb{R}_{+}$ such that
                \begin{equation}
                    \label{Condition of lambdas der 1}
                    (\ell^{\epsilon})_{x}(x,\theta)\begin{cases}
                    >0,\quad x<\widehat{x}_{\epsilon}(\theta) \\
                    =0,\quad x=\widehat{x}_{\epsilon}(\theta)\\
                    <0,\quad x>\widehat{x}_{\epsilon}(\theta).
                    \end{cases}
                \end{equation}
        \item One has that
        \begin{equation}
            \label{limit assumption}
            \lim_{x\uparrow\infty}\ell^{\epsilon}(x,\theta)=-\infty\quad\text{and}\quad\ell^{\epsilon}(0,\theta):=\lim_{x\downarrow 0}\ell^{\epsilon}(x,\theta)\text{ is positive and finite}.
        \end{equation}
        \item \label{Zero for ell 1}One has that $\underline{\widehat{x}}_{\epsilon}(\theta):=\inf\{x\geq \widehat{x}_{\epsilon}(\theta):\ell^{\epsilon}(x,\theta)=\ell^{\epsilon}(0,\theta)\}$ is finite.
    \end{enumerate}
\end{assumption}


Assumption \ref{Assumption for l}, will guarantee the optimality of a barrier strategy for (\ref{value function}). Notice that, similar assumptions appear in \cite{AlvarezStationarySingularControl} and \cite{Kunwai}, among others, in the context of ergodic singular stochastic control problems.
\begin{remark}
    \begin{enumerate}
        \item The process of Remark \ref{remark: benchmark SDE} and the setting of Remark \ref{remark: benchmark example} are such that Assumption \ref{Assumption for l} is satisfied.
        \item Another example is the Verhulst-Pearl logistic model with dynamics
    \begin{equation}
        \label{remark: eq: Verhulst-Pearl model}
        dX_{t}=X_{t}(\kappa-\alpha X_{t})dt+\sigma X_{t}dW_{t}^{\mathbb{P}},\quad X_{0}=x,
    \end{equation}
    where $\kappa,\alpha,\sigma>0$. It is clear that Assumption \ref{Assumption for dynamics of SDE} is satisfied and, moreover, one can check that Assumption \ref{Ass: Non explosion} is valid.\ Moreover, for the running profit $\pi$ and proportional reward $c$ as specified in Remark \ref{remark: benchmark example}, it follows that Assumptions \ref{Condition of lambdas der 1} is satisfied.
    \end{enumerate}
\end{remark}
\subsection{Existence of a solution to (\ref{free-boundary problem})}
\label{section: Existence of a solution to First ODE}

To prove existence and uniqueness of a classical solution to (\ref{free-boundary problem}), we follow \textit{the shooting method} as in \cite{CohenHening} (see also Section 7.3 in \cite{StoerNumAn}). To that end, for arbitrarily fixed $\beta\in \mathbb{R}_{+}$ and $\gamma\in\mathbb{R}$, we introduce the following auxiliary boundary-value problem for $\phi_{\beta}^{\gamma}(\cdot,\theta):\mathbb{R}_{+}\to\mathbb{R}$:
\begin{equation}
    \label{Auxiliary Second order ODE}
    \begin{cases}
        \frac{1}{2}\sigma^{2}(x)(\phi_{\beta}^{\gamma})_{x}(x,\theta)+b(x)\phi_{\beta}^{\gamma}(x,\theta)-\frac{\epsilon}{2}\sigma^{2}(x)(\phi_{\beta}^{\gamma})^{2}(x,\theta)=\ell^{\epsilon}(\beta,\theta)-\pi(x,\theta)+\gamma, \quad x<\beta,\\
        \phi_{\beta}^{\gamma}(x,\theta)=c(x),\quad x\geq\beta.
    \end{cases}
\end{equation}

\begin{proposition}[Regular solution to (\ref{Auxiliary Second order ODE})]
    \label{E&U of first order ODE}
    Let Assumptions \ref{Assumption for dynamics of SDE} and \ref{Ass: Non explosion} hold. For fixed $\beta\in \mathbb{R}_{+}$ and $\gamma\in\mathbb{R}$, the boundary-value problem (\ref{Auxiliary Second order ODE}) has a unique solution $\phi_{\beta}^{\gamma}(\cdot,\theta)\in C^{1}(\mathbb{R}_{+})$, for any $\theta\in\mathbb{R}_{+}$.
\end{proposition}
\begin{proof}
To prove that problem (\ref{Auxiliary Second order ODE}) has a unique regular solution, we borrow the argument of Section 4.2 in \cite{CohenHening}.

For $\epsilon=0$, \eqref{Auxiliary Second order ODE} reduces to a linear first-order ODE. Thus, the result follows from Theorem 3.6.2 in \cite{coddington1984theory}.

For $\epsilon>0$, we introduce the function $f:
    \mathbb{R}\to \mathbb{R} $ such that $f(x,\theta):=-\frac{\log{(y(x,\theta))}}{\epsilon}$, where $y(\cdot,\theta):
    \mathbb{R}_{+}\to \mathbb{R}_{+}$ solves
    \begin{equation}
        \label{eq:second order ode cole-hopf}
        \begin{cases}
            \frac{1}{2}\sigma^{2}(x)y_{xx}(x,\theta)+b(x)y_{x}(x,\theta)+(\ell^{\epsilon}(\beta,\theta)-\pi(x,\theta)+\gamma)\epsilon y(x,\theta)=0,\quad x\in (0,\beta) \\
    y(x,\theta)=\frac{1}{c(x)},\quad y_{x}(x,\theta)=-\epsilon,\quad x\geq \beta.
        \end{cases}
    \end{equation}
    Hence, $f_{x}(x)$ solves (\ref{Auxiliary Second order ODE}). In (\ref{eq:second order ode cole-hopf}), all the coefficient are continuous due to the Assumptions (\ref{Assumption for dynamics of SDE}) and (\ref{Assumptions for profit fun}). Hence, by Theorem 3.6.2 in \cite{coddington1984theory}, there exists a unique regular solution on $[\alpha,\infty)$, for every $\alpha\in(0,\beta)$. Since the Cole-Hopf transformation is one-to-one and onto we obtain existence and uniqueness of a solution as in the claim.
\end{proof}
\begin{lemma}[Perturbation of (\ref{Auxiliary Second order ODE})]
    \label{Pertrubed solutions lemma}
    For any $\beta\in \mathbb{R}_{+}$, there exists $C:=C(\beta)$ such that,
    \begin{equation}
        \sup_{x\in (0,\beta]}\big| \phi_{\beta}^{\gamma}(x,\theta)-\phi^{0}_{\beta}(x,\theta)\big|\leq C|\gamma|,\quad \text{for sufficient small }\gamma\in\mathbb{R}.
    \end{equation}
\end{lemma}

\begin{proof}
    The proof follows arguments completely similar to those employed in the proof of Lemma 2 in \cite{CohenHening} and it is therefore omitted.
\end{proof}
In the sequel, we write $\phi_{\beta}$ for $\phi_{\beta}^{0}$, and recall that $\phi_{\beta}(\cdot,\theta)\in C^{1}(\mathbb{R}_{+})$ by Proposition \ref{E&U of first order ODE}. To proceed, we then define
\begin{equation}
    \label{Candidate upper free-boundaries}
    B_{\epsilon}(\theta):=\{\beta\in\mathbb{R}_{+}|\phi_{\beta}(x,\theta)\geq c(x),\; x\in (0,\beta]\},\quad \beta_{\epsilon}(\theta):=\inf B_{\epsilon}(\theta),
\end{equation}
and we have the following result on the structure of the problem's state space.
\begin{proposition}
    \label{properties of b(a)}
    Let Assumptions \ref{Assumption for dynamics of SDE}, \ref{Ass: Non explosion}, \ref{Assumptions for profit fun} and \ref{Assumption for l} hold. Then, one has:
    \begin{enumerate}
        \item \label{hat x is not admissible free boundary}$(0,\widehat{x}_{\epsilon}(\theta)]\cap B_{\epsilon}(\theta)=\emptyset$;
        \item \label{underline x lies in B}$\widehat{\underline{x}}_{\epsilon}(\theta)\in B^{\epsilon}(\theta)$;
        \item \label{infimum of b}$\widehat{x}_{\epsilon}(\theta)\leq \beta_{\epsilon}(\theta)\leq \underline{\widehat{x}}_{\epsilon}(\theta)$.
    \end{enumerate}
\end{proposition}
\begin{proof}
    \textbf{Proof of (1):} Fix $\beta\in(0,\widehat{x}_{\epsilon}(\theta)]$ and $\gamma>0$. We define the function $F^{\gamma}_{\beta}(x,\theta):=\phi_{\beta}^{\gamma}(x,\theta)-c(x)$ and, since for $x=\beta$ we have $\phi_{\beta}^{\gamma}(\beta,\theta)=c(\beta)$, it holds $F^{\gamma}_{\beta}(\beta,\theta)=0$. We then plug $x=\beta$ in (\ref{Auxiliary Second order ODE}) and obtain that
    \begin{align*}
        \frac{1}{2}\sigma^{2}(\beta)&(\phi_{\beta}^{\gamma})_{x}(\beta,\theta)=\ell^{\epsilon}(\beta,\theta)-\big(b(\beta)\phi_{\beta}^{\gamma}(\beta,\theta)-\frac{1}{2}\sigma^{2}(\beta)(\phi_{\beta}^{\gamma})^{2}(\beta,\theta)+\pi(\beta,\theta) \big)+\gamma;
    \end{align*}
    equivalently,
    \begin{align*}
        &\frac{1}{2}\sigma^{2}(\beta)\partial_{x}(\phi_{\beta}^{\gamma}(\cdot,\theta)-c(\cdot))(\beta)=\ell^{\epsilon}(\beta,\theta) \\
        &\quad-\underbrace{\big(b(\beta)c(\beta)+\pi(\beta,\theta)-\frac{1}{2}\sigma^{2}(\beta)\big(\epsilon c^{2}(\beta)-c_{x}(\beta)\big)\big)}_{=\ell^{\epsilon}(\beta,\theta)}+\gamma,
    \end{align*}
    that is,
    \begin{align*}
        \frac{1}{2}\sigma^{2}(\beta)&(F_{\beta}^{\gamma})_{x}(\beta)=\ell^{\epsilon}(\beta,\theta)-\ell^{\epsilon}(\beta,\theta)+\gamma=\gamma>0.
    \end{align*}
    This implies that $(F_{\beta}^{\gamma})_{x}(\beta,\theta)>0$. 
    Arguing by contradiction, we want to show that for any $x\in (0,\beta)$ we have that $F_{\beta}^{\gamma}(x,\theta)<0$, which, thanks to Lemma \ref{Pertrubed solutions lemma}, in turn allows to conclude that $F_{\beta}(x,\theta)\leq 0$ for $x\leq\beta$; that is, $\phi_{\beta}(x,\theta)\leq c(x)$ for $x\leq\beta$. Hence, we assume that there exists $x_{0}(\theta):=\max\{x\in (0,\beta): \phi_{\beta}^{\gamma}(x,\theta)=c(x)\}$. Then,
    \begin{align*}
        \frac{1}{2}\sigma^{2}(x_{0}(\theta))(\phi_{\beta}^{\gamma})_{x}(x_{0}(\theta))&+b(x_{0}(\theta))\phi_{\beta}^{\gamma}(x_{0}(\theta))-\frac{\epsilon}{2}\sigma^{2}(x_{0}(\theta))(\phi_{\beta}^{\gamma})^{2}(x_{0}(\theta))+\pi(x_{0}(\theta))\\
        &=\ell^{\epsilon}(\beta,\theta)+\gamma,
    \end{align*}
    which, rearranging terms, yields
    \begin{align*}
        &\frac{1}{2}\sigma^{2}(x_{0}(\theta))\partial_{x}(\phi_{\beta}^{\gamma}(\cdot,\theta)-c(\cdot))(x_{0}(\theta))=\ell^{\epsilon}(\beta,\theta)\\
        &\quad\quad-\underbrace{\big(b(x_{0}(\theta))c(x_{0}(\theta))+\pi(x_{0}(\theta),\theta)-\frac{1}{2}\sigma^{2}(x_{0})\big(\epsilon c^{2}(x_{0}(\theta))-c_{x}(x_{0}(\theta))\big)\big)}_{=\ell^{\epsilon}(x_{0}(\theta),\theta)}+\gamma.
    \end{align*}
    Hence, thanks to (\ref{Condition of lambdas der 1}) in Assumption \ref{Assumption for l},
    \begin{align*}
        \frac{1}{2}\sigma^{2}(x_{0}(\theta))(F_{\beta}^{\gamma})_{x}(x_{0}(\theta),\theta)=\underbrace{\big(\ell^{\epsilon}(\beta,\theta)-\ell^{\epsilon}(x_{0}(\theta),\theta)\big)}_{>0}+\gamma>0,
    \end{align*}
    so that $(F_{\beta}^{\gamma})_{x}(x_{0}(\theta),\theta)>0$, which contradicts Lemma \ref{Elementary lemma}.
    
    \vspace{0.25cm}
    \textbf{Proof of (2):} As in the previous step, we define $F_{\underline{\widehat{x}}_{\epsilon}(\theta)}^{\gamma}(x,\theta):=\phi_{\underline{\widehat{x}}_{\epsilon}(\theta)}^{\gamma}(x,\theta)-c(x)$. For $\gamma<0$, we plug $x=\underline{\widehat{x}}_{\epsilon}(\theta)$ into the (\ref{Auxiliary Second order ODE}). Then, using the fact that $\phi_{\underline{\widehat{x}}_{\epsilon}}^{\gamma}(\underline{\widehat{x}}_{\epsilon}(\theta),\theta)=c(\widehat{\underline{x}}_{\epsilon}(\theta))$, which yields $F_{\widehat{\underline{x}}_{\epsilon}(\theta)}^{\gamma}(\widehat{\underline{x}}_{\epsilon}(\theta),\theta)=0$, following the same arguments of Step 1 above we obtain 
    \begin{equation*}
        \frac{1}{2}\sigma^{2}(\underline{\widehat{x}}_{\epsilon}(\theta))(F_{\widehat{\underline{x}}_{\epsilon}(\theta)}^{\gamma})_{x}(\underline{\widehat{x}}_{\epsilon}(\theta),\theta)=\gamma<0,
    \end{equation*} 
    which implies that $(F_{\widehat{\underline{x}}_{\epsilon}(\theta)}^{\gamma})_{x}(\underline{\widehat{x}}_{\epsilon}(\theta))<0$. We want to show that for any $x\in (0,\underline{\widehat{x}}_{\epsilon}(\theta))$ we have $\phi_{\widehat{\underline{x}}_{\epsilon}(\theta)}^{\gamma}(x,\theta)>c(x),$ so to conclude by Lemma \ref{Pertrubed solutions lemma} that $\phi_{\widehat{\underline{x}}_{\epsilon}(\theta)}(x,\theta)\geq c(x),\; x\in (0,\underline{\widehat{x}}_{\epsilon}(\theta))$. 
    
    We argue again by contradiction and assume that there exists $x_{1}(\theta):=\max\{x\in (0,\underline{\widehat{x}}_{\epsilon}(\theta)): \phi^{\gamma}_{\widehat{\underline{x}}_{\epsilon}(\theta)}(x,\theta)=c(x)\}=\max\{x\in (0,\underline{\widehat{x}}_{\epsilon}(\theta)): F^{\gamma}_{\widehat{\underline{x}}_{\epsilon}(\theta)}(x,\theta)=0\}$. Hence, feeding $x=x_{1}(\theta)$ into (\ref{Auxiliary Second order ODE}) we have
    \begin{align*}
        \frac{1}{2}\sigma^{2}(x_{1}(\theta))(\phi_{\widehat{\underline{x}}_{\epsilon}(\theta)}^{\gamma})_{x}(x_{1}(\theta))&+b(x_{1}(\theta))\phi_{\widehat{\underline{x}}_{\epsilon}(\theta)}^{\gamma}(x_{1}(\theta))\\
        &-\frac{\epsilon}{2}\sigma^{2}(x_{1}(\theta))(\phi_{\widehat{\underline{x}}_{\epsilon}(\theta)}^{\gamma})^{2}(x_{1}(\theta))+\pi(x_{1}(\theta))=\ell^{\epsilon}(\underline{\widehat{x}}_{\epsilon}(\theta),\theta)+\gamma.
    \end{align*}
    Rearranging the terms we find that
    \begin{align*}
        &\frac{1}{2}\sigma^{2}(x_{1}(\theta))\partial_{x}(\phi_{\widehat{\underline{x}}_{\epsilon}(\theta)}^{\gamma}(\cdot,\theta)-c(\cdot))_{x}(x_{1}(\theta))=\ell^{\epsilon}(\widehat{\underline{x}}_{\epsilon}(\theta),\theta)\\
        &\quad\quad-\underbrace{\big(b(x_{1}(\theta))c(x_{1}(\theta))+\pi(x_{1}(\theta),\theta)-\frac{1}{2}\sigma^{2}(x_{1}(\theta))\big(\epsilon c^{2}(x_{1}(\theta))-c_{x}(x_{1}(\theta))\big)\big)}_{=\ell^{\epsilon}(x_{1}(\theta),\theta)}+\gamma;
    \end{align*}
    that is,
    \begin{align*}
        \frac{1}{2}\sigma^{2}(x_{1}(\theta))(F_{\widehat{\underline{x}}_{\epsilon}(\theta)}^{\gamma})_{x}(x_{1}(\theta))&=\ell^{\epsilon}(\underline{\widehat{x}}_{\epsilon}(\theta),\theta)-\ell^{\epsilon}(x_{1}(\theta),\theta)+\gamma<0,
    \end{align*}
    where the last inequality comes from Assumption \ref{Assumption for l}-(\ref{Condition of lambdas der 1}). Hence, $(F_{\widehat{\underline{x}}_{\epsilon}(\theta)}^{\gamma})_{x}(x_{1}(\theta))<0$, which contradicts Lemma \ref{Elementary lemma}.
    
    \vspace{0.25 cm}
    
    \textbf{Proof of (3):} This follows from the previous steps.
    \end{proof}
We also have the following result on the boundedness of $\phi_{\beta_{\epsilon}(\theta)}$.

Now we are in the position to introduce the candidate optimal potential function $V^{\epsilon}$ as 

\begin{equation}
    \label{candidate solution}
    V^{\epsilon}(x,\theta):=\begin{cases}
                               -\int_{x}^{\beta_{\epsilon}(\theta)}\phi_{\beta_{\epsilon}(\theta)}(y,\theta)dy,\displaystyle\quad &0<x<\beta_{\epsilon}(\theta),  \\
                               \int_{\beta_{\epsilon}(\theta)}^{x}c(y)dy,\quad &x\geq\beta_{\epsilon}(\theta), \displaystyle
                             \end{cases}
\end{equation}
where $\beta_{\epsilon}(\theta)$ given by (\ref{Candidate upper free-boundaries}) is then such that $\beta_{\epsilon}(\theta):=\inf\{x\in\mathbb{R}_{+}:V_{x}^{\epsilon}(x,\theta)=c(x)\}$.
\begin{theorem}[Existence and Uniqueness of solution to (\ref{free-boundary problem})]
    \label{E&U of free boundary problem}
    Let Assumptions \ref{Assumption for dynamics of SDE}, \ref{Ass: Non explosion}, \ref{Assumptions for profit fun} and \ref{Assumption for l} hold. Let $\beta_{\epsilon}(\theta)$ as in (\ref{Candidate upper free-boundaries}) and $V^{\epsilon}$ as in (\ref{candidate solution}). Defining $\lambda^{\epsilon}(\theta):=\ell^{\epsilon}(\beta_{\epsilon}(\theta),\theta))$, the couple $(V^{\epsilon}(\cdot,\theta),\lambda^{\epsilon}(\theta))$, with $V^{\epsilon}(\cdot,\theta)\in C^{2}(\mathbb{R}_{+})$, is the unique solution to (\ref{free-boundary problem}).
\end{theorem}

\begin{proof}
    First of all, we show that $V^{\epsilon}(\cdot,\theta)\in C^{2}(\mathbb{R}_{+})$. By definition of $V^{\epsilon}(\cdot,\theta)$ (cf.\ (\ref{candidate solution})) it is sufficient to show that $V^{\epsilon}_{xx}(\beta_{\epsilon}(\theta),\theta),\theta)=c_{x}(\beta_{\epsilon}(\theta))$. To that end, plugging $x=\beta_{\epsilon}(\theta)$ in (\ref{Auxiliary Second order ODE}) (for $\beta=\beta_{\epsilon}(\theta)$) we obtain
    \begin{align}
        \frac{1}{2}\sigma^{2}(\beta_{\epsilon}(\theta))\partial_{x}\phi_{\beta_{\epsilon}(\theta)}(\beta_{\epsilon}(\theta),\theta)&=\ell^{\epsilon}(\beta_{\epsilon}(\theta),\theta)-\pi(\beta_{\epsilon}(\theta),\theta)-b(\beta_{\epsilon}(\theta))\phi_{\beta_{\epsilon}(\theta)}(\beta_{\epsilon}(\theta),\theta) \notag \\
        &\quad+\frac{\epsilon}{2}\sigma^{2}(\beta_{\epsilon}(\theta))(\phi_{\beta_{\epsilon}(\theta)})^{2}(\beta_{\epsilon}(\theta),\theta) \notag \\
        &=\ell^{\epsilon}(\beta_{\epsilon}(\theta),\theta)-\pi(\beta_{\epsilon}(\theta),\theta)-b(\beta_{\epsilon}(\theta))c(\beta_{\epsilon}(\theta)) \notag \\
        &\quad+\frac{\epsilon}{2}\sigma^{2}(\beta_{\epsilon}(\theta))c^{2}(\beta_{\epsilon}(\theta))=\frac{1}{2}\sigma^{2}(\beta_{\epsilon}(\theta))\partial_{x}c(\beta_{\epsilon}(\theta)), 
    \end{align}
    where the last equation follows from (\ref{function for eigenvalue}). Hence, given the strict positivity of $\sigma$ we conclude the desired equation.

    We now move on by showing that $(V^{\epsilon}(\cdot,\theta),\lambda^{\epsilon}(\theta))$ solve (\ref{free-boundary problem}). For $x\geq\beta_{\epsilon}(\theta)$, $V^{\epsilon}(\cdot,\theta)$ as in (\ref{candidate solution}) satisfies
    \begin{align*}
        V_{x}^{\epsilon}(x,\theta)&=c(x) \\
        \mathcal{L}^{\epsilon}V^{\epsilon}(x,\theta)+\pi(x,\theta)&=\frac{1}{2}\sigma^{2}(x)c_{x}(x)+b(x)c(x)-\frac{\epsilon}{2}\sigma^{2}(x)c^{2}(x)+\pi(x,\theta)\\
        &=\ell^{\epsilon}(x,\theta)\leq \ell^{\epsilon}(\beta_{\epsilon}(\theta),\theta)
    \end{align*}
    where the last inequality comes from the Assumption \ref{Condition of lambdas der 1} as $\beta_{\epsilon}(\theta)\geq\widehat{x}_{\epsilon}(\theta)$ (cf. Proposition \ref{properties of b(a)}-(\ref{infimum of b})). 
    
    On the other hand, for $x\in (0,\beta_{\epsilon}(\theta))$ it is clear that, $V_{x}^{\epsilon}(x,\theta)=\phi_{\beta_{\epsilon}(\theta)}(x,\theta)$. Thanks to Proposition \ref{Candidate upper free-boundaries} and Lemma \ref{Pointwise conv result} we have that $\beta_{\epsilon}(\theta)\in B_{\epsilon}(\theta)$, hence $V_{x}^{\epsilon}(x,\theta)\geq c(x)$. To conclude, the equation
    \begin{equation*}
        \mathcal{L}^{\epsilon}V^{\epsilon}(x,\theta)+\pi(x,\theta)=\ell^{\epsilon}(\beta_{\epsilon}(\theta),\theta)=\lambda^{\epsilon}(\theta),\quad x\in (0,\beta_{\epsilon}(\theta)),
    \end{equation*}
    is satisfied since $\phi_{\beta_{\epsilon}(\theta)}(\cdot,\theta)$ solves (\ref{Auxiliary Second order ODE}).
\end{proof}

\subsection{Verification Result}

The analysis of Section \ref{section: Existence of a solution to First ODE} reveals that there exists a free boundary $\beta_{\epsilon}(\theta)$ splitting the state-space into a region where $\phi_{\beta_{\epsilon}(\theta)}(x,\theta)=c(x)$ (the so-called action region) and a region where $\phi_{\beta_{\epsilon}(\theta)}(x,\theta)>c(x)$ (the so-called inaction region). The next definition introduces the notion of the Skorokhod reflection problem, consisting of finding the couple $(X^{\xi^{*}(\theta)},\xi^{*}(\theta))$ such that $X_{t}^{\xi^{*}(\theta)}\leq \beta_{\epsilon}(\theta),\; \forall t\geq 0\; \mathbb{P}$-a.s.\ The control $\xi^{*}(\theta)$ solving the Skorokhod reflection problem will be shown to be optimal in Theorem \ref{theorem: Verification result}.

\begin{definition}[Skorokhod Reflection Problem]
    \label{def Skorokhod}
    Let $\mathcal{D}[0,\infty)$ be the space of càdlàg processes on $[0,\infty)$. Given $x\in\mathbb{R}_{+}$, $\mathbb{Q}\in\widehat{\mathcal{Q}}(x)$, $\beta\in\mathbb{R}_{+}$, and $\epsilon> 0$, the process $(X^{\xi},\xi)\in \mathcal{D}[0,\infty)\times \mathcal{A}_{e}(x)$ is said to be the solution of the Skorokhod reflection problem $\textbf{SP}(x,\beta;\mathbb{Q},\psi)$ for the $\mathbb{Q}$-Brownian motion $W^{\mathbb{Q}}$ if it satisfies the following properties:
    \begin{enumerate}
        \item \label{item: integral form of reflected diffusion}$X^{\xi}_{t}=x+\int_{0}^{t}(b(X^{\xi}_{s})+\sigma(X^{\xi}_{s})\psi(X^{\xi}_{s}))ds+\int_{0}^{t}\sigma(X^{\xi}_{s})dW^{\mathbb{Q}}_{s}-\xi_{t},\quad \mathbb{Q}\otimes dt$-a.s. \\
        \item \label{item: reflecetion domain}$X^{\xi}_{t}\in (0,\beta],\quad \mathbb{Q}\otimes dt$-a.s. \\
        \item \label{item: positive control} $\int_{0}^{T}\boldsymbol{1}_{\{X^{\xi}_{s}<\beta\}}d\xi_{s}=0, \quad\mathbb{Q}\otimes dt$-a.s.
    \end{enumerate}
\end{definition}
\begin{proposition}
    \label{Prop: Optimal policy}
    For any $\theta\in\mathbb{R}_{+}$, there exists $\xi^{*}(\theta)\in\mathcal{A}_{e}(x)$ such that $(X^{\xi^{*}(\theta)},\xi^{*}(\theta))$ is the unique solution to $\textbf{SP}(x,\beta_{\epsilon}(\theta);\mathbb{Q},\psi)$ with
    \begin{equation}
        \label{eq: Optimal Policy}
        \xi^{*}_{t}(\theta)=\sup_{s\leq t}\big(I(X^{\xi^{*}}_{s}(\theta))-\beta_{\epsilon}(\theta) \big)^{+},
    \end{equation}
    where $I(X)_{t}:=x+\int_{0}^{t}\big(b(X_{s})+\sigma(X_{s})\psi(X_{s})\big)ds+\int_{0}^{t}\sigma(X_{s})dW_{s}^{\mathbb{Q}}$.
\end{proposition}
\begin{proof}
    Since the properties of Theorem 4.1 in \cite{Tanaka} are satisfied (recall that $\psi$ is bounded and locally-Lipschitz continuous as $\mathbb{Q}\in \widehat{\mathcal{Q}}(x)$; cf.\ Definition \ref{Admissible strategies}), we obtain that $\textbf{SP}(x,\beta;\mathbb{Q},\psi)$ has a unique solution $(X^{\overline{\xi}},\overline{\xi})$ and $X^{\overline{\xi}}$ is pathwise unique. Thanks to (\ref{eq: Optimal Policy}), it is standard to see that $(X^{\xi^{*}(\theta)},\xi^{*}(\theta))$ satisfies the properties of $\textbf{SP}(x,\beta;\mathbb{Q},\psi)$, and by uniqueness we conclude that $\overline{\xi}=\xi^{*}(\theta)$. Furthermore, under Assumption \ref{Ass: Non explosion}, the state space of the optimally controlled process $X^{\xi^{*}(\theta)}$ is $(0,\beta_{\epsilon}(\theta)]$, with $0$ being not attainable and $\beta_{\epsilon}(\theta)$ being reflecting. It this then a standard result in the theory of one-dimensional diffusion that the process $X^{\xi^{*}(\theta)}$ cannot reach $0$ in finite time with positive probability. Finally, since $X^{\xi^{*}(\theta)}_{t}\in (0,\beta_{\epsilon}(\theta)],\; \mathbb{Q}\otimes dt$-a.s., the integrability conditions in Definition \ref{def: set of admissible controls} are satisfied. Thus, $\xi^{*}(\theta)\in\mathcal{A}_{e}(x)$.  
\end{proof}

As in \cite{CohenHening} and \cite{LiangLiuZervos} (see Remark 5 in \cite{CohenHening} and Section 5 in \cite{LiangLiuZervos}), we also have that the potential function $V^{\epsilon}$ is unbounded as $x\downarrow 0$, namely $V^{\epsilon}(x,\theta)\downarrow -\infty$ as $x\downarrow 0$ (see Proposition \ref{prop: uniform boundness of phi}).\ Such unboundedness of $V^{\epsilon}$ requires extra care in the proof of verification theorem \ref{theorem: Verification result} below. For this reason a truncated version of the potential function will be introduced in Step 2 of the proof of Theorem \ref{theorem: Verification result}, extending to our setting with non constant instantaneous reward from exerting control the approach developed in \cite{LiangLiuZervos}.

\begin{theorem}[Verification Theorem]
    \label{theorem: Verification result}
    Let Assumptions \ref{Assumption for dynamics of SDE}, \ref{Ass: Non explosion}, \ref{Assumptions for profit fun} and \ref{Assumption for l} hold. For every $x\in\mathbb{R}_{+}$, let $\xi^{*}(\theta)$ solve $\textbf{SP}(x,\beta_{\epsilon}(\theta);\mathbb{Q}^{*},\psi^{*})$, where $\mathbb{Q}^{*}(\theta)\in\widehat{Q}(x)$ is such that $\frac{d\mathbb{Q}^{*}}{d\mathbb{P}}\big|_{\mathcal{F}_{t}}:=\psi_{t}^{*}$, with $\psi_{t}^{*}:=-\epsilon\sigma(X^{\xi^{*}(\theta)}_{t})V_{x}^{\epsilon}(X^{\xi^{*}(\theta)}_{t},\theta),\; \mathbb{Q}^{*}\otimes dt-$a.s. Then $(\xi^{*}(\theta),\mathbb{Q}^{*}(\theta))\in\mathcal{A}_{e}(x)\times\widehat{\mathcal{Q}}(x)$ realizes a saddle point in (\ref{value function}) and
    \begin{align}
        \label{optimal eigenvalue}
        \lambda^{\epsilon}(\theta)=b(\beta_{\epsilon}(\theta))c(\beta_{\epsilon}(\theta))+\pi(\beta_{\epsilon}(\theta),\theta)-\frac{1}{2}\sigma^{2}(\beta_{\epsilon}(\theta))\big(\epsilon c^{2}(\beta_{\epsilon}(\theta))-c_{x}(\beta_{\epsilon}(\theta))\big).
    \end{align}
\end{theorem}
\begin{proof}
    We split the proof into two steps.

    \vspace{0.25 cm}
    
    \textbf{Step 1:} Let $T>0$, $x\in\mathbb{R}_{+}$ and $\epsilon\geq 0$. Recall Theorem \ref{E&U of free boundary problem} and introduce a sequence of $\mathbb{F}-$stopping times $(\tau_{n}^{*}(\theta))_{n\in\mathbb{N}}$ such that $\tau_{n}^{*}(\theta):=\inf\{t\geq 0: X^{\xi^{*}(\theta)}_{t}\notin [1/n,n]\}$. Then, fixing $\mathbb{Q}\in\widehat{\mathcal{Q}}(x)$ with Girsanov kernel $\psi$ as in Definition \ref{Admissible strategies} and applying Itô-Meyer's formula to $(V^{\epsilon}(X^{\xi^{*}(\theta)}_{T\wedge \tau_{n}^{*}(\theta)},\theta))_{T\geq 0}$, we have under $\mathbb{Q}$ (cf. also (\ref{propotional costs integral})) (recalling that $\xi^{*}(\theta)$ may cause a jump only at $t=0$)
    \begin{align}
        \label{Ito formula}
        V^{\epsilon}(X^{\xi^{*}(\theta)}_{T\wedge \tau_{n}^{*}(\theta)},\theta)=&V^{\epsilon}(x,\theta)+\int_{0}^{T\wedge \tau_{n}^{*}(\theta)}\bigg(\frac{1}{2}\sigma^{2}(X^{\xi^{*}(\theta)}_{s})V^{\epsilon}_{xx}(X^{\xi^{*}(\theta)}_{s},\theta)+\big(b(X^{\xi^{*}(\theta)}_{s})\notag\\
        &+\sigma(X^{\xi^{*}(\theta)}_{s})\psi_{s}\big)V^{\epsilon}_{x}(X^{\xi^{*}(\theta)}_{s},\theta)\bigg)ds+\int_{0}^{T\wedge \tau_{n}^{*}(\theta)}\sigma(X^{\xi^{*}(\theta)}_{s})V^{\epsilon}_{x}(X^{\xi^{*}(\theta)}_{s},\theta)dW^{\mathbb{Q}}_{s} \notag \\
        &-\int_{0}^{T\wedge \tau_{n}^{*}(\theta)}V^{\epsilon}_{x}(X^{\xi^{*}(\theta)}_{s},\theta)d(\xi^{*})_{s}^{c}(\theta)-\big(V^{\epsilon}(X^{\xi^{*}(\theta)}_{0+})-V^{\epsilon}(X^{\xi^{*}(\theta)}_{0})\big).
    \end{align}
    Noting that by continuity of $\sigma(\cdot) V_{x}^{\epsilon}(\cdot,\theta)$ we have $\mathbb{E}^{\mathbb{Q}}_{x}\bigg[ \int_{0}^{T\wedge \tau_{n}^{*}(\theta)}\sigma(X^{\xi^{*}(\theta)}_{s})V^{\epsilon}_{x}(X^{\xi^{*}(\theta)}_{s},\theta)dW_{s} \bigg]=0,\; \text{for any }T>0,\;n\in\mathbb{N}$, by taking expectations in (\ref{Ito formula}) with respect to $\mathbb{Q}$ we obtain 
    \begin{align}
        \label{expcted Ito formula}
        \mathbb{E}^{\mathbb{Q}}_{x}\big[ V^{\epsilon}(X^{\xi^{*}(\theta)}_{T\wedge \tau_{n}^{*}(\theta)},\theta)\big]&\geq V^{\epsilon}(x,\theta)+\mathbb{E}^{\mathbb{Q}}_{x}\bigg[\int_{0}^{T\wedge \tau_{n}^{*}(\theta)}\big( \lambda^{\epsilon}(\theta)-\pi(X^{\xi^{*}(\theta)}_{s},\theta)-\frac{1}{2\epsilon}(\psi_{s})^{2}\big)ds\bigg] \notag \\ &\quad-\mathbb{E}^{\mathbb{Q}}_{x}\bigg[\int_{0}^{T\wedge \tau_{n}^{*}(\theta)}V^{\epsilon}_{x}(X^{\xi^{*}(\theta)}_{s},\theta)d(\xi^{*})_{s}^{c}(\theta)\bigg] \notag\\
        &\quad-\boldsymbol{1}_{\{x>\beta_{\epsilon}(\theta)\}}\int_{0}^{x-\beta_{\epsilon}(\theta)}V^{\epsilon}_{x}(x-r,\theta)dr, 
    \end{align}
     where we have also used the fact that, for any admissible $\psi$, $V^{\epsilon}(\cdot,\theta)$ satisfies (cf. (\ref{free-boundary problem}))
     \begin{align*}
         \frac{1}{2}\sigma^{2}(X^{\xi^{*}(\theta)}_{s})V^{\epsilon}_{xx}(X^{\xi^{*}(\theta)}_{s},\theta)&+\big(b(X^{\xi^{*}(\theta)}_{s})+\sigma(X^{\xi^{*}(\theta)}_{s})\psi_{s}\big)V^{\epsilon}_{x}(X^{\xi^{*}(\theta)}_{s},\theta)\\ &\geq \lambda^{\epsilon}(\theta)-\pi(X^{\xi^{*}(\theta)}_{s},\theta)-\frac{1}{2\epsilon}(\psi_{s})^{2},\quad \mathbb{Q}\otimes dt-\text{a.s.},
     \end{align*}
     and  
     \begin{equation}
        \label{eq: value function mean value theorem}
         \big(V^{\epsilon}(X^{\xi^{*}(\theta)}_{0+})-V^{\epsilon}(X^{\xi^{*}(\theta)}_{0})\big)=\boldsymbol{1}_{\{x>\beta_{\epsilon}(\theta)\}}\int_{0}^{x-\beta_{\epsilon}(\theta)}V^{\epsilon}_{x}(x-r,\theta)dr.
     \end{equation}
     Since, $\int_{0}^{\infty}\boldsymbol{1}_{(0,\beta_{\epsilon}(\theta))}(X^{\xi^{*}(\theta)}_{s})d(\xi^{*})_{s}(\theta)=0$, and $V^{\epsilon}_{x}(X^{\xi^{*}(\theta)}_{t},\theta)=c(X^{\xi^{*}(\theta)}_{t})$ when $X^{\xi^{*}(\theta)}_{t}\geq\beta_{\epsilon}(\theta)$ $\;\mathbb{Q}\otimes dt-$a.s., from (\ref{expcted Ito formula}) we have
    \begin{align}
        \mathbb{E}^{\mathbb{Q}}_{x}\big[ V^{\epsilon}(X^{\xi^{*}(\theta)}_{T\wedge \tau_{n}^{*}(\theta)},\theta)\big]&\geq V^{\epsilon}(x,\theta)+\mathbb{E}^{\mathbb{Q}}_{x}\bigg[\int_{0}^{T\wedge \tau_{n}^{*}(\theta)}\big( \lambda^{\epsilon}(\theta)-\pi(X^{\xi^{*}(\theta)}_{s},\theta)-\frac{1}{2\epsilon}(\psi_{s}(X_{s}^{\xi^{*}(\theta)}))^{2}\big)ds\bigg] \notag \\ &\quad-\mathbb{E}^{\mathbb{Q}}_{x}\bigg[\int_{0}^{T\wedge \tau_{n}^{*}(\theta)}c(X^{\xi^{*}(\theta)}_{s})d(\xi^{*})^{c}_{s}(\theta)\bigg]-\boldsymbol{1}_{\{x>\beta_{\epsilon}(\theta)\}}\int_{0}^{x-\beta_{\epsilon}(\theta)}c(x-r)dr.
    \end{align}
    Rearranging the terms and recalling (\ref{propotional costs integral}) we find
    \begin{align}
        \label{rearrnged terms ito formula}
        \mathbb{E}^{\mathbb{Q}}_{x}\big[ V^{\epsilon}(X^{\xi^{*}(\theta)}_{T\wedge \tau_{n}^{*}(\theta)},\theta)\big]&+\mathbb{E}^{\mathbb{Q}}_{x}\bigg[\int_{0}^{T\wedge \tau_{n}^{*}(\theta)}\big(\pi(X^{\xi^{*}(\theta)}_{s},\theta)+\frac{1}{2\epsilon}(\psi_{s}(X_{s}^{\xi^{*}(\theta)}))^{2}\big)\bigg]\geq V^{\epsilon}(x,\theta)\notag \\ &+\lambda^{\epsilon}(\theta)\mathbb{E}^{\mathbb{Q}}_{x}\big[T\wedge \tau_{n}^{*}(\theta)\big]-\mathbb{E}^{\mathbb{Q}}_{x}\bigg[\int_{0}^{T\wedge \tau_{n}^{*}(\theta)}c(X^{\xi^{*}(\theta)}_{s})\circ d(\xi^{*})_{s}(\theta)\bigg].
    \end{align}
    Given that $X^{\xi^{*}(\theta)}_{t}\in (0,\beta_{\epsilon}(\theta)]
    ;\mathbb{Q}\otimes dt-$a.s., passing to the limits as $n\uparrow\infty$ in (\ref{rearrnged terms ito formula}), invoking monotone and dominated convergence theorems, we can exchange limits with expectations and obtain
    \begin{align}
        \label{rearrnged terms ito formula without n}
        \mathbb{E}^{\mathbb{Q}}_{x}\big[ V^{\epsilon}(X^{\xi^{*}(\theta)}_{T},\theta)\big]&+\mathbb{E}^{\mathbb{Q}}_{x}\bigg[\int_{0}^{T}\big(\pi(X^{\xi^{*}(\theta)}_{s},\theta)+\frac{1}{2\epsilon}(\psi_{s}(X^{\xi^{*}(\theta)}_{s})^{2}\big)\bigg]\geq V^{\epsilon}(x,\theta)\notag \\ &+\lambda^{\epsilon}(\theta)T-\mathbb{E}^{\mathbb{Q}}_{x}\bigg[\int_{0}^{T}c(X^{\xi^{*}(\theta)}_{s})\circ d(\xi^{*})_{s}(\theta)\bigg].
    \end{align}
    Rearranging terms, dividing by $T$ and sending $T\uparrow\infty$ in (\ref{rearrnged terms ito formula without n}) we obtain that
    \begin{align}
        \liminf_{T\uparrow\infty}\frac{1}{T}\mathbb{E}^{\mathbb{Q}}_{x}&\bigg[\int_{0}^{T}\big(\pi(X^{\xi^{*}(\theta)}_{s},\theta)+\frac{1}{2\epsilon}(\psi_{s}(X_{s}^{\xi^{*}(\theta)}))^{2}\big)+\int_{0}^{T}c(X^{\xi^{*}(\theta)}_{s})\circ d(\xi^{*})_{s}(\theta)\bigg] \notag \\
        &\geq \lambda^{\epsilon}(\theta)-\limsup_{T\uparrow\infty}\frac{1}{T}\mathbb{E}^{\mathbb{Q}}_{x}\big[ V^{\epsilon}(X^{\xi^{*}(\theta)}_{T},\theta)\big].
    \end{align}
    Hence, given again that $X_{t}^{\xi^{*}(\theta)}\in (0,\beta_{\epsilon}(\theta)],\;\mathbb{Q}\otimes dt-$a.s. and $V^{\epsilon}(\cdot,\theta)$ is continuous on $(0,\beta_{\epsilon}(\theta)]$, we find
    \begin{align}
        \liminf_{T\uparrow\infty}\frac{1}{T}\mathbb{E}^{\mathbb{Q}}_{x}&\bigg[\int_{0}^{T}\big(\pi(X^{\xi^{*}(\theta)}_{s},\theta)+\frac{1}{2\epsilon}(\psi_{s}(X_{s}^{\xi^{*}(\theta)}))^{2}\big)+\int_{0}^{T}c(X^{\xi^{*}(\theta)}_{s})\circ d(\xi^{*})_{s}(\theta)\bigg]\geq \lambda^{\epsilon}(\theta);
    \end{align}
    that is,
    \begin{equation}
        \label{eq: last ineq step 1}
        J^{\epsilon}(x;\xi^{*}(\theta),\mathbb{Q},\theta)\geq \lambda^{\epsilon}(\theta),\quad\text{for any }\mathbb{Q}\in\widehat{\mathcal{Q}}(x).
    \end{equation}
    Hence,
    \begin{equation}
        \sup_{\xi\in\mathcal{A}_{e}(x)}\inf_{\mathbb{Q}\in\widehat{\mathcal{Q}}(x)}J^{\epsilon}(x;\xi,\mathbb{Q},\theta)\geq\lambda^{\epsilon}(\theta)
    \end{equation}
    \textbf{Step 2:}
    Let $x>0$, $\xi\in\mathcal{A}_{e}(x)$, and $\beta\in (\widehat{x}_{\epsilon}(\theta),\beta_{\epsilon}(\theta))$. We introduce the sequence of $\mathbb{F}$-stopping times $(\tau_{n})_{n\in\mathbb{N}}$ defined by $\tau_{n}:=\inf\{t\geq 0: X^{\xi}_{t}\notin [1/n,n]\}$. If $\phi_{\beta}$ is a solution to (\ref{Auxiliary Second order ODE}), we define the truncated version of the potential gradient $\overline{\phi}_{\beta}$ as
    \begin{equation}
        \label{eq: truncated phi}
        \overline{\phi}_{\beta}(x,\theta):=\begin{cases}
            \phi_{\beta}(x,\theta),\quad &x>\alpha_{\beta}, \\
            c(x),\quad &x\leq \alpha_{\beta},
        \end{cases}
    \end{equation}
    where $\alpha_{\beta}:=\inf\{x>0:\phi_{\beta}(x,\theta)=c(x)\}$. Provided that $\alpha_{\beta}$ exists and is finite, the structure of the truncated gradient potential function ensures that $\overline{\phi}_{\beta}(\cdot,\theta)\in C(\mathbb{R}_{+})$. Since $\beta<\beta_{\epsilon}(\theta)$, Lemma \ref{lemma: truncated phi} implies that $\alpha_{\beta}<\widehat{x}_{\epsilon}(\theta)$ and 
    $$\alpha_{\beta}\to 0, \quad \text{as}\quad \beta\to\beta_{\epsilon}(\theta)^{-}.$$ 
    We then define $\overline{V}^{\beta}(x,\theta):=\overline{V}^{\beta}(\alpha_{\beta},\theta)+\int_{\alpha_{\beta}}^{x}\overline{\phi}_{\beta}(y,\theta)dy$ for $x\in (0,\infty)$, which is $C^{1}(\mathbb{R}_{+})\cap C^{2}(\mathbb{R}_{+}\setminus\{\alpha_{\beta}\})$.
    Recalling the structure of the operator $\mathcal{L}^{\epsilon}$ in (\ref{generator}), we define a measure $\mathbb{Q}^{\beta}(\theta)\in \widehat{\mathcal{Q}}(x)$ with the Girsanov kernel $\psi^{\beta}$ given by
    \begin{align*}
        \psi^{\beta}_{t}:&=\arg \min_{\psi_{t}}\bigg\{ \frac{1}{2}\sigma^{2}(X^{x,\xi}_{t})\overline{V}^{\beta}_{xx}(X^{x,\xi}_{t},\theta)+\big(b(X^{x,\xi}_{t})+\sigma(X^{x,\xi}_{t})\psi_{t}\big)\overline{V}^{\beta}_{x}(X^{x,\xi}_{t},\theta)+\frac{1}{2\epsilon}\psi^{2}_{t} \bigg\} \\
        &=\arg \min_{\psi_{t}}\bigg\{ \sigma(X^{x,\xi}_{t})\psi_{t}\overline{V}^{\beta}_{x}(X^{x,\xi}_{t},\theta)+\frac{1}{2\epsilon}\psi^{2}_{t} \bigg\} \\
        &=-\epsilon\sigma(X^{x,\xi}_{t})\overline{V}^{\beta}_{x}(X^{x,\xi}_{t},\theta),\quad \mathbb{P}\otimes dt\text{-a.s.}
    \end{align*}
    We claim that $\mathbb{Q}^{\beta}(\theta)$ is admissible according to Definition \ref{Admissible strategies}. 
    First, given that $\overline{V}^{\beta}_{x}(x,\theta)=\overline{\phi}_{\beta}(x,\theta),x\in (0,\infty)$, we obtain from Proposition \ref{section: Existence of a solution to First ODE} and Assumption \ref{Assumption for dynamics of SDE} that the mapping $x\mapsto -\epsilon\sigma^{2}(x)\overline{V}^{\beta}_{x}(x,\theta)$ is locally Lipschitz. Consequently, $S_{x}^{\mathbb{Q}^{\beta}(\theta)}$ as in (\ref{eq: speed measure wrt Q}) satisfies Assumption \ref{Ass: Non explosion}. Furthermore, from Assumption \ref{Assumption for dynamics of SDE}-(\ref{Linear growth of drift and vol}) and the structure of $\overline{\phi}_{\beta}$, there exists a constant $K(\beta,\theta)>0$ such that $|\psi^{\beta}_{t}|^{2}\leq K(\beta,\theta)(1+|X^{\xi}_{t}|^{2\zeta}),\; \mathbb{P}\otimes dt$-a.s. Hence, (\ref{eq: Novikov condition}) holds because $\xi\in\mathcal{A}_{e}(x)$ and due to the sublinear growth of $\sigma$. Thus, $\mathbb{Q}^{\beta}(\theta)$ is admissible.

    
    Recalling that $\overline{V}^{\beta}(\cdot,\theta)\in C^{1}(\mathbb{R}_{+})\cap C^{2}(\mathbb{R}_{+}\setminus \{\alpha_{\beta}\})$. Applying Itô-Meyer's formula as in \cite{Meyer1976} to $(\overline{V}^{\beta}(X^{\xi}_{T\wedge \tau_{n}},\theta))_{T\geq 0}$, arguing as in the proof of Theorem 2 in \cite{DavisZervos1998pair}, and taking expectations under $\mathbb{Q}^{\beta}(\theta)$, we obtain (recalling that for any $\xi\in\mathcal{A}_{e}(x)$, it holds that $\xi_{t}=\xi^{c}_{t}+\sum_{s\leq t}\Delta\xi_{s}$, where $\xi^{c}$ is the continuous part of $\xi$):
    \begin{align}
        \label{expcted Ito formula arbitrary xi}
        \mathbb{E}^{\mathbb{Q}^{\beta}(\theta)}_{x}\big[ \overline{V}^{\beta}(X^{\xi}_{T\wedge \tau_{n}},\theta)\big]&=\overline{V}^{\beta}(x,\theta)+\mathbb{E}^{\mathbb{Q}^{\beta}(\theta)}_{x}\bigg[\int_{0}^{T\wedge \tau_{n}}\bigg(\frac{1}{2}\sigma^{2}(X^{\xi}_{s})\overline{V}^{\beta}_{xx}(X^{\xi}_{s},\theta) \notag \\
        &\quad +\big(b(X^{\xi}_{s})+\sigma(X^{\xi}_{s})\psi_{s}^{\beta}\big)\overline{V}^{\beta}_{x}(X^{\xi}_{s},\theta)\bigg)ds \bigg] -\mathbb{E}^{\mathbb{Q}^{\beta}(\theta)}_{x}\bigg[\int_{0}^{T\wedge \tau_{n}}\overline{V}^{\beta}_{x}(X^{\xi}_{s},\theta)d\xi_{s}^{c}\bigg]\notag \\
        &\quad-\mathbb{E}^{\mathbb{Q}^{\beta}(\theta)}_{x}\bigg[ \sum_{s\leq T\wedge \tau_{n}}\big( \overline{V}^{\beta}(X^{\xi}_{s+},\theta)-\overline{V}^{\beta}(X^{\xi}_{s},\theta) \big) \bigg]. 
        \end{align}
    Since
    \begin{equation}
        \label{eq: jumps for arb xi}
        \sum_{s\leq T\wedge \tau_{n}}\big( \overline{V}^{\beta}(X^{\xi}_{s+},\theta)-\overline{V}^{\beta}(X^{\xi}_{s},\theta)\big)=\boldsymbol{1}_{\{\Delta\xi_{s}>0\}}\int_{0}^{\Delta\xi_{s}}\overline{V}^{\beta}_{x}(X^{\xi}_{s}-r,\theta)dr,\quad\mathbb{Q}^{\beta}(\theta)\text{-a.s.},
    \end{equation}
    using (\ref{eq: jumps for arb xi}) and the observation that $\frac{1}{2\epsilon}(\psi^{\beta}_{t}(X_{t}^{x,\xi}))^{2}=\frac{\epsilon}{2}\sigma^{2}(X^{x,\xi}_{t})(\overline{V}^{\beta}_{x})^{2}(X^{x,\xi}_{t},\theta),\;\mathbb{Q}^{\beta}(\theta)\otimes dt$-a.s., (\ref{expcted Ito formula arbitrary xi}) gives
    \begin{align}
        \label{expcted Ito formula arbitrary xi II}
        \mathbb{E}^{\mathbb{Q}^{\beta}(\theta)}_{x}\big[ \overline{V}^{\beta}(X^{\xi}_{T\wedge \tau_{n}},\theta)\big]&=\overline{V}^{\beta}(x,\theta)+\mathbb{E}^{\mathbb{Q}^{\beta}(\theta)}_{x}\bigg[\int_{0}^{T\wedge \tau_{n}}\bigg(\mathcal{L}^{\epsilon}\overline{V}^{\beta}(X^{\xi}_{t},\theta)-\frac{\epsilon}{2}\sigma^{2}(X^{\xi}_{t})(\overline{V}^{\beta}_{x})^{2}(X^{\xi}_{t},\theta)\bigg)dt\bigg] \notag\\
        &\quad-\mathbb{E}^{\mathbb{Q}^{\beta}(\theta)}_{x}\bigg[\int_{0}^{T\wedge \tau_{n}}\overline{V}^{\beta}_{x}(X^{\xi}_{s},\theta)\circ d\xi_{s}\bigg].
    \end{align}
    We now aim to establish an upper bound for $\mathcal{L}^{\epsilon}\overline{V}^{\beta}$. Using the definition of $\overline{V}^{\beta}$, that $\overline{\phi}_{\beta}(x,\theta)=\phi_{\beta}(x,\theta),\; x\in (\alpha_{\beta},\infty)$, and the fact that solves $\phi_{\beta}$ solves (\ref{Auxiliary Second order ODE}) with $\gamma=0$ on $(\alpha_{\beta},\infty)$, we find that
    \begin{align}
        \label{eq: applied first order operator to truncated function mid domain}
        \mathcal{L}^{\epsilon}\overline{V}^{\beta}(x,\theta)&=\frac{1}{2}\sigma^{2}(X_{t}^{\xi})(\phi_{\beta})_{x}(x,\theta)+b(X_{t}^{\xi})\phi_{\beta}(x,\theta)\notag \\
        &\quad\quad-\frac{\epsilon}{2}\sigma^{2}(X_{t}^{\xi})(\phi_{\beta})^{2}(x,\theta)=\ell^{\epsilon}(\beta,\theta)-\pi(x,\theta),\quad x\in (\alpha_{\beta},\beta).
    \end{align}
    Moreover, we know that $\overline{\phi}_{\beta}(x,\theta)=c(x),\;x\in (0,\alpha_{\beta}]\cup [\beta,\infty)$ and recalling definition of $\ell^{\epsilon}(x,\theta)$ (cf. \eqref{function for eigenvalue}) we get that
    \begin{align}
        \label{eq: applied first order operator to truncated function outside domains}
        \mathcal{L}^{\epsilon}\overline{V}^{\beta}(x,\theta)&=\frac{1}{2}\sigma^{2}(x)(\phi_{\beta})_{x}(x,\theta)+b(x)\phi_{\beta}(x,\theta)\notag \\
        &\quad\quad-\frac{\epsilon}{2}\sigma^{2}(x)(\phi_{\beta})^{2}(x,\theta)=\ell^{\epsilon}(x,\theta)-\pi(x,\theta),\quad x\in (0,\alpha_{\beta})\cup [\beta,\infty).
    \end{align}
    Therefore, using Assumption \ref{Assumption for l}-(\ref{Condition of lambdas der 1}) and the fact that $\alpha_{\beta}<\widehat{x}_{\epsilon}(\theta)$ (cf. Lemma \ref{lemma: well posedness of truncated problem and limiting behavior}), one can show that
    \begin{equation}
        \label{eq: bounded on differential operator evaluated on V in truncated domains}
        \mathcal{L}^{\epsilon}\overline{V}^{\beta}(x,\theta)\leq\begin{cases}
            \ell^{\epsilon}(\widehat{x}_{\epsilon}(\theta),\theta)-\pi(x,\theta),\quad &x<\alpha_{\beta}, \\
            \ell^{\epsilon}(\beta,\theta)-\pi(x,\theta),\quad &x\geq \beta.
        \end{cases}
    \end{equation}
     Substituting \eqref{eq: applied first order operator to truncated function mid domain} and \eqref{eq: bounded on differential operator evaluated on V in truncated domains} into \eqref{expcted Ito formula arbitrary xi II}, rearranging terms, and using the lower bound $\overline{V}^{\beta}_{x}(x,\theta)\geq c(x)$, we obtain
     \begin{align}
        \label{expcted inequality arbitrary xi I}
        \mathbb{E}^{\mathbb{Q}^{\beta}(\theta)}_{x}&\bigg[\int_{0}^{T\wedge \tau_{n}}\bigg(\pi(X^{\xi}_{t},\theta)+\frac{1}{2\epsilon}(\psi^{\beta}(X_{t}^{\xi}))^{2}\bigg)dt\bigg] \notag \\
        &\leq \overline{V}^{\beta}(x,\theta)-\mathbb{E}^{\mathbb{Q}^{\beta}(\theta)}_{x}\big[ \overline{V}^{\beta}(X^{\xi}_{T\wedge \tau_{n}},\theta)\big]-\mathbb{E}^{\mathbb{Q}^{\beta}(\theta)}_{x}\bigg[\int_{0}^{T\wedge \tau_{n}}c(X^{\xi}_{s})\circ d\xi_{s}\bigg] \notag \\
        &\quad +\ell^{\epsilon}(\widehat{x}_{\epsilon}(\theta),\theta)\mathbb{E}^{\mathbb{Q}^{\beta}(\theta)}_{x}\bigg[\int_{0}^{T\wedge\tau_{n}}\boldsymbol{1}_{\{(0, \alpha_{\beta})\}}(X_{s}^{\xi})ds\bigg] \notag \\
        &\quad +\ell^{\epsilon}(\beta,\theta)\mathbb{E}^{\mathbb{Q}^{\beta}(\theta)}_{x}\bigg[\int_{0}^{T\wedge\tau_{n}}\boldsymbol{1}_{(\alpha_{\beta},\infty)}(X_{s}^{\xi})ds\bigg] \notag \\
        &\quad -\mathbb{E}^{\mathbb{Q}^{\beta}(\theta)}_{x}\big[ \overline{V}^{\beta}(X^{\xi}_{T\wedge\tau_{n}},\theta)\big] -\mathbb{E}^{\mathbb{Q}^{\beta}(\theta)}_{x}\bigg[\int_{0}^{T\wedge \tau_{n}}\big(\overline{V}^{\beta}_{x}(X^{\xi}_{s},\theta)-c(X^{\xi}_{t})\big)\circ d\xi_{s}\bigg].
    \end{align}
    Using again the fact that $\overline{V}^{\beta}_{x}(X^{\xi}_{t},\theta)\geq c(X^{\xi}_{t})\geq 0,\;\mathbb{Q}^{\beta}(\theta)\otimes dt$-a.s. and the structure of $\overline{V}^{\beta}$,
   \begin{align}
        \label{eq: bound before passing ergodic limits}
        \mathbb{E}^{\mathbb{Q}^{\beta}(\theta)}_{x}\bigg[\int_{0}^{T\wedge \tau_{n}}&\bigg(\pi(X^{\xi}_{t},\theta)+\frac{1}{2\epsilon}(\psi^{\beta}(X_{t}^{\xi}))^{2}\bigg)dt+\int_{0}^{T\wedge \tau_{n}}c(X^{\xi}_{s},\theta)\circ d\xi_{s}\bigg]\notag \\
        &\leq \overline{V}^{\beta}(x,\theta)-\mathbb{E}^{\mathbb{Q}^{\beta}(\theta)}_{x}\big[ \overline{V}^{\beta}(X^{\xi}_{T\wedge \tau_{n}},\theta)\big]+\ell^{\epsilon}(\widehat{x}_{\epsilon}(\theta),\theta)\mathbb{E}^{\mathbb{Q}^{\beta}(\theta)}_{x}\bigg[\int_{0}^{T\wedge\tau_{n}}\boldsymbol{1}_{\{(0, \alpha_{\beta})\}}(X_{s}^{\xi})ds\bigg]\notag \\
        &\quad+\ell^{\epsilon}(\beta,\theta)\mathbb{E}^{\mathbb{Q}^{\beta}(\theta)}_{x}\bigg[\int_{0}^{T\wedge\tau_{n}}\boldsymbol{1}_{(\alpha_{\beta},\infty)}(X_{s}^{\xi})ds\bigg] \notag \\
        &\leq \overline{V}^{\beta}(x,\theta)-\mathbb{E}^{\mathbb{Q}^{\beta}(\theta)}_{x}\big[ \overline{V}^{\beta}(X^{\xi}_{T\wedge \tau_{n}},\theta)\boldsymbol{1}_{(0,\alpha_{\beta})}(X_{T\wedge \tau_{n}}^{\xi})+\overline{V}^{\beta}(X^{\xi}_{T\wedge \tau_{n}},\theta)\boldsymbol{1}_{(\alpha_{\beta},\infty)}(X_{T\wedge \tau_{n}}^{\xi})\big] \notag \\
        &\quad+\ell^{\epsilon}(\widehat{x}_{\epsilon}(\theta),\theta)\mathbb{E}^{\mathbb{Q}^{\beta}(\theta)}_{x}\bigg[\int_{0}^{T\wedge\tau_{n}}\boldsymbol{1}_{\{(0, \alpha_{\beta})\}}(X_{s}^{\xi})ds\bigg]
        +\ell^{\epsilon}(\beta,\theta)\mathbb{E}^{\mathbb{Q}^{\beta}(\theta)}\big[T\wedge\tau_{n}\big] \notag \\
        &\leq\overline{V}^{\beta}(x,\theta)+\sup_{x\in (0,\alpha_{\beta}]}\big|\overline{V}^{\beta}(x,\theta)\big|-\mathbb{E}^{\mathbb{Q}^{\beta}(\theta)}_{x}\big[\overline{V}^{\beta}(X^{\xi}_{T\wedge \tau_{n}},\theta)\boldsymbol{1}_{(\alpha_{\beta},\infty)}(X_{T\wedge \tau_{n}}^{\xi})\big] \notag \\
        &\quad+\ell^{\epsilon}(\widehat{x}_{\epsilon}(\theta),\theta)\mathbb{E}^{\mathbb{Q}^{\beta}(\theta)}_{x}\bigg[\int_{0}^{T\wedge\tau_{n}}\boldsymbol{1}_{\{(0, \alpha_{\beta})\}}(X_{s}^{\xi})ds\bigg]+\ell^{\epsilon}(\beta,\theta)\mathbb{E}^{\mathbb{Q}^{\beta}(\theta)}\big[T\wedge\tau_{n}\big].
    \end{align}
    Recalling that $\phi_{\beta}$ solves (\ref{Auxiliary Second order ODE}) with $\gamma=0$, from Lemma \ref{lemma: lemma 5 from Cohen} we have that $\phi_{\beta}(x,\theta)\geq c(x)\geq 0$ for $x\in (\alpha_{\beta},\infty)$, which in turn, for $x\in (\alpha_{\beta},\infty)$, gives us
    \begin{equation}
        \label{eq: lower bound of truncated potential function}
        \overline{V}^{\beta}(x,\theta)=\overline{V}^{\beta}(\alpha_{\beta},\theta)+\int_{\alpha_{\beta}}^{x}\overline{\phi}_{\beta}(y,\theta)dy=\overline{V}^{\beta}(\alpha_{\beta},\theta)+\int_{\alpha_{\beta}}^{x}\phi_{\beta}(y,\theta)dy\geq \overline{V}^{\beta}(\alpha_{\beta},\theta).
    \end{equation}
    Using (\ref{eq: lower bound of truncated potential function}) in (\ref{eq: bound before passing ergodic limits}), we find that
    \begin{align}
        &\mathbb{E}^{\mathbb{Q}^{\beta}(\theta)}_{x}\bigg[\int_{0}^{T\wedge \tau_{n}}\bigg(\pi(X^{\xi}_{t},\theta)+\frac{1}{2\epsilon}(\psi^{\beta}(X_{t}^{\xi}))^{2}\bigg)dt+\int_{0}^{T\wedge \tau_{n}}c(X^{\xi}_{s},\theta)\circ d\xi_{s}\bigg]
        \leq \overline{V}^{\beta}(x,\theta)\notag \\
        &\quad+\sup_{x\in (0,\alpha_{\beta}]}\big|\overline{V}^{\beta}(x,\theta) \big|-\overline{V}^{\beta}(\alpha_{\beta},\theta)\mathbb{P}_{x}(X^{\xi}_{T\wedge\tau_{n}}>\alpha_{\beta}) +\ell^{\epsilon}(\widehat{x}_{\epsilon}(\theta),\theta)\mathbb{E}^{\mathbb{Q}^{\beta}(\theta)}_{x}\bigg[\int_{0}^{T\wedge\tau_{n}}\boldsymbol{1}_{\{(0, \alpha_{\beta}]\}}(X_{s}^{\xi})ds\bigg]\notag \\
        &\quad+\ell^{\epsilon}(\beta,\theta)\mathbb{E}^{\mathbb{Q}^{\beta}(\theta)}\big[T\wedge\tau_{n}\big].
    \end{align}
    Letting $n\uparrow\infty$ and invoking the monotone convergence theorem yields
    \begin{align}
        \label{eq: ineq of Ito before passing limit to T}
        &\mathbb{E}^{\mathbb{Q}^{*}(\theta)}_{x}\bigg[\int_{0}^{T}\bigg(\pi(X^{\xi}_{t},\theta)+\frac{1}{2\epsilon}(\psi^{\beta}(X_{t}^{\xi}))^{2}\bigg)dt+\int_{0}^{T}c(X^{\xi}_{s},\theta)\circ d\xi_{s}\bigg] \leq \overline{V}^{\beta}(x,\theta)+\sup_{x\in (0,\alpha_{\beta}]}\big|\overline{V}^{\beta}(x,\theta) \big|\notag \\
        &\quad\quad\quad-\overline{V}^{\beta}(\alpha_{\beta},\theta)\mathbb{P}_{x}(X^{\xi}_{T\wedge\tau_{n}}>\alpha_{\beta})+\ell^{\epsilon}(\beta,\theta)T+\ell^{\epsilon}(\widehat{x}_{\epsilon}(\theta),\theta)\mathbb{E}^{\mathbb{Q}^{\beta}(\theta)}_{x}\bigg[\int_{0}^{T}\boldsymbol{1}_{\{(0, \alpha_{\beta}]\}}(X_{s}^{\xi})ds\bigg].
    \end{align}
    Rearranging terms in (\ref{eq: ineq of Ito before passing limit to T}), dividing by $T$, letting $T\uparrow\infty$, and using the fact that $\overline{V}^{\beta}$ is bounded, we have
    \begin{align}
        \label{eq: equations with liminfs}
        \liminf_{T\uparrow\infty}\frac{1}{T}&\mathbb{E}^{\mathbb{Q}^{\beta}(\theta)}_{x}\bigg[\int_{0}^{T}\bigg(\pi(X^{\xi}_{t},\theta)+\frac{1}{2\epsilon}(\psi^{\beta}(X_{t}^{\xi}))^{2}\bigg)dt+\int_{0}^{T}c(X^{\xi}_{s},\theta)\circ d\xi_{s}\bigg] \notag \\
        &\leq \ell^{\epsilon}(\beta,\theta)+\ell^{\epsilon}(\widehat{x}_{\epsilon}(\theta),\theta)\limsup_{T\to\infty}\frac{1}{T}\mathbb{E}^{\mathbb{Q}^{\beta}(\theta)}_{x}\bigg[\int_{0}^{T}\boldsymbol{1}_{\{(0, \alpha_{\beta}]\}}(X_{s}^{\xi})ds\bigg],
    \end{align}
   where in (\ref{eq: equations with liminfs}) we have used the property $\liminf_{n}(v_{n}+r_{n})\leq \liminf_{n}v_{n}+\limsup_{n} r_{n}$.
    At this point, we aim to show that the controlled process $X^{\xi}$ under $\mathbb{Q}^{\beta}(\theta)$ admits a stationary distribution. To that end, since $\overline{\phi}^{\beta}$ is bounded, we define the scale function $S^{\mathbb{Q}^{\beta}(\theta)}_{x}(x):=\exp\big(\int_{x}^{\beta}\frac{2b(y)}{\sigma^{2}(y)}dy-2\epsilon\int_{x}^{\beta}\overline{\phi}^{\beta}(y,\theta)dy\big)$ and the speed measure $m_{x}^{\mathbb{Q}^{\beta}(\theta)}(x):=\frac{2}{\sigma^{2}(x)S^{\mathbb{Q}^{\beta}(\theta)}_{x}(x)}$, which satisfies Assumption \ref{Ass: Non explosion}. It remains to verify the integrability condition. We have:
    \begin{align}
        \int_{0}^{\infty}m^{\mathbb{Q}^{\beta}(\theta)}_{x}(x)dx&=\int_{0}^{\infty}\frac{2}{\sigma^{2}(x)S^{\mathbb{Q}^{\beta}(\theta)}_{x}(x)}dx=\int_{0}^{\infty}\frac{2}{\sigma^{2}(x)S^{\mathbb{P}}_{x}(x)}\exp\bigg(2\epsilon\int_{x}^{\beta}\overline{\phi}_{\beta}(y,\theta)dy \bigg)dx \notag \\
        &=\int_{0}^{\infty}\frac{2}{\sigma^{2}(x)S^{\mathbb{P}}_{x}(x)}\exp\bigg(2\epsilon\bigg(\int_{x}^{\beta}\overline{\phi}_{\beta}(y,\theta)dy\bigg)\boldsymbol{1}_{(0,\beta)}(x)\notag \\
        &\quad -2\epsilon\bigg(\int_{\beta}^{x}c(y)dy\bigg)\boldsymbol{1}_{[\beta,\infty)}(x)\bigg) dx \notag \\
        &\leq \int_{0}^{\infty}\frac{2}{\sigma^{2}(x)S^{\mathbb{P}}_{x}(x)}\exp\bigg(2\epsilon\bigg(\int_{x}^{\beta}\overline{\phi}_{\beta}(y,\theta)dy\bigg)\boldsymbol{1}_{(0,\beta)}(x)\bigg)dx \notag \\
        &\leq e^{M(\beta,\theta)}\int_{0}^{\infty}\frac{2}{\sigma^{2}(x)S^{\mathbb{P}}_{x}(x)}dx<\infty,
    \end{align}
    where in the first inequality we have used the fact that $c(x)$ is positive, and in the second we applied Proposition \ref{prop: uniform boundness of phi} for $\overline{\phi}_{\beta}$. Then, by ergodicity (see p. 37 in \cite{BorodinSalminen2012}), we find
    \begin{align}
        \label{eq: ergodicity of indicator function}
        \limsup_{T\to\infty}\frac{1}{T}\mathbb{E}^{\mathbb{Q}^{\beta}(\theta)}_{x}\bigg[\int_{0}^{T}\boldsymbol{1}_{\{(0, \alpha_{\beta}]\}}(X_{s}^{\xi})ds\bigg]&=\int_{\mathbb{R}_{+}}\boldsymbol{1}_{\{(0, \alpha_{\beta}]\}}(x)m_{x}^{\mathbb{Q}^{\beta}(\theta)}(x)dx \notag \\
        &=\int_{0}^{\alpha_{\beta}}m_{x}^{\mathbb{Q}^{\beta}(\theta)}(x)dx\leq e^{M(\beta,\theta)}m^{\mathbb{P}}((0,\alpha_\beta)).
    \end{align}
    Hence, using (\ref{eq: ergodicity of indicator function}) in (\ref{eq: equations with liminfs}) and recalling (\ref{Ergodic criterion}), we obtain
    \begin{align}
        \label{eq: final liminf bound}
        \inf_{\mathbb{Q}\in \widehat{\mathcal{Q}}(x)}J^{\epsilon}(x;\xi,\mathbb{Q},\theta)&\leq\liminf_{T\uparrow\infty}\frac{1}{T}\mathbb{E}^{\mathbb{Q}^{\beta}(\theta)}_{x}\bigg[\int_{0}^{T}\bigg(\pi(X^{\xi}_{t},\theta)+\frac{1}{2\epsilon}(\psi^{\beta}(X_{t}^{\xi})(X_{t}^{\xi}))^{2}\bigg)dt+\int_{0}^{T}c(X^{\xi}_{s},\theta)\circ d\xi_{s}\bigg] \notag \\
        &\leq \ell^{\epsilon}(\beta,\theta)+\ell^{\epsilon}(\widehat{x}_{\epsilon}(\theta),\theta)e^{M(\beta,\theta)}m^{\mathbb{P}}((0,\alpha_\beta)).
    \end{align}
    Then, letting $\beta\to\beta_{\epsilon}(\theta)^{-}$, we note that $\alpha_{\beta}\to 0$ and $M(\beta_{\epsilon}(\theta),\theta)$ is bounded (cf. Proposition \ref{prop: uniform boundness of phi}). From the continuity of the mapping $\beta\mapsto \ell^{\epsilon}(\beta,\theta)$, we get
    \begin{equation}
        \inf_{\mathbb{Q}\in \widehat{\mathcal{Q}}(x)}J^{\epsilon}(x;\xi,\mathbb{Q},\theta)\leq\ell^{\epsilon}(\beta_{\epsilon}(\theta),\theta)
    \end{equation}
    for an arbitrary $\xi\in \mathcal{A}_{e}(x)$, which in turn yields
    \begin{equation}
        \sup_{\xi\in \mathcal{A}_{e}(x)}\inf_{\mathbb{Q}\in \widehat{\mathcal{Q}}(x)}J^{\epsilon}(x;\xi,\mathbb{Q},\theta)\leq\lambda^{\epsilon}(\theta).
    \end{equation}
    Finally, combining this with Step 1, we conclude that 
    \begin{equation}
        \label{eq: saddle-point}
        \lambda^{\epsilon}(\theta)=\sup_{\xi\in\mathcal{A}_{e}(x)}\inf_{\mathbb{Q}\in\widehat{\mathcal{Q}}(x)}J^{\epsilon}(x,\xi,\mathbb{Q};\theta).
    \end{equation}  
\end{proof}

\begin{remark}
    \label{remark: saddle-point} As a byproduct of the verification theorem, we have obtained in (\ref{eq: saddle-point}) that the zero-sum game between the decision maker choosing $\xi$ and Nature choosing $\mathbb{Q}$ has a value.
\end{remark}

\section{Mean-Field Equilibrium}
\label{section: MFE part}
In the following, we prove existence and uniqueness of the mean-field equilibrium (cf. Definition
\ref{Def Mean-Field equilibrium}) by an application of Schauder-Tychonof fixed-point theorem. Let $\mathcal{P}(\mathbb{R}_{+},\mathcal{B}(\mathbb{R}_{+}))$ be the space of
probability measures on $\mathbb{R}_{+}$ with the Borel $\sigma$-field, endowed with the weak topology.
\subsection{Continuity and boundedness of the free-boundary with respect to \texorpdfstring{$\theta$}{theta}}
In this subsection, we establish continuity and bounds of the map $\theta \mapsto \beta_{\epsilon}(\theta),\; \theta\in\mathbb{R}_{+}$. For our subsequent analysis, we introduce the following assumptions.

\begin{assumptions}
    \label{Ass: Mean-Field Assumptions}
    \begin{enumerate}
        \item \label{Ass: inada cond of profit fun}   
        \begin{equation}
            \lim_{\theta\downarrow 0}\pi_{x}(x,\theta)=\infty\;\text{ and }\;\lim_{\theta\uparrow \infty}\pi_{x}(x,\theta)=0.
        \end{equation}
        \item \label{Larsy-Lions cond} $\pi_{x\theta}$ is continuous and it is such that $\pi_{x\theta}(x,\theta)<0$, for any $(x,\theta)\in\mathbb{R}_{+}^{2}$.
        \item \label{ass: Lipschtiz property of pi wrt theta} There exist $\delta\in (0,1)$ and $C>0$ such that,
        \begin{equation}
            \big| \pi(x,\theta_{2})-\pi(x,\theta_{1}) \big|\leq C(1+|x|^{\delta})|\theta_{2}-\theta_{1}|,
        \end{equation}
        for any $\theta_{1},\theta_{2}\in\mathbb{R}_{+}$ and $x\in\mathbb{R}_{+}$.
        \item The function $\underline{\ell}^{\epsilon}:\mathbb{R}_{+}\to\mathbb{R}$, with $\underline{\ell}^{\epsilon}(x):=b(x)c(x)-\frac{1}{2}\sigma^{2}(x)(\epsilon c^{2}(x)-c_{x}(x))$, satisfies the following:
        \begin{enumerate}
            \item \label{item: assumption on lambda bar}For any $\epsilon\geq 0$, there exist $\widehat{y}_{\epsilon}\in\mathbb{R}_{+}$ such that,
                \begin{equation}
                    \label{Ass: Cond on der of lambda bar}
                    (\underline{\ell}^{\epsilon})_{x}(x)\begin{cases}
                    >0,\quad x<\widehat{y}_{\epsilon} \\
                    =0,\quad x=\widehat{y}_{\epsilon}\\
                    <0,\quad x>\widehat{y}_{\epsilon}.
                    \end{cases}
                \end{equation}
            \item One has that
            \begin{equation}
                \label{Ass: limit assumption of lambda bar}
                \lim_{x\uparrow\infty}\underline{\ell}^{\epsilon}(x)=-\infty\quad\text{and}\quad\underline{\ell}^{\epsilon}(0):=\lim_{x\downarrow 0}\underline{\ell}^{\epsilon}(x)\quad\text{is finite}.
            \end{equation}
            \item \label{Ass: Zero for lambda bar} One has that $\underline{\widehat{y}}_{\epsilon}:=\inf\{x\geq \widehat{y}_{\epsilon}(\theta):\underline{\ell}^{\epsilon}(x)=\underline{\ell}^{\epsilon}(0)\}$ is finite.
        \end{enumerate}
    \end{enumerate}
\end{assumptions}
\begin{remark}
    Assumptions \ref{Ass: Mean-Field Assumptions}-(\ref{Ass: inada cond of profit fun}), \ref{Ass: Mean-Field Assumptions}-(\ref{Larsy-Lions cond}), and \ref{Ass: Mean-Field Assumptions}-(\ref{ass: Lipschtiz property of pi wrt theta}) are required to establish the existence of a mean-field equilibrium. In particular, Assumption \ref{Ass: Mean-Field Assumptions}-(\ref{Larsy-Lions cond}) is crucial for showing the monotonicity of the free boundary with respect to the mean-field parameter. This monotonicity plays a pivotal role in determining the structure of the set of admissible mean-field parameters and implies the uniqueness of the mean-field equilibrium. From an economic perspective, Assumption \ref{Ass: Mean-Field Assumptions}-(\ref{Ass: inada cond of profit fun}) imposes Inada conditions on our problem. Furthermore, Assumption \ref{Ass: Mean-Field Assumptions}-(\ref{Larsy-Lions cond}) indicates that the marginal net profit, denoted by $\pi_{x}$, decreases as the aggregate index $\theta$ increases. From a technical standpoint, this assumption can be viewed as a stronger version of the Lasry-Lions monotonicity (see Remark 3 in \cite{CaoFerrariDianetti}). Finally, Assumption \ref{Ass: Mean-Field Assumptions}-(\ref{item: assumption on lambda bar}) is essential for introducing an auxiliary problem (independent of $\theta$), which is instrumental for the existence of the mean-field equilibrium. From an applied perspective, this auxiliary problem corresponds to a scenario where the agent faces no competition, a case also analyzed in Section \ref{Section: Case study}.
\end{remark}

Our first result is related to the monotonicity of the map $\theta\mapsto \beta_{\epsilon}(\theta),\; \theta\in\mathbb{R}_{+}$.

\begin{lemma}
    \label{lemma: monotonicity of fb wrt theta}
    The map $\theta\mapsto\beta_{\epsilon}(\theta),\theta\in\mathbb{R}_{+}$, is nonincreasing.
\end{lemma}
\begin{proof}
    Let $\theta_{1},\theta_{2}\in\mathbb{R}_{+}$ with $\theta_{1}<\theta_{2}$. From Proposition \ref{Candidate upper free-boundaries} we know that $\beta_{\epsilon}(\theta_{1})$ is well-defined and we introduce $\phi_{\beta_{\epsilon}(\theta_{1})}(\cdot,\theta_{1})$ and $\phi_{\beta_{\epsilon}(\theta_{1})}(\cdot,\theta_{2})$, which are the unique classical solutions to (\ref{Auxiliary Second order ODE}) for $\gamma=0$, $\beta=\beta_{\epsilon}(\theta_{1})$ and $\theta=\theta_{1}$ and $\theta=\theta_{2}$, respectively. Hence, from Proposition \ref{Proposition: Comparison Principle wrt theta}-(\ref{item: comparison wrt theta}) we obtain that $\phi_{\beta_{\epsilon}(\theta_{1})}(x,\theta_{2})\geq\phi_{\beta_{\epsilon}(\theta_{1})}(x,\theta_{1})\geq c(x)$ for any $x\in (0,\beta_{\epsilon}(\theta_{1})]$, which implies that $\beta_{\epsilon}(\theta_{1})\in B_{\epsilon}(\theta_{2})$ (cf. (\ref{Candidate upper free-boundaries})) and $\beta_{\epsilon}(\theta_{2})=\inf B_{\epsilon}(\theta_{2})\leq \beta_{\epsilon}(\theta_{1})$.
\end{proof}
\begin{lemma}
    \label{lemma: Robust lower bound of free boundary}
    There exists $\underline{\beta}_{\epsilon}\in\mathbb{R}_{+}$, such that $\beta_{\epsilon}(\theta)\geq \underline{\beta}_{\epsilon}$ for any $\theta\in\mathbb{R}_{+}$.
\end{lemma}

\begin{proof}
    Let $\gamma>0$, $\theta\in\mathbb{R}_{+}$, and define the function $\psi^{\gamma}(x,\theta):=\phi_{\beta_{\epsilon}(\theta)}(x,\theta)-\underline{\phi}^{\gamma}(x)$, where $\phi_{\beta_{\epsilon}(\theta)}(\cdot,\theta)$ satisfies (\ref{Auxiliary Second order ODE}) for $\beta=\beta_{\epsilon}(\theta)$ and $\gamma=0$, and $\underline{\phi}^{\gamma}$ satisfies
    \begin{equation}
        \label{eq: lower function wrt theta}
    \begin{cases}
        \frac{1}{2}\sigma^{2}(x)\underline{\phi}^{\gamma}_{x}(x)+b(x)\underline{\phi}^{\gamma}(x)-\frac{\epsilon}{2}\sigma^{2}(x)(\underline{\phi}^{\gamma})^{2}(x)=\underline{\ell}^{\epsilon}(\beta_{\epsilon}(\theta))-\gamma, \quad x\in (0,\beta_{\epsilon}(\theta))\\
        \underline{\phi}^{\gamma}(x)=c(x),\quad x\in [\beta_{\epsilon}(\theta),\infty).
    \end{cases}
    \end{equation}
    Based on the proof of Proposition \ref{E&U of first order ODE}, we can show that the function $\underline{\phi}^{\gamma}$ uniquely solves (\ref{eq: lower function wrt theta}) and it is such that $\underline{\phi}^{\gamma}\in C^{1}(\mathbb{R}_{+})$. Hence, it follows that $\psi^{\gamma}(\cdot,\theta)$ is the unique continuously differentiable solution to
    \begin{equation}
        \label{eq: difference in robust bound result}
    \begin{cases}
        \frac{1}{2}\sigma^{2}(x)(\psi^{\gamma})_{x}(x,\theta)+b(x)\psi^{\gamma}(x,\theta)-\frac{\epsilon}{2}\sigma^{2}(x)(\psi^{\gamma})^{2}(x,\theta)-\epsilon\sigma^{2}\underline{\phi}(x)\psi^{\gamma}(x,\theta)\\
        \quad \quad \quad \quad \quad\;\;\;=\big(\pi(\beta_{\epsilon}(\theta),\theta)-\pi(x,\theta)\big)+\gamma, \quad x\in (0,\beta_{\epsilon}(\theta))\\
        \psi^{\gamma}(x,\theta)=0,\quad x\in [\beta_{\epsilon}(\theta),\infty).
    \end{cases}
    \end{equation}
    Notice that $\pi_{x}(x,\theta)\geq 0$ for any $\theta\in\mathbb{R}_{+}$ (cf. Assumptions \ref{Assumptions for profit fun}-(\ref{Larsy-Lions cond}) and Assumption \ref{Ass: Mean-Field Assumptions}-(\ref{Ass: inada cond of profit fun})), and that, thanks to Assumption \ref{Ass: Mean-Field Assumptions}-(\ref{Ass: inada cond of profit fun}),
    \begin{equation*}
        \pi(\beta_{\epsilon}(\theta),\theta)-\pi(x,\theta)=\int_{x}^{\beta_{\epsilon}(\theta)}\pi_{x}(y,\theta)dy\geq 0,\quad x\leq \beta_{\epsilon}(\theta).
    \end{equation*}
    Then, plugging $x=\beta_{\epsilon}(\theta)$ in (\ref{eq: difference in robust bound result}), we obtain $\psi_{x}^{\gamma}(\beta_{\epsilon}(\theta),\theta)>0$. We want to show that $\psi^{\gamma}(x,\theta)\leq 0$ for any $x\in (0,\beta_{\epsilon}(\theta)]$. Arguing by contradiction, we assume that there exists $z_{0}(\theta):=\sup\{x\in (0,\beta_{\epsilon}(\theta)): \psi^{\gamma}(x,\theta)=0\}$ (which is well-defined due to the continuity of $\psi^{\gamma}(\cdot,\theta)$; cf. Proposition \ref{E&U of first order ODE}, we plug $x=z_{0}(\theta)$ in (\ref{eq: difference in robust bound result})), and we obtain
    \begin{equation}
        \frac{1}{2}\sigma^{2}(z_{0}(\theta))(\psi^{\gamma})_{x}(z_{0}(\theta),\theta)=\int_{z_{0}(\theta)}^{\beta_{\epsilon}(\theta)}\pi_{x}(y,\theta)dy+\gamma>0,
    \end{equation}
    where the last inequality follows from Assumptions \ref{Assumptions for profit fun}-(\ref{Concativity}). Hence, we reach to a contradiction with Lemma \ref{Elementary lemma}. Consequently, we have that $c(x)\leq \phi_{\beta_{\epsilon}(\theta)}(x,\theta)\leq \underline{\phi}^{\gamma}(x)$ for any $x\in (0,\beta_{\epsilon}(\theta)]$. Finally, thanks to (\ref{Ass: Mean-Field Assumptions}), (\ref{Ass: Cond on der of lambda bar}), and (\ref{Ass: limit assumption of lambda bar}), we can mimic the steps of the proof of Proposition \ref{properties of b(a)} to show the existence of $\underline{\beta}_{\epsilon}:=\sup\{x\in (0,\beta_{\epsilon}(\theta)]:\underline{\phi}^{\gamma}(x)\geq c(x)\}>0$, with $\underline{\beta}_{\epsilon}\leq \beta_{\epsilon}(\theta)$. Then we conclude as in Lemma \ref{lemma: monotonicity of fb wrt theta}.
\end{proof}

\begin{lemma}
    \label{lemma: Lip cont of Lambda wrt theta}
    For any $\theta_{1},\theta_{2}\in\mathbb{R}_{+}$, there exists $C_{0}:=C_{0}(\theta_{2})>0$ such that
    \begin{equation}
        \label{eq: Lip cont of Lambda wrt theta}
        \big|\lambda^{\epsilon}(\theta_{2})-\lambda^{\epsilon}(\theta_{1})\big|\leq C_{0}|\theta_{2}-\theta_{1}|.
    \end{equation}
\end{lemma}
\begin{proof}
    For arbitrary $\theta_{1},\theta_{2}\in\mathbb{R}_{+}$, we recall (\ref{second version of functional}) and for convenience we denote $\mathbb{Q}_{i}^{*}:=\mathbb{Q}^{*}(\theta_{i}),\;i=1,2$. Then
    \begin{align}
        \lambda^{\epsilon}(\theta_{2})-\lambda^{\epsilon}(\theta_{1})&=\sup_{\xi\in\mathcal{A}_{e}(x)}\inf_{\mathbb{Q}\in\widehat{\mathcal{Q}}(x)}J^{\epsilon}(x;\xi,\mathbb{Q},\theta_{2})-\sup_{\xi\in\mathcal{A}_{e}(x)}\inf_{\mathbb{Q}\in\widehat{\mathcal{Q}}(x)}J^{\epsilon}(x;\xi,\mathbb{Q},\theta_{1}) \notag \\
        &\leq J^{\epsilon}(x;\xi^{*}(\theta_{2}),\mathbb{Q}_{1}^{*},\theta_{2})-J^{\epsilon}(x;\xi^{*}(\theta_{2}),\mathbb{Q}_{1}^{*},\theta_{1}) \notag \\
        &\leq \limsup_{T\uparrow\infty}\frac{1}{T}\mathbb{E}^{\mathbb{Q}_{1}^{*}}_{x}\bigg[ \int_{0}^{T}\big| \pi(X^{\xi^{*}(\theta_{2})}_{t},\theta_{2})- \pi(X^{\xi^{*}(\theta_{2})}_{t},\theta_{1})\big|dt \bigg] \notag \\
        &\leq C\limsup_{T\uparrow\infty}\frac{1}{T}\mathbb{E}^{\mathbb{Q}_{1}^{*}}_{x}\bigg[ \int_{0}^{T}(1+|X^{\xi^{*}(\theta_{2})}_{t}|^{\delta})dt\bigg]\big|\theta_{2}-\theta_{1}\big|, \notag
    \end{align}
    where in the second inequality we have used the property $\liminf_{n}\alpha_{n}-\liminf_{n}\beta_{n}\leq \limsup_{n}(\alpha_{n}-\beta_{n})$, and in the third inequality Assumption \ref{Ass: Mean-Field Assumptions}-(\ref{ass: Lipschtiz property of pi wrt theta}). Since, $X_{t}^{\xi^{*}(\theta_{2})}\in (0,\beta_{\epsilon}(\theta_{2})],\; \mathbb{Q}_{1}^{*}-$a.s., we obtain that
    \begin{equation}
        \mathbb{E}^{\mathbb{Q}_{1}^{*}}_{x}\big[ \big|X_{t}^{\xi^{*}(\theta_{2})} \big|^{\delta} \big]\leq \beta_{\epsilon}^{\delta}(\theta_{2})<\infty.
    \end{equation}
    Hence,
    \begin{equation}
        \limsup_{T\uparrow\infty}\frac{1}{T}\mathbb{E}^{\mathbb{Q}_{1}^{*}}\bigg[ \int_{0}^{T}(1+|X^{\xi^{*}(\theta_{2})}_{t}|^{\delta})dt\bigg]\leq \limsup_{T\uparrow\infty}\frac{1}{T}\int_{0}^{T}(1+\beta_{\epsilon}^{\delta}(\theta_{2}))dt=1+\beta_{\epsilon}^{\delta}(\theta_{2})<\infty.
    \end{equation}
    Then for $C_{0}(\theta_{2}):=C(1+\beta_{\epsilon}^{\delta}(\theta_{2})$ we conclude.
\end{proof}
Our next result is about the continuity of the map $\theta\mapsto V_{x}^{\epsilon}(x,\theta),\;x\in\mathbb{R}_{+}$.
\begin{proposition}
    \label{prop: Lip cont of phi wrt theta}
    The map $\theta\mapsto V_{x}^{\epsilon}(x,\theta),\;x\in\mathbb{R}_{+}$, is locally Lipschitz continuous.
\end{proposition}
\begin{proof}
    Let $\theta_{1},\theta_{2}\in\mathbb{R}_{+}$. We focus on the case of $\theta_{1}\leq \theta_{2}$, since the proof is analogous in the other case. By Lemma \ref{lemma: monotonicity of fb wrt theta}, we have that $\beta_{\epsilon}(\theta_{2})\leq \beta_{\epsilon}(\theta_{1})$. For $x\in \mathbb{R}_{+}$, we know that $V_{x}^{\epsilon}(\cdot,\theta_{i})$ satisfies (\ref{Auxiliary Second order ODE}), for $\beta=\beta_{\epsilon}(\theta_{i}),\;i=1,2$ and $\gamma=0$. Hence,
    \begin{align}
        V_{x}^{\epsilon}(x,\theta_{1})&-V_{x}^{\epsilon}(x,\theta_{2})=\big(V_{x}^{\epsilon}(x,\theta_{1})-V_{x}^{\epsilon}(x,\theta_{2})\big)\boldsymbol{1}_{\{x\in (0,\beta_{\epsilon}(\theta_{2}))\}}\notag \\
        &+\big(V_{x}^{\epsilon}(x,\theta_{1})-V_{x}^{\epsilon}(x,\theta_{2})\big)\boldsymbol{1}_{\{x\in [\beta_{\epsilon}(\theta_{2}),\beta_{\epsilon}(\theta_{1})]\}}+\big(V_{x}^{\epsilon}(x,\theta_{1})-V_{x}^{\epsilon}(x,\theta_{2})\big)\boldsymbol{1}_{\{x\in [\beta_{\epsilon}(\theta_{1}),\infty)\}} \notag \\
        &=\big(V_{x}^{\epsilon}(x,\theta_{1})-V_{x}^{\epsilon}(x,\theta_{2})\big)\boldsymbol{1}_{\{x\in (0,\beta_{\epsilon}(\theta_{2}))\}}+\big(V_{x}^{\epsilon}(x,\theta_{1})-c(x)\big)\boldsymbol{1}_{\{x\in [\beta_{\epsilon}(\theta_{2}),\beta_{\epsilon}(\theta_{1}))\}},
    \end{align}
    where we have used the fact that $V_{x}^{\epsilon}(x,\theta_{i})=c(x),\; x\in [\beta_{
    \epsilon}(\theta_{1}),\infty)$ for $i=1,2$.

    Take now $x\in [\beta_{\epsilon}(\theta_{2}),\beta_{\epsilon}(\theta_{1}))$ and define $F_{2}^{(\theta_{1},\theta_{2})}(x):=V_{x}^{\epsilon}(x,\theta_{1})-c(x)$. Notice that, actually, by Proposition \ref{section: Existence of a solution to First ODE}, there exists $\overline{M}(\theta_{1},\theta_{2})>0$ such that $\sup_{x\in [\beta_{\epsilon}(\theta_{2}),\beta_{\epsilon}(\theta_{1}))]}\big| F_{2}^{(\theta_{1},\theta_{2})}(x)\big|\leq \overline{M}(\theta_{1},\theta_{2})$.\ We start by showing that there exists $C_{2}(\theta_{1},\theta_{2})>0$ such that $\big|F^{(\theta_{1},\theta_{2})}_{2}(x)\big|\leq C_{2}(\theta_{1},\theta_{2})|\theta_{1}-\theta_{2}|,\;x\in [\beta_{\epsilon}(\theta_{2}),\beta_{\epsilon}(\theta_{1}))$. To that end, notice that by (\ref{Auxiliary Second order ODE}) (for $\beta=\beta_{\epsilon}(\theta_{1})$) $F^{(\theta_{1},\theta_{2})}_{2}$ satisfies
    \begin{align}
            \label{eq: ODE for F_2}
            \frac{1}{2}\sigma^{2}(x)\partial_{x}F^{(\theta_{1},\theta_{2})}_{2}&+b(x)F^{(\theta_{1},\theta_{2})}_{2}-\frac{\epsilon}{2}\sigma^{2}(x)(F_{2}^{(\theta_{1},\theta_{2})})^{2}(x)+\epsilon\sigma^{2}(x)c(x)F^{(\theta_{1},\theta_{2})}_{2}(x) \\
            &=\big(\lambda^{\epsilon}(\theta_{1})-\ell^{\epsilon}(x,\theta_{2})\big)-\big(\pi(x,\theta_{1})-\pi(x,\theta_{2})\big),\;x\in [\beta_{\epsilon}(\theta_{2}),\beta_{\epsilon}(\theta_{1})), \notag
    \end{align}
    with $F_{2}^{(\theta_{1},\theta_{2})}(\beta_{\epsilon}(\theta_{1}))=0$. For $x\in [\beta_{\epsilon}(\theta_{2}),\beta_{\epsilon}(\theta_{1}))$, by Assumption \ref{Assumption for l}-(\ref{Condition of lambdas der 1}) we have that
    \begin{equation}
       \label{eq: difference of lambdas ineq}
       \ell^{\epsilon}(x,\theta_{2})-\lambda^{\epsilon}(\theta_{1})\leq \ell^{\epsilon}(\beta_{\epsilon}(\theta_{2}),\theta_{2})-\lambda^{\epsilon}(\theta_{1})=\lambda^{\epsilon}(\theta_{2})-\lambda^{\epsilon}(\theta_{1}),
    \end{equation}
    and, by the fundamental theorem of calculus (for $x\in [\beta_{\epsilon}(\theta_{2}),\beta_{\epsilon}(\theta_{1}))$, using (\ref{eq: ODE for F_2}) and (\ref{eq: difference of lambdas ineq}) we have
    \begin{align}
        0\leq F^{(\theta_{1},\theta_{2})}_{2}(x)&=F^{(\theta_{1},\theta_{2})}_{2}(\beta_{\epsilon}(\theta_{1})))-\int_{x}^{\beta_{\epsilon}(\theta_{1})}\partial_{x}F^{(\theta_{1},\theta_{2})}_{2}(y)dy \notag \\
        &=-\int_{x}^{\beta_{\epsilon}(\theta_{1})}\frac{2}{\sigma^{2}(y)}\bigg( b(x)+\frac{\epsilon}{2}\sigma^{2}(y)F^{(\theta_{1},\theta_{2})}_{2}(y)-\epsilon\sigma^{2}(y)c(y) \bigg)F^{(\theta_{1},\theta_{2})}_{2}(y)dy \notag \\
        &\quad-\int_{x}^{\beta_{\epsilon}(\theta_{1})}\frac{2}{\sigma^{2}(y)}\bigg(\big(\lambda^{\epsilon}(\theta_{1})-\ell^{\epsilon}(y,\theta_{2})\big)-\big(\pi(y,\theta_{1})-\pi(y,\theta_{2})\big)\bigg)dy \notag \\
        &\leq\int_{x}^{\beta_{\epsilon}(\theta_{1})}\frac{2}{\sigma^{2}(y)}\bigg( +b(x)+\epsilon\sigma^{2}(y)c(y) \bigg)F^{(\theta_{1},\theta_{2})}_{2}(y)dy \notag \\
        &\quad+\int_{x}^{\beta_{\epsilon}(\theta_{1})}\frac{2}{\sigma^{2}(y)}\bigg(\big(\lambda^{\epsilon}(\theta_{2})-\lambda^{\epsilon}(\theta_{1})\big)-\big(\pi(y,\theta_{2})-\pi(y,\theta_{1})\big)\bigg)dy. \label{eq: first ineq for cont wrt theta}
    \end{align}
    Hence,
    \begin{align}
        0\leq F^{(\theta_{1},\theta_{2})}_{2}(x)&\leq \bigg(\int_{x}^{\beta_{\epsilon}(\theta_{1})}\frac{2}{\sigma^{2}(y)}dy\bigg)\big|\lambda^{\epsilon}(\theta_{2})-\lambda^{\epsilon}(\theta_{1})\big|+\int_{x}^{\beta_{\epsilon}(\theta_{1})}\frac{2}{\sigma^{2}(y)}\big| \pi(y,\theta_{2})-\pi(y,\theta_{1}) \big|dy\notag \\
        &\quad+\int_{x}^{\beta_{\epsilon}(\theta_{1})}\bigg(\frac{2|b(x)|}{\sigma^{2}(y)}+2\epsilon\overline{c}\bigg)F_{2}^{(\theta_{1},\theta_{2})}(y) dy.
    \end{align}
    Then, by Grönwall inequality,
    \begin{align}
        0\leq F^{(\theta_{1},\theta_{2})}_{2}(x)&\leq \bigg[\bigg(\int_{x}^{\beta_{\epsilon}(\theta_{1})}\frac{2}{\sigma^{2}(y)}dy\bigg)\big|\lambda^{\epsilon}(\theta_{2})-\lambda^{\epsilon}(\theta_{1})\big|+\int_{x}^{\beta_{\epsilon}(\theta_{1})}\frac{2}{\sigma^{2}(y)}\big| \pi(x,\theta_{1})-\pi(x,\theta_{2}) \big|dy\bigg]\notag \\
        &\quad \cdot\exp{\bigg( \int_{x}^{\beta_{\epsilon}(\theta_{1})}\bigg(\frac{2|b(x)|}{\sigma^{2}(y)}+2\epsilon\overline{c}\bigg)dy\bigg)}, \notag
    \end{align}
    which, using Assumption \ref{Assumptions for profit fun} and Lemma \ref{lemma: Lip cont of Lambda wrt theta}, yields
    \begin{align}
        0&\leq F^{(\theta_{1},\theta_{2})}_{2}(x)\notag \\
        &\leq \bigg(C_{0}\bigg(\int_{x}^{\beta_{\epsilon}(\theta_{1})}\frac{2}{\sigma^{2}(y)}dy\bigg)+ \int_{x}^{\beta_{\epsilon}(\theta_{1})}\frac{2C(1+|y|^{\delta})}{\sigma^{2}(y)}dy\bigg) \notag\\
        &\quad\cdot\exp{\bigg( \int_{x}^{\beta_{\epsilon}(\theta_{1})}\bigg(\frac{2|b(x)|}{\sigma^{2}(y)}+2\epsilon\overline{c}\bigg)dy\bigg)}|\theta_{2}-\theta_{1}|\notag \\
        &\leq\bigg(C_{0}\bigg(\int_{\beta_{\epsilon}(\theta_{2})}^{\beta_{\epsilon}(\theta_{1})}\frac{2}{\sigma^{2}(y)}dy\bigg)+ \int_{\beta_{\epsilon}(\theta_{2})}^{\beta_{\epsilon}(\theta_{1})}\frac{2C(1+|y|^{\delta})}{\sigma^{2}(y)}dy\bigg) \notag\\
        &\quad\cdot\exp{\bigg( \int_{\beta_{\epsilon}(\theta_{2})}^{\beta_{\epsilon}(\theta_{1})}\bigg(\frac{2|b(x)|}{\sigma^{2}(y)}+2\epsilon\overline{c}\bigg)dy\bigg)}|\theta_{2}-\theta_{1}|=:C_{2}(\theta_{1},\theta_{2})|\theta_{2}-\theta_{1}|.
        \label{ineq: Lip property for F1}
    \end{align}
    This gives the desired result. 
    
    Take now $x\in (0,\beta_{\epsilon}(\theta_{2})]$. We want to show that there exists $C_{1}(x,\theta_{1},\theta_{2})>0$ such that, $F_{1}^{(\theta_{1},\theta_{2})}(x):=V^{\epsilon}_{x}(x,\theta_{1})-V^{\epsilon}_{x}(x,\theta_{2})$, satisfies $\big|F_{1}^{(\theta_{1},\theta_{2})}(x)\big|\leq C_{1}(x,\theta_{1},\theta_{2})|\theta_{2}-\theta_{1}|,\; x\in (0,\beta_{\epsilon}(\theta_{2})]$. Notice that $F_{1}^{(\theta_{1},\theta_{2})}$ is the unique classical solution to 
    \begin{align}
            &\frac{1}{2}\sigma^{2}(x)\partial_{x}F_{1}^{(\theta_{1},\theta_{2})}(x)+b(x)F_{1}^{(\theta_{1},\theta_{2})}(x)-\frac{\epsilon}{2}\sigma^{2}(x)(F_{1}^{(\theta_{1},\theta_{2})})^{2}(x)-\epsilon\sigma^{2}(x)V_{x}^{\epsilon}(x,\theta_{1})F_{1}^{(\theta_{1},\theta_{2})}(x)\\
            &=\big(\lambda^{\epsilon}(\theta_{1})-\lambda^{\epsilon}(\theta_{2})\big)-\big(\pi(x,\theta_{1})-\pi(x,\theta_{2})\big),\;x\in (0,\beta_{\epsilon}(\theta_{2})), \notag
    \end{align}
    with $F_{1}^{(\theta_{1},\theta_{2})}(\beta_{\epsilon}(\theta_{2}))=F_{2}^{(\theta_{1},\theta_{2})}(\beta_{\epsilon}(\theta_{2}))$. Following the same steps as in the case of $F_{2}$, we have that
    \begin{align}
        &\big|F^{(\theta_{1},\theta_{2})}_{1}(x)\big|=\big|F^{(\theta_{1},\theta_{2})}_{1}(\beta_{\epsilon}(\theta_{2}))-\int_{x}^{\beta_{\epsilon}(\theta_{2})}\partial_{x}F^{(\theta_{1},\theta_{2})}_{1}(y)dy\big| \notag \\
        &\leq \big|F^{(\theta_{1},\theta_{2})}_{1}(\beta_{\epsilon}(\theta_{2}))\big|+\int_{x}^{\beta_{\epsilon}(\theta_{2})}\frac{2}{\sigma^{2}(y)}\big| b(x)+\frac{\epsilon}{2}\sigma^{2}(y)F^{(\theta_{1},\theta_{2})}_{1}(y)+\epsilon\sigma^{2}(y)V_{x}^{\epsilon}(y,\theta_{1}) \big|\cdot\big|F^{(\theta_{1},\theta_{2})}_{1}(y)\big|dy \notag \\
        &\quad+\int_{x}^{\beta_{\epsilon}(\theta_{2})}\frac{2}{\sigma^{2}(y)}\bigg(\big|\lambda^{\epsilon}(\theta_{2})-\lambda^{\epsilon}(\theta_{1})\big|+\big|\pi(y,\theta_{2})-\pi(y,\theta_{1})\big|\bigg)dy \notag \\
        &\leq \big|F^{(\theta_{1},\theta_{2})}_{1}(\beta_{\epsilon}(\theta_{2}))\big|+\int_{x}^{\beta_{\epsilon}(\theta_{2})}\bigg(\frac{2|b(x)|}{\sigma^{2}(y)}+\epsilon\big|F^{(\theta_{1},\theta_{2})}_{1}(y)\big|+2\epsilon \big|V^{\epsilon}_{x}(y,\theta_{1})\big|\bigg)\big|F^{(\theta_{1},\theta_{2})}_{1}(y)\big|dy \notag \\
        &\quad+\bigg( \int_{x}^{\beta_{\epsilon}(\theta_{2})}\frac{2}{\sigma^{2}(y)}dy\bigg)\big|\lambda^{\epsilon}(\theta_{2})-\lambda^{\epsilon}(\theta_{1}) \big|+\int_{x}^{\beta_{\epsilon}(\theta_{2})}\frac{2}{\sigma^{2}(y)}\big|\pi(y,\theta_{2})-\pi(y,\theta_{1})\big|dy. \label{eq: first eq for F_1}
    \end{align}
    From (\ref{ineq: Lip property for F1}), we know that $\big|F_{2}^{(\theta_{1},\theta_{2})}(x)\big|\leq C_{2}(\theta_{1},\theta_{2})|\theta_{2}-\theta_{1}|$ on $[\beta_{\epsilon}(\theta_{2}),\beta_{\epsilon}(\theta_{1}))$. 
    From Assumption \ref{Ass: Mean-Field Assumptions}-(\ref{ass: Lipschtiz property of pi wrt theta}), Proposition \ref{prop: uniform boundness of phi} and Lemma \ref{lemma: Lip cont of Lambda wrt theta}, we obtain
    \begin{align}
        \label{eq: Lip cont of boundary equation}
        \big|F^{(\theta_{1},\theta_{2})}_{1}(x)\big|&\leq \bigg( C_{2}(\theta_{1},\theta_{2})+C_{0}\int_{x}^{\beta_{\epsilon}(\theta_{2})}\frac{2}{\sigma^{2}(y)}dy+C\int_{x}^{\beta_{\epsilon}(\theta_{2})}\frac{2(1+|y|^{\delta})}{\sigma^{2}(y)}dy \bigg)\big|\theta_{2}-\theta_{1}\big| \notag \\
        &\quad+\int_{x}^{\beta_{\epsilon}(\theta_{2})}\bigg(\frac{2|b(x)|}{\sigma^{2}(y)}+\epsilon(|V^{\epsilon}_{x}(y,\theta_{1})|+|V^{\epsilon}_{x}(y,\theta_{2})|)+2\epsilon |V^{\epsilon}_{x}(y,\theta_{1})|\big)\bigg)\big|F^{(\theta_{1},\theta_{2})}_{1}(y)\big|dy \notag \\
        &=\bigg( C_{2}(\theta_{1},\theta_{2})+C_{0}\int_{x}^{\beta_{\epsilon}(\theta_{2})}\frac{2}{\sigma^{2}(y)}dy+C\int_{x}^{\beta_{\epsilon}(\theta_{2})}\frac{2(1+|y|^{\delta})}{\sigma^{2}(y)}dy \bigg)\big|\theta_{2}-\theta_{1}\big| \notag \\
        &\quad+\int_{x}^{\beta_{\epsilon}(\theta_{2})}\bigg(\frac{2|b(x)|}{\sigma^{2}(y)}+\epsilon(V^{\epsilon}_{x}(y,\theta_{1})+V^{\epsilon}_{x}(y,\theta_{2}))+2\epsilon V^{\epsilon}_{x}(y,\theta_{1})\big)\bigg)\big|F^{(\theta_{1},\theta_{2})}_{1}(y)\big|dy
    \end{align}
    where in the equality we use the fact that $V_{x}^{\epsilon}(y,\theta)$ is positive on $[x,\beta_{\epsilon}(\theta_{2})]$. Then, by Grönwall inequality and Proposition \ref{prop: uniform boundness of phi}, 
    \begin{align}
        \big|F^{(\theta_{1},\theta_{2})}_{1}(x)\big|&\leq \bigg[C_{2}(\theta_{1},\theta_{2})+C_{0}\int_{x}^{\beta_{\epsilon}(\theta_{2})}\frac{2}{\sigma^{2}(y)}dy+C\int_{x}^{\beta_{\epsilon}(\theta_{2})}\frac{2(1+|y|^{\delta})}{\sigma^{2}(y)}dy \bigg]\notag \\
        &\quad\cdot\exp\bigg( \int_{x}^{\beta_{\epsilon}(\theta_{2})}\frac{2|b(x)|}{\sigma^{2}(y)}dy+\epsilon (3M(\theta_{1})+M(\theta_{2}))\bigg)\big|\theta_{2}-\theta_{1}\big|=:\overline{C}_{2}(x,\theta_{1},\theta_{2})\big|\theta_{2}-\theta_{1}\big|. \label{ineq: upper Lip prop of F1}
    \end{align} 
\end{proof}
\begin{corollary}
    \label{cor: joint cont of Vx}
    The map $(x,\theta)\mapsto V^{\epsilon}_{x}(x,\theta)$ is continuous on $\mathbb{R}^{2}_{+}$.
\end{corollary}
Introduce the inaction region $\mathcal{C}$ and the action region $\mathcal{S}$ as it follows:
\begin{align}
    \mathcal{C}&:=\{(x,\theta)\in\mathbb{R}_{+}^{2}: V^{\epsilon}_{x}(x,\theta)> c(x)\} \\
    \mathcal{S}&:=\{(x,\theta)\in\mathbb{R}_{+}^{2}: V^{\epsilon}_{x}(x,\theta)=c(x)\},
\end{align}
and remember that 
\begin{equation}
    \label{eq: free boundary for the joint cont}
    \beta_{\epsilon}(\theta)=\inf\{x\in\mathbb{R}_{+}:V^{\epsilon}_{x}(x,\theta)=c(x)\}.
\end{equation}
We then have the following continuity result.
\begin{theorem}
\label{theorem: continuity of free boundary}
The map $\theta\mapsto \beta_{\epsilon}(\theta)$ is continuous.    
\end{theorem}

\begin{proof}
    We split the proof into two steps.
    \vspace{0.25 cm}
    
    \textbf{Step 1:} In this step we show that the map $\theta\mapsto \beta_{\epsilon}(\theta)$ is right-continuous. Fix $\theta\in\mathbb{R}_{+}$, and a sequence $\{\theta^{n}\}_{n\in\mathbb{N}}$ such that $\theta^{n}\searrow\theta$ as $n\uparrow\infty$. Since the map $\theta\mapsto\beta_{\epsilon}(\theta)$ is nonincreasing (cf. Lemma \ref{lemma: monotonicity of fb wrt theta}) we obtain that $\beta_{\epsilon}(\theta^{n})\leq \beta_{\epsilon}(\theta)$, for any $n\in\mathbb{N}$, which implies $\beta_{\epsilon}(\theta^{+}):=\lim_{n\uparrow\infty}\beta_{\epsilon}(\theta^{n})\leq \beta_{\epsilon}(\theta)$. It remains to show that $\beta_{\epsilon}(\theta)\leq \beta_{\epsilon}(\theta^{+})$. To that end, we observe that $(\theta^{n},\beta_{\epsilon}(\theta^{n}))\in \mathcal{S}$, for any $n\in\mathbb{N}$, and $(\theta^{n},\beta_{\epsilon}(\theta^{n}))\to(\theta,\beta_{\epsilon}(\theta^{+}))$ as $n\uparrow\infty$. Thanks to continuity of $(x,\theta)\mapsto V_{x}^{\epsilon}(x,\theta)$ (cf. Corollary \ref{cor: joint cont of Vx}) we know that $\mathcal{S}$ is closed, hence $(\theta,\beta_{\epsilon}(\theta^{+}))\in\mathcal{S}$. Now, (\ref{eq: free boundary for the joint cont}) gives $\beta_{\epsilon}(\theta)\leq \beta_{\epsilon}(\theta^{+})$.
    \vspace{0.25 cm}

    \textbf{Step 2:} Now we prove that map $\theta\mapsto\beta_{\epsilon}(\theta)$ is left-continuous. In order to do this, we borrow ideas from \cite{DeAngelis_Cont_of_OS} (see also \cite{FerrariDebt}). Arguing by contradiction, we assume that there exists $\theta_{0}\in\Theta$ such that $\beta_{\epsilon}(\theta_{0})<\beta_{\epsilon}(\theta^{-}_{0})$, where $\beta_{\epsilon}(\theta^{-}_{0}):=\lim_{\delta\downarrow 0}\beta_{\epsilon}(\theta_{0}-\delta)$. The limit exists due to the monotonicity of $\theta\mapsto\beta_{\epsilon}(\theta)$ (cf. Lemma \ref{lemma: monotonicity of fb wrt theta}). Then, we can choose $x_{1},x_{2}\in\mathbb{R}_{+}$ such that $\beta_{\epsilon}(\theta_{0})<x_{1}<x_{2}<\beta_{\epsilon}(\theta^{-}_{0})$ and $\theta_{1}<\theta_{0}$. We define a rectangular domain denoted by $\mathcal{R}$ with vertices $(x_{1},\theta_{1}),(x_{1},\theta_{0}),(x_{2},\theta_{1})$ and $(x_{2},\theta_{0})$. Notice that $\mathcal{R}\subset\mathcal{C}$ and $[x_{1},x_{2}]\times \{\theta_{0}\}\subset \mathcal{S}$. From (\ref{free-boundary problem}) we know that $V^{\epsilon}$ satisfies
    \begin{equation}
        \label{eq: ODE for (x,theta)}
        \begin{cases}
                    \mathcal{L}^{\epsilon}V^{\epsilon}(x,\theta)+\pi(x,\theta)=\lambda^{\epsilon}(\theta),\quad (x,\theta)\in [x_{1},x_{2}]\times [\theta_{1},\theta_{0}), \\
                    V_{x}^{\epsilon}(x,\theta_{0})=c(x),\quad x\in [x_{1},x_{2}].
        \end{cases}
    \end{equation}
    Denote by $C_{c}^{\infty}((x_{1},x_{2}))$ the set of functions with infinitely many continuous derivatives and compact support in $(x_{1},x_{2})$. Pick arbitrary $\psi\in C_{c}^{\infty}((x_{1},x_{2}))$ such that $\psi\geq 0$ and $\int_{x_{1}}^{x_{2}}\psi(x)dx>0$, and, for $\theta\in [\theta_{1},\theta_{0})$, multiply the first equation in (\ref{eq: ODE for (x,theta)}) by $\psi$ and integrate both sides over $(x_{1},x_{2})$. This gives
    \begin{equation}
        \int_{x_{1}}^{x_{2}}\big(\mathcal{L}^{\epsilon}V^{\epsilon}(x,\theta)+\pi(x,\theta))\psi(x)dx=\int_{x_{1}}^{x_{2}}\lambda^{\epsilon}(\theta)\psi(x)dx.
    \end{equation}
    Rearranging terms and using integration by parts on the left-hand side we obtain
    \begin{align}
        -\int_{x_{1}}^{x_{2}}\bigg(\frac{1}{2}\sigma^{2}(x)\psi(x)\bigg)_{x}V^{\epsilon}_{x}(x,\theta)dx&=\int_{x_{1}}^{x_{2}}\big(\lambda^{\epsilon}(\theta)-b(x)V_{x}^{\epsilon}(x,\theta)\notag \\
        -\pi(x,\theta)&+\frac{\epsilon}{2}\sigma^{2}(x)(V^{\epsilon}_{x})^{2}(x,\theta)\big)\psi(x)dx.
    \end{align}
    From Proposition \ref{prop: Lip cont of phi wrt theta} we know that the map $\theta\mapsto V_{x}^{\epsilon}(x,\theta),\; x\in\mathbb{R}_{+}$, is continuous. Hence, taking limits as $\theta\uparrow \theta_{0}$, by invoking the dominated convergence theorem, we obtain
    \begin{align}
        -\int_{x_{1}}^{x_{2}}\bigg(\frac{1}{2}\sigma^{2}(x)\psi(x)\bigg)_{x}V^{\epsilon}_{x}(x,\theta_{0})dx&=\int_{x_{1}}^{x_{2}}\big(\lambda^{\epsilon}(\theta_{0})-b(x)V_{x}^{\epsilon}(x,\theta_{0})\notag \\
        -\pi(x,\theta_{0})&+\frac{\epsilon}{2}\sigma^{2}(x)(V^{\epsilon}_{x})^{2}(x,\theta_{0})\big)\psi(x)dx.
    \end{align}
    Since now $V^{\epsilon}_{x}(x,\theta_{0})=c(x),\; x\in [x_{1},x_{2}]$, recalling (\ref{function for eigenvalue}) and that $\lambda^{\epsilon}(\theta)=\ell^{\epsilon}(\beta_{\epsilon}(\theta),\theta)$, applying again integration by parts on the left-hand side and rearranging the terms, we obtain
    \begin{equation}
        \label{eq: integral of eigenvalues}
        \int_{x_{1}}^{x_{2}}\big(\ell^{\epsilon}(\beta_{\epsilon}(\theta_{0}),\theta_{0})-\ell^{\epsilon}(x,\theta_{0})\big)\psi(x)dx=0.
    \end{equation}
    However, the left-hand side of (\ref{eq: integral of eigenvalues}) is strictly negative by Assumption \ref{Assumption for l}-(\ref{Condition of lambdas der 1}). Hence, we have a contradiction.
  \end{proof}

\subsection{Existence and Uniqueness of the Ergodic MFG Equilibrium}

\begin{proposition}
    \label{Prop: existence of stationary distribution}
    For $\epsilon\geq 0$, the following hold:
    \begin{enumerate}
        \item \label{existence of stat dist} For any $\theta\in\mathbb{R}_{+}$, there exists a stationary distribution of $(X^{\xi^{*}(\theta)}_{t})_{t\geq 0}$ under $\mathbb{Q}^{*}(\theta)$, denoted by $\nu^{\theta,\epsilon}\in\mathcal{P}(\mathbb{R}_{+},\mathcal{B}(\mathbb{R}_{+}))$, and its density, denoted by $m^{\theta,\epsilon}$, is such that
    \begin{equation}
        m^{\theta,\epsilon}(x)=\frac{2}{\nu^{\theta,\epsilon}((0,\beta_{\epsilon}(\theta)])\sigma^{2}(x)}\exp{\bigg(-\int_{x}^{\beta_{\epsilon}(\theta)}\frac{2b(x)}{\sigma^{2}(y)}dy+2\epsilon\int_{x}^{\beta_{\epsilon}(\theta)}V_{x}^{\epsilon}(y,\theta)dy \bigg)}\boldsymbol{1}_{(0,\beta_{\epsilon}(\theta)]}(x),
    \end{equation}
    where 
    $$\nu^{\theta,\epsilon}((0,\beta_{\epsilon}(\theta)]):=\int_{0}^{\beta_{\epsilon}(\theta)}\frac{2}{\sigma^{2}(x)}\exp{\bigg(-\int_{x}^{\beta_{\epsilon}(\theta)}\frac{2b(x)}{\sigma^{2}(y)}dy+2\epsilon\int_{x}^{\beta_{\epsilon}(\theta)}V_{x}^{\epsilon}(y,\theta)dy \bigg)}dx$$
    \item \label{continuity of stat dist wrt theta} The map $\theta\mapsto \nu^{\theta,\epsilon}$ is continuous.
    \end{enumerate}
\end{proposition}
\begin{proof}
    Let $x\in\mathbb{R}_{+}$ and $\theta\in\mathbb{R}_{+}$. From Proposition \ref{Prop: Optimal policy}, we know that under $\mathbb{Q}^{*}(\theta)\in\widehat{\mathcal{Q}}(x)$, with $\frac{d\mathbb{Q}^{*}(\theta)}{d\mathbb{P}}\big|_{\mathcal{F}_{t}}=-\epsilon\sigma(X^{\xi^{*}(\theta)}_{t})V^{\epsilon}_{x}(X^{\xi^{*}(\theta)}_{t},\theta)\;\mathbb{P}$-a.s. (cf. Theorem \ref{theorem: Verification result}), the process $(X^{\xi^{*}(\theta)}_{t})_{t\geq 0}$ evolves as 
    \begin{equation}
        \label{eq: optimal diffusion}
        dX^{\xi^{*}(\theta)}_{t}=\big(b(X^{\xi^{*}(\theta)}_{t})-\epsilon\sigma^{2}(X^{\xi^{*}(\theta)}_{t})V^{\epsilon}_{x}(X^{\xi^{*}(\theta)}_{t},\theta)\big)dt+\sigma(X^{\xi^{*}(\theta)}_{t})dW^{\mathbb{Q}^{*}(\theta)}_{t}-d\xi^{*}_{t}(\theta),
    \end{equation}
    with $X^{\xi^{*}(\theta)}_{0}=x$, and $(X^{\xi^{*}(\theta)},\xi^{*}(\theta))$ solving $\textbf{SP}(x,\beta_{\epsilon}(\theta);\mathbb{Q}^{*}(\theta),-\epsilon\sigma(\cdot) V^{\epsilon}_{x}(\cdot,\theta))$ (cf.\ Definition \ref{def Skorokhod}). From Proposition \ref{prop: uniform boundness of phi}, there exists $M:=M(\theta)>0$ such that $|V_{x}^{\epsilon}(x,\theta)|\leq M(\theta)$ for any $x\in [0,\infty)$ (see (\ref{candidate solution})), so that 
    \begin{align}
        \nu^{\theta,\epsilon}((0,\beta_{\epsilon}(\theta)])&=\int_{0}^{\beta_{\epsilon}(\theta)}\frac{2}{\sigma^{2}(x)}\exp\bigg(-\int_{x}^{\beta_{\epsilon}(\theta)}\frac{2b(x)}{\sigma^{2}(y)}dy+2\epsilon\int_{x}^{\beta_{\epsilon}(\theta)}V^{\epsilon}_{x}(y,\theta)dy\bigg)dx \notag \\
        &\leq \int_{0}^{\beta_{\epsilon}(\theta)}\frac{2}{\sigma^{2}(x)}\exp\bigg(-\int_{x}^{\beta_{\epsilon}(\theta)}\frac{2b(x)}{\sigma^{2}(y)}dy+2\epsilon M(\theta)\bigg)dx \notag \\
        &=:\exp(2\epsilon M_{\epsilon}(\theta)\beta_{\epsilon}(\theta))m^{\mathbb{P}}((0,\beta_{\epsilon}(\theta)])<\infty,
    \end{align}
    by Assumption \ref{Ass: Non explosion}.

    Then, from Section 36 of Chapter II in \cite{BorodinSalminen2012} the process $(X^{\xi^{*}(\theta)}_{t})_{t\geq 0}$ is ergodic and has invariant measure $\nu^{\theta,\epsilon}\in\mathcal{P}(\mathbb{R}_{+},\mathcal{B}(\mathbb{R}_{+}))$ with
    \begin{equation}
        \label{eq: stationary distribution}
        \nu^{\theta,\epsilon}((0,x))=\int_{0}^{x}m^{\theta,\epsilon}(y)dy,\quad x\in(0,\beta_{\epsilon}(\theta)].
    \end{equation}
    Claim \ref{continuity of stat dist wrt theta} follows from Proposition \ref{prop: Lip cont of phi wrt theta} and Theorem \ref{theorem: continuity of free boundary}.
\end{proof}
We are now in the position to prove the main result of this section. To this end, recall Assumption \ref{Ass: Regularity of Mean-Field term} we introduce the operator $\mathcal{T}:\mathbb{R}_{+}\to\mathbb{R}_{+}$, as
\begin{equation}
    \label{eq: Mean-Field Operator}
    \mathcal{T}\theta:=F\bigg(\int_{\mathbb{R}_{+}}f(x)\nu^{\theta,\epsilon}(dx)\bigg)=F(\langle f, \nu^{\theta,\epsilon} \rangle),
\end{equation}
where $\langle f,\nu^{\theta,\epsilon}\rangle:= \int_{\mathbb{R}_{+}}f(x)\nu^{\theta,\epsilon}(dx)$. Thanks to the previous results, we can now prove the
existence and uniqueness of a stationary mean-field equilibrium as in Definition \ref{Def Mean-Field equilibrium}.
\begin{theorem}
    \label{theorem: existence and uniqueness of MFE}
    Let Assumptions \ref{Assumption for dynamics of SDE}, \ref{Ass: Non explosion}, \ref{Assumptions for profit fun}, \ref{Ass: Regularity of Mean-Field term}, \ref{Assumption for l} and \ref{Ass: Mean-Field Assumptions} hold. For any $\epsilon\geq0$, there exists a unique $\theta^{\epsilon}\in\mathbb{R}_{+}$ such that $\theta^{\epsilon}=\mathcal{T}\theta^{\epsilon}$.
\end{theorem}

\begin{proof}
    Let $\epsilon\geq 0$. We divide the proof into three steps.
    
    \vspace{0.25 cm}
    \textbf{Step 1: Set of relevant $\theta$.} Let $x\in\mathbb{R}_{+}$ and $\theta\in\mathbb{R}_{+}$. We start by obtaining a lower bound for $\mathcal{T}\theta$, which is uniform with respect to $\theta$. For $\underline{\beta}_{\epsilon}\in \mathbb{R}_{+}$ as in Lemma \ref{lemma: Robust lower bound of free boundary}, and arguing as in Proposition \ref{E&U of first order ODE}, we can find $\underline{V}^{\epsilon}\in C^{2}(\mathbb{R}_{+})$ to be the unique classical solution to the problem
    \begin{equation}
        \label{eq: lower free boundary problem for MFE}
        \begin{cases}
            \frac{1}{2}\sigma^{2}(x)\underline{V}^{\epsilon}_{xx}(x)+b(x)\underline{V}^{\epsilon}_{x}(x)-\frac{\epsilon}{2}\sigma^{2}(x)\big(\underline{V}^{\epsilon}_{x}\big)^{2}(x)=\underline{\lambda}^{\epsilon},\quad x<\underline{\beta}_{\epsilon}, \\
            \underline{V}^{\epsilon}_{x}(x)=c(x),\quad x\geq \underline{\beta}_{\epsilon}.
        \end{cases}
    \end{equation}    
     Recall $\mathbb{Q}^{*}(\theta)$ (cf. Theorem \ref{theorem: Verification result}). Then, Proposition \ref{Prop: Optimal policy} allows to construct an $\mathbb{F}$-adapted pair $(Y^{\zeta},\zeta)\in \mathcal{D}[0,\infty)\times\mathcal{A}$ such that 
    \begin{equation}
        \label{eq: lower diffusion process}
        dY^{\zeta}_{t}=(b(Y^{\zeta}_{t})-\epsilon\sigma^{2}(Y^{\zeta}_{t})\underline{V}^{\epsilon}_{x}(Y^{\zeta}_{t}))dt+\sigma(Y^{\zeta}_{t})dW^{\mathbb{Q}^{*}(\theta)}_{t}-d\zeta_{t},\quad Y^{\zeta}_{0}=x,
    \end{equation}
    and $(Y^{\zeta},\zeta)$ solves $\textbf{SP}(x,\underline{\beta}_{\epsilon};\mathbb{Q}^{*}(\theta),-\epsilon\sigma\underline{V}^{\epsilon}_{x})$. Notice that, while $Y^{\zeta}$ is independent of $\theta$, its expectation under $\mathbb{Q}^{*}(\theta)$ is not. However, this is easily fixed. We define a new complete probability space $(\widehat{\Omega},\widehat{\mathcal{F}},\widehat{\mathbb{Q}})$ supporting a Brownian motion $(\widehat{W}_{t})_{t\geq 0}$, let $(\widehat{\mathcal{F}}^{o}_{t})_{t\geq 0}$ be the filtration generated by Brownian motion $\widehat{W}$, and denote by $\widehat{\mathbb{F}}:=(\widehat{\mathcal{F}}_{t})_{t\geq 0}$ its augmentation with the $\widehat{\mathbb{Q}}$-null sets. Hence, we introduce
    \begin{equation}
        \label{eq: new set of control under new prob space}
        \widehat{\mathcal{A}}:=\{(\widehat{\xi}_{t})_{t\geq 0},\;\widehat{\mathbb{F}}\text{-adapted, nondecreasing, leftcontinuous and such that }\widehat{\xi}_{0}=0,\;\widehat{\mathbb{Q}}\text{-a.s.}\}.
    \end{equation}
    Thanks to Lemma 5.5 in \cite{DeAngelisMilazzo}, since $\zeta\in\mathcal{A}$ we can find $\widehat{\zeta}\in\widehat{\mathcal{A}}$ that is $(\widehat{\mathcal{F}}^{o}_{t})_{t\geq0}$-predictable and such that $\text{Law}_{\mathbb{Q}^{*}(\theta)}(W^{\mathbb{Q}^{*}(\theta)},\zeta)=\text{Law}_{\widehat{\mathbb{Q}}}(\widehat{W},\widehat{\zeta})$. Therefore, from Lemma 5.6 in \cite{DeAngelisMilazzo} we have that
    \begin{equation}
        \label{eq: equality of the laws}
        \text{Law}_{\mathbb{Q}^{*}(\theta)}(W^{\mathbb{Q}^{*}(\theta)},Y^{\zeta},\zeta)=\text{Law}_{\widehat{\mathbb{Q}}}(\widehat{W},\widehat{Y}^{\widehat{\zeta}},\widehat{\zeta}),
    \end{equation}
    where $(\widehat{Y}^{\widehat{\zeta}})$ is the unique strong solution on $(\widehat{\Omega},\widehat{\mathcal{F}},\widehat{\mathbb{Q}},\widehat{\mathbb{F}})$ to
    \begin{equation}
        \label{eq: Y in new prob space}
        d\widehat{Y}^{\widehat{\zeta}}_{t}=\big(b(\widehat{Y}^{\widehat{\zeta}}_{t})-\epsilon\sigma^{2}(\widehat{Y}^{\widehat{\zeta}}_{t})\underline{V}^{\epsilon}_{x}(\widehat{Y}^{\widehat{\zeta}}_{t})\big)dt+\sigma(\widehat{Y}^{\widehat{\zeta}}_{t})d\widehat{W}_{t}-d\widehat{\zeta}_{t},\quad \widehat{Y}^{\widehat{\zeta}}_{0}=x,
    \end{equation}
    subject to $(\widehat{Y}^{\widehat{\zeta}},\widehat{\zeta})$ uniquely solving $\textbf{SP}(x,\underline{\beta}_{\epsilon};\widehat{\mathbb{Q}},-\epsilon\sigma\underline{V}^{\epsilon}_{x})$.
    Hence, we have that
    \begin{equation}
        \label{eq:equality of expecation values for different prob spaces}
        \mathbb{E}^{\widehat{\mathbb{Q}}}_{x}\big[ f(\widehat{Y}^{\widehat{\zeta}}_{t}) \big]=\mathbb{E}^{\mathbb{Q}^{*}(\theta)}_{x}\big[ f(Y_{t}^{\zeta}) \big],\text{ for any }t\geq 0.
    \end{equation}
    Furthermore, arguing as in Proposition \ref{Prop: existence of stationary distribution}, we can show that $(\widehat{Y}^{\widehat{\zeta}}_{t})_{t\geq 0}$ admits a unique stationary distribution denoted by $\underline{\nu}^{\epsilon}\in\mathcal{P}(\mathbb{R}_{+},\mathcal{B}(\mathbb{R}_{+}))$. Thus, by ergodicity of $\widehat{Y}^{\widehat{\zeta}}$ (see Section 36 of Chapter II in \cite{BorodinSalminen2012}), we obtain that
    \begin{equation}
        \langle f,\underline{\nu}^{\epsilon}\rangle=\lim_{T\uparrow\infty}\frac{1}{T}\int_{0}^{T}\mathbb{E}^{\widehat{\mathbb{Q}}}_{x}\big[ f(\widehat{Y}^{\widehat{\zeta}}_{t}) \big]dt,
    \end{equation}
    and we define
    \begin{equation}
        \label{eq: lower mean-field bound}
        \underline{\theta}_{1}^{\epsilon}:=F\big(\langle f,\underline{\nu}^{\epsilon}\rangle\big)\quad\text{ and }\quad\overline{\theta}_{1}^{\epsilon}:=F\big(f(\beta_{\epsilon}(\underline{\theta}^{\epsilon}_{1}))\big).
    \end{equation}
    It is then clear that
    \begin{equation}
        \label{eq: upper mean-field bound}
        \underline{\theta}_{1}^{\epsilon}:=F\big(\langle f,\underline{\nu}^{\epsilon}\rangle\big)\leq F\big(f(\underline{\beta}_{\epsilon})\big)\leq F\big(f(\beta_{\epsilon}(\underline{\theta}^{\epsilon}_{1})\big)=\overline{\theta}^{\epsilon}_{1},
    \end{equation}
    where the first inequality follows from Assumption \ref{Ass: Regularity of Mean-Field term}-(\ref{Ass: Regularity of Mean-Field term-1}) and the second inequality is due to Lemma \ref{lemma: Robust lower bound of free boundary}.

    For $\theta\in [\underline{\theta}_{1}^{\epsilon},\overline{\theta}_{1}^{\epsilon}]$, we know that $\beta_{\epsilon}(\overline{\theta}_{1}^{\epsilon})\leq \beta_{\epsilon}(\theta)\leq \beta_{\epsilon}(\underline{\theta}_{1}^{\epsilon})$ (cf. Lemma \ref{lemma: monotonicity of fb wrt theta}). Furthermore, it holds
    \begin{equation}
        V^{\epsilon}_{x}(x,\theta)=\phi_{\beta_{\epsilon}(\theta)}(x,\theta)\leq \phi_{\beta_{\epsilon}(\underline{\theta}_{1}^{\epsilon})}(x,\theta)\leq \phi_{\beta_{\epsilon}(\underline{\theta}_{1}^{\epsilon})}(x,\overline{\theta}_{1}^{\epsilon})=:\widehat{V}^{\epsilon}_{x}(x),\quad x\in (0,\beta_{\epsilon}(\theta)],
    \end{equation}
    where the first inequality follows from Proposition \ref{prop: comparison princ wrt free-boundary} and the second follows from Proposition \ref{Proposition: Comparison Principle wrt theta}. Additionally, $\widehat{V}^{\epsilon}_{x}$ is the unique solution to (\ref{Auxiliary Second order ODE}) for $\beta=\beta_{\epsilon}(\underline{\theta}_{1}^{\epsilon})$, $\theta=\overline{\theta}_{1}^{\epsilon}$ and $\gamma=0$. We introduce the $\mathbb{F}$-adapted process $(Z_{t})_{t\geq 0}$ with dynamics
    \begin{equation}
        \label{eq: upper diffusion procees}
        dZ_{t}=(b(Z_{t})-\epsilon\sigma^{2}(Z_{t})\widehat{V}^{\epsilon}_{x}(Z_{t}))dt+\sigma(Z_{t})dW_{t}^{\mathbb{Q}^{*}(\theta)},\quad Z_{0}=x.
    \end{equation}
    Thanks to the regularity of $\widehat{V}^{\epsilon}$ (cf. Proposition \ref{E&U of first order ODE}), equation (\ref{eq: upper diffusion procees}) admits a unique strong solution. Arguing as in Proposition \ref{Prop: Optimal policy}, we can find a pair $(Z^{\xi(\underline{\theta}^{\epsilon}_{1})},\xi(\underline{\theta}^{\epsilon}_{1}))\in \mathcal{D}[0,\infty)\times\mathcal{A}$ that uniquely solves $\textbf{SP}(x,\beta_{\epsilon}(\overline{\theta}_{1}^{\epsilon});\mathbb{Q}^{*}(\theta),-\epsilon\sigma\widehat{V}^{\epsilon}_{x})$. Consequently, given that $f$ is increasing, Proposition \ref{Prop: Comparison Principle for reflected/unreflected SDEs} gives us
    \begin{equation}
        \label{eq: inequality of expected values for Z}
        \mathbb{E}^{\mathbb{Q}^{*}(\theta)}_{x}\big[f(Z^{\xi(\underline{\theta}^{\epsilon}_{1})})\big]\leq \mathbb{E}^{\mathbb{Q}^{*}(\theta)}_{x}\big[f(X^{\xi^{*}(\theta)}_{t}) \big],
    \end{equation}
    where $(X^{\xi^{*}(\theta)}_{t})_{t\geq 0}$ is the strong solution to
    \begin{equation}
        \label{eq: optimal diffusion on the new prob space}
        dX_{t}^{\xi^{*}(\theta)}=(b(X_{t}^{\xi^{*}(\theta)})-\epsilon\sigma^{2}(X_{t}^{\xi^{*}(\theta)})V^{\epsilon}_{x}(X_{t}^{\xi^{*}(\theta)},\theta))dt+\sigma(X_{t}^{\xi^{*}(\theta)})dW_{t}^{\mathbb{Q}^{*}(\theta)}-d\xi^{*}_{t}(\theta),
    \end{equation}
    with $\xi^{*}(\theta)$ as in (\ref{eq: Optimal Policy}) (cf. Proposition \ref{Prop: Optimal policy}). Furthermore, using again Lemma 5.5 and Lemma 5.6 from \cite{DeAngelisMilazzo} we can find $(\widehat{Z}^{\widehat{\xi}(\underline{\theta}^{\epsilon}_{1})},\widehat{\xi}(\underline{\theta}^{\epsilon}_{1}))\in\mathcal{D}[0,\infty)\times\widehat{\mathcal{A}}$ such that $(\widehat{Z}^{\widehat{\xi}(\underline{\theta}^{\epsilon}_{1})},\widehat{\xi}(\underline{\theta}^{\epsilon}_{1}))$ solves $\textbf{SP}(x,\beta_{\epsilon}(\overline{\theta}_{1}^{\epsilon});\widehat{\mathbb{Q}},-\epsilon\sigma\widehat{V}^{\epsilon}_{x})$, where $\widehat{Z}^{\widehat{\xi}(\underline{\theta}^{\epsilon}_{1})}$ is the unique strong solution to
    \begin{equation}
        d\widehat{Z}^{\widehat{\xi}(\underline{\theta}^{\epsilon}_{1})}_{t}=(b(\widehat{Z}^{\widehat{\xi}(\underline{\theta}^{\epsilon}_{1})}_{t})-\epsilon\sigma^{2}(\widehat{Z}^{\widehat{\xi}(\underline{\theta}^{\epsilon}_{1})}_{t})\widehat{V}^{\epsilon}_{x}(\widehat{Z}^{\widehat{\xi}(\underline{\theta}^{\epsilon}_{1})}_{t}))dt+\sigma(\widehat{Z}^{\widehat{\xi}(\underline{\theta}^{\epsilon}_{1})}_{t})dW_{t}^{\widehat{\mathbb{Q}}}-d\widehat{\xi}_{t}(\underline{\theta}^{\epsilon}_{1})
    \end{equation}
    and
    \begin{equation}
        \label{eq: equation of expected values for Z}
        \mathbb{E}^{\widehat{\mathbb{Q}}}_{x}\big[ f(\widehat{Z}^{\widehat{\xi}(\underline{\theta}^{\epsilon}_{1})}_{t}) \big]=\mathbb{E}^{\mathbb{Q}^{*}(\theta)}_{x}\big[ f(Z^{\xi(\underline{\theta}^{\epsilon}_{1})}_{t}) \big],\text{ for any }t\geq 0.
    \end{equation} We claim that $(\widehat{Z}_{t}^{\widehat{\xi}(\underline{\theta}^{\epsilon}_{1})})_{t\geq 0}$ admits a stationary distribution under $\widehat{\mathbb{Q}}$. Indeed, arguing as in the proof of Proposition \ref{Prop: existence of stationary distribution}, we can show that there exists unique $\widehat{\nu}^{\epsilon}\in\mathcal{P}(\mathbb{R}_{+},\mathcal{B}(\mathbb{R}_{+}))$. Hence, from ergodicity of $\widehat{Z}^{\widehat{\xi}(\underline{\theta}^{\epsilon}_{1})}$, (\ref{eq: inequality of expected values for Z}) and (\ref{eq: equation of expected values for Z}) we have
    \begin{equation}
        \langle f,\widehat{\nu}^{\epsilon}\rangle=\lim_{T\uparrow\infty}\frac{1}{T}\int_{0}^{T}\mathbb{E}^{\widehat{\mathbb{Q}}}_{x}\big[ f(\widehat{Z}_{t}^{\widehat{\xi}(\underline{\theta}^{\epsilon}_{1})})\big]dt\leq \lim_{T\uparrow\infty}\frac{1}{T}\int_{0}^{T}\mathbb{E}^{\mathbb{Q}^{*}(\theta)}_{x}\big[f(X_{t}^{\xi^{*}(\theta)}) \big]dt=\langle f,\nu^{\theta,\epsilon}\rangle,
    \end{equation}
    which in turn, by monotonicity of $F$, leads to
    \begin{equation}
        \label{eq: lower bound of theta}
        \underline{\theta}^{\epsilon}:=F(\langle f,\widehat{\nu}^{\epsilon}\rangle)\leq \mathcal{T}\theta.
    \end{equation}
    To find an upper bound for $\mathcal{T}\theta$, is easier and we proceed as follows. Since $X_{t}^{\xi^{*}(\theta)}\in (0,\beta_{\epsilon}(\theta)],\;\mathbb{Q}^{*}_{x}(\theta)$-a.s. and thanks to the monotonicity of $f$ and $F$ (see Assumption \ref{Ass: Regularity of Mean-Field term}-(\ref{Ass: Regularity of Mean-Field term-1}))), we obtain
    \begin{equation}
        \label{eq: upper bound of theta}
        \mathcal{T}\theta=F(\langle f,\nu^{\theta,\epsilon}\rangle)\leq F(f(\beta_{\epsilon}(\theta)) \leq F(f(\beta_{\epsilon}(\underline{\theta}^{\epsilon})))=:\overline{\theta}^{\epsilon},
    \end{equation}
    for any $\theta\geq \underline{\theta}^{\epsilon}$. In the last inequality we have used the fact that the map $\theta\mapsto\beta_{\epsilon}(\theta)$ is nonincreasing (cf. Lemma \ref{lemma: monotonicity of fb wrt theta}). Thus, combining (\ref{eq: lower bound of theta}) and (\ref{eq: upper bound of theta}) we conclude that any potential fixed point of $\mathcal{T}$ must lie in the convex, compact set
    \begin{equation}
        \label{eq: set of relevant theta}
        K^{\epsilon}:=[\underline{\theta}^{\epsilon},\overline{\theta}^{\epsilon}]\subset \mathbb{R}_{+}.
    \end{equation}

    \textbf{Step 2: Continuity of $\mathcal{T}$.} We define the map $\mathcal{T}_{1}:K^{\epsilon}\to \mathcal{P}(\mathbb{R}_{+},\mathcal{B}(\mathbb{R}_{+}))$ by
    \begin{equation}
        \label{eq: first mean-field map}
        \mathcal{T}_{1}\theta=\nu^{\theta,\epsilon},\quad \theta\in K^{\epsilon}.
    \end{equation}
    The map is well-defined and continuous thanks to Proposition \ref{Prop: existence of stationary distribution}-(\ref{continuity of stat dist wrt theta}). Next, we denote by $\mathcal{T}_{2}:\mathcal{P}(\mathbb{R}_{+},\mathcal{B}(\mathbb{R}_{+}))\to K^{\epsilon}$ the map
    \begin{equation}
        \label{eq: second mean-field map}
        \mathcal{T}_{2}\nu:=F\bigg( \int_{\mathbb{R}_{+}}f(x)\nu(dx) \bigg).
    \end{equation}
    Since the functions $f$ and $F$ are continuous and the probability measures have compact support, the map $\mathcal{T}_{2}$ is clearly continuous. Concluding, the map $\mathcal{T}:=\mathcal{T}_{2}\circ\mathcal{T}_{1}:K^{\epsilon}\to K^{\epsilon}$ is continuous in the convex compact set $K^{\epsilon}$ and, by Schauder-Tychonof fixed-point theorem (Corollary 17.56 in \cite{aliprantis1999infinite}) there exists $\theta^{\epsilon}\in K^{\epsilon}$ such that $\mathcal{T}\theta^{\epsilon}=\theta^{\epsilon}$.

    \textbf{Step 3: Uniqueness.} Let $\theta^{\epsilon}\in K^{\epsilon}$ be the fixed-point of $\mathcal{T}$ and let $\tilde{\theta}^{\epsilon}\in K^{\epsilon}$ be another fixed-point of $\mathcal{T}$ such that $\theta^{\epsilon}\neq \tilde{\theta}^{\epsilon}$. Without loss of generality, we assume that $\theta^{\epsilon}>\tilde{\theta}^{\epsilon}$. Then, by monotonicity of $\theta\mapsto \beta_{\epsilon}(\theta)$ (cf. Lemma \ref{lemma: monotonicity of fb wrt theta}), we have that $\beta_{\epsilon}(\theta^{\epsilon})\leq \beta_{\epsilon}(\tilde{\theta}^{\epsilon})$, and mimicking Step 1 we obtain that
    \begin{equation}
        \theta^{\epsilon}=\mathcal{T}\theta^{\epsilon}\leq \mathcal{T}\tilde{\theta}^{\epsilon}=\tilde{\theta}^{\epsilon},
    \end{equation}
    which leads to a contradiction.
\end{proof}

\section{A Case Study: A Stylized Dirty-Capacity Reduction Model}
\label{Section: Case study}

This section numerically illustrates our previous findings through a stylized mean-field game of irreversible reductions of emission-intensive productive capacity under Knightian uncertainty.\ The purpose of the example is to provide a tractable benchmark satisfying the assumptions of the paper and to study how the equilibrium reacts to changes in the ambiguity level.

We interpret $X^{0}$ as a firm-level index of dirty productive capacity, residual fossil-based exposure, or emission-intensive activity in the absence of active reductions. Its dynamics are given by
\begin{equation}
    \label{eq: uncontrolled SDE example}
    dX^{0}_{t}=(\kappa- \alpha X^{0}_{t})dt+\sigma X^{0}_{t}dW_{t}^{\mathbb{P}},
    \quad X^{0}_{0}=x,
\end{equation}
where $\kappa>0$ captures the baseline tendency of this capacity index to recover or be maintained, $\alpha>0$ is the mean-reversion speed, and $\sigma>0$ measures exogenous fluctuations in operating conditions. Notice that, for an arbitrary $x_{0}>0$, we have
\begin{equation}
    \label{eq: unattainability of example test}
    S^{\mathbb{P}}_{x}(x)=
    \exp\bigg(
    -\int_{x_{0}}^{x}\frac{2(\kappa-\alpha y)}{\sigma^{2}y^{2}}dy
    \bigg)
    \sim
    \exp\bigg(
    \frac{2\kappa}{\sigma^{2}x}
    +\frac{2\alpha}{\sigma^{2}}\ln x
    \bigg),
    \qquad x\downarrow 0,
\end{equation}
which readily implies that Assumption \ref{Ass: Non explosion} is satisfied.

The singular control $\xi\in \mathcal{A}_{e}(x)$ represents cumulative irreversible reductions of dirty capacity, for instance through retirement, decommissioning, or costly adaptation to environmental constraints. The representative firm receives a constant marginal reward $c>0$ from such reductions. This may be interpreted in reduced form as a retirement value, a subsidy, avoided compliance costs, or a salvage benefit.

On the revenue side, we use the reduced-form isoelastic inverse-demand specification
\[
    p(\theta)=\theta^{-(1+\delta)}, \qquad \delta\in(0,1),
\]
where $\theta$ denotes an aggregate dirty-capacity index. The instantaneous profit is therefore
\[
    \pi(x,\theta)=x^{\delta}p(\theta)
    =x^{\delta}\theta^{-(1+\delta)}.
\]
This specification captures the negative effect of aggregate dirty capacity on individual profitability in a tractable way. At equilibrium, the aggregate index satisfies
\begin{equation}
    \theta=
    \bigg(\int_{\mathbb{R}_{+}}x^{\delta}\nu_{\infty}(dx)\bigg)^{1/\delta},
\end{equation}
where $\nu_{\infty}$ is the stationary distribution of the optimally controlled state.

Knightian uncertainty concerns the firm's imperfect knowledge of the law governing the future evolution of this dirty-capacity index. In reduced form, this may reflect ambiguity about future regulation, carbon prices, operating conditions, demand conditions, input costs, technological obsolescence, or the effective depreciation of emission-intensive capacity. We study how such ambiguity affects the optimal reduction boundary, the stationary distribution, and the corresponding mean-field equilibrium.

From Remarks \ref{remark: benchmark SDE} and \ref{remark: benchmark example} we obtain that Assumptions \ref{Assumption for dynamics of SDE}, \ref{Ass: Non explosion}, \ref{Assumptions for profit fun}, \ref{Ass: Regularity of Mean-Field term}, \ref{Assumption for l} and \ref{Ass: Mean-Field Assumptions} are satisfied. Hence, there exists a unique equilibrium of the ergodic MFG $(\beta_{\epsilon}(\theta^{\epsilon}),\theta^{\epsilon},\mathbb{Q}^{*}(\theta^{\epsilon}))$ which satisfies Definition \ref{Def Mean-Field equilibrium}. Moreover, there exists $\phi_{\beta_{\epsilon}(\theta^{\epsilon})}(\cdot,\theta^{\epsilon})\in C^{1}(\mathbb{R}_{+})$ such that
\begin{align}
    \label{eq: TVP problem case study}
        \frac{1}{2}\sigma^{2}x^{2}(\phi_{\beta_{\epsilon}(\theta^{\epsilon})})_{x}(x,\theta^{\epsilon})
        &+(\kappa-\alpha x)\phi_{\beta_{\epsilon}(\theta^{\epsilon})}(x,\theta^{\epsilon})
        -\frac{\epsilon}{2}\sigma^{2}x^{2}(\phi_{\beta_{\epsilon}(\theta^{\epsilon})})^{2}(x,\theta^{\epsilon})  \\
        &=\lambda^{\epsilon}(\theta^{\epsilon})
        -x^{\delta}(\theta^{\epsilon})^{-(1+\delta)},
        \quad x<\beta_{\epsilon} (\theta^{\epsilon}), \notag \\
        \phi_{\beta_{\epsilon}(\theta^{\epsilon})}(x,\theta^{\epsilon})
        &=c,\quad x\geq \beta_{\epsilon}(\theta^{\epsilon}), \notag
\end{align}
where
\begin{equation}
    \lambda^{\epsilon}(\theta^{\epsilon})
    =
    c\big(\kappa-\alpha\beta_{\epsilon}(\theta^{\epsilon})\big)
    +(\beta_{\epsilon}(\theta^{\epsilon}))^{\delta}
    (\theta^{\epsilon})^{-(1+\delta)}
    -\frac{\epsilon}{2}\sigma^{2}c^{2}(\beta_{\epsilon}(\theta^{\epsilon}))^{2},
\end{equation}
see \eqref{optimal eigenvalue} in Theorem \ref{theorem: Verification result}. In the numerical specification used below, where $\kappa=\alpha=1$, this becomes
\[
    \lambda^{\epsilon}(\theta^{\epsilon})
    =
    c\big(1-\beta_{\epsilon}(\theta^{\epsilon})\big)
    +(\beta_{\epsilon}(\theta^{\epsilon}))^{\delta}
    (\theta^{\epsilon})^{-(1+\delta)}
    -\frac{\epsilon}{2}\sigma^{2}c^{2}(\beta_{\epsilon}(\theta^{\epsilon}))^{2}.
\]

Given that the equilibrium cannot be determined explicitly, we introduce a policy iteration algorithm for its evaluation. The policy iteration method introduced by Bellman \cite{bellman1957dynamic} is a numerical method for solving Hamilton-Jacobi-Bellman equations. Recently, it has been generalized to the case of mean-field games; see Cacace et al. \cite{CacaceCamilliGoffi} and Camilli and Tang \cite{CAMILLI2022126138}. Our algorithm is inspired by \cite{dianetti2024exploratory} and is described in the following table.

\begin{algorithm}
\caption{Policy Iteration Algorithm}
\label{algo: PIA II}
\KwIn{$\theta^{(0)}=\underline{\theta}^{\epsilon}$.}
\For{$n = 0$ \KwTo $N-1$}{
    \KwIn{$\beta^{(0,n)}=\widehat{\underline{x}}_{\epsilon}(\theta^{(n)})$. Find $\phi^{(0,n)}(x,\theta^{(n)})\in C^{1}(\mathbb{R}_{+})$ solution to (\ref{Auxiliary Second order ODE}) for $\beta=\beta^{(0,n)}$ and $\gamma=0$.}
    \For{$k=0$ \KwTo $N-1$}{Calculate
    \begin{equation}
        \beta^{*}=\max\{x\in (\widehat{x}_{\epsilon}(\theta^{(n)}),\beta^{(k,n)}):\phi^{(k)}(x,\theta^{(n)})=c\}.
    \end{equation}
    Find a $\phi^{*}(\cdot,\theta^{(n)})\in C^{1}(\mathbb{R}_{+})$ solution to the following equations:
    \begin{align} 
        \label{algo: eq: solution given beta*}
        &\frac{1}{2}\sigma^{2}(x)(\phi^{*})_{x}(x,\theta^{(n)})+b(x)\phi^{*}(x,\theta^{(n)})-\frac{\epsilon}{2}\sigma^{2}(x)(\phi^{*})^{2}(x,\theta^{(n)})\\
        &\quad\quad\quad=\ell^{\epsilon}(\beta^{*},\theta^{(n)})-\pi(x,\theta^{(n)}),\quad x\leq\beta^{*}, \notag\\
        &\quad\quad\quad\phi^{*}(x,\theta^{(n)})=c,\quad x\geq\beta^{*}.
    \end{align}
    Update policy as follows
    \begin{equation}
        \label{algo: eq: update the new free-boundary}
        \beta^{(k+1,n)}=\begin{cases}
            \beta^{*},\text{ if }\phi^{*}(x,\theta^{(n)})\geq c,\;\text{for any }x\in (0,\beta^{*}], \\
            \beta^{(k,n)},\text{ otherwise.}
        \end{cases}
    \end{equation}}
    Update mean-field equilibrium as follows
    \begin{equation}
        \label{algo: eq: update new theta}
    	\theta^{(n+1)}=\begin{cases}
    	    F\big(\int_{\mathbb{R}_{+}}f(x)\nu^{\beta^{(N,n)}}(dx;\theta^{(n)})\big)=:\theta^{*},\quad\text{if }\theta^{*}\in [\underline{\theta}^{\epsilon},\overline{\theta}^{\epsilon}], \\
            \theta^{(n)},\quad\text{otherwise.}
            \end{cases}
    \end{equation}}
\Return{$(\beta^{(N,N)},\theta^{(N)})$.}
\end{algorithm}

\subsection{Sensitivity Analysis}
In this section, we employ the policy iteration algorithm introduced above to numerically explore the sensitivity of the mean-field equilibrium with respect to the ambiguity parameter $\epsilon$ and the volatility parameter $\sigma$. More precisely, we study how these parameters affect the equilibrium reduction boundary $\beta_{\epsilon}(\theta^{\epsilon})$ and the equilibrium aggregate index $\theta^{\epsilon}$. We then examine the stationary distribution of the optimally controlled dirty-capacity index for different levels of ambiguity. Unless otherwise stated, the baseline parameters are
\[
    \kappa=\alpha=\sigma_{\textnormal{default}}=1,
    \qquad c=1,\qquad \delta=0.6,\qquad \epsilon_{\textnormal{default}}=1.
\]
With the normalization $\kappa=\alpha=1$, the deterministic uncontrolled dynamics revert to the level $\kappa/\alpha=1$. These parameter values are chosen for numerical illustration and are not meant as a calibration exercise.

\subsubsection{Level of Ambiguity $\epsilon$} 
\label{subsection: sensitivity wrt ambiguity}

The ambiguity parameter $\epsilon$ determines the extent to which the representative firm accounts for deviations from the reference model $\mathbb{P}$. Equivalently, it governs the impact of the worst-case probability measure $\mathbb{Q}^{*}$ in the robust control problem. As illustrated in Figure \ref{fig: free-boundary vs. epsilon}, the reflection boundary decreases as the level of ambiguity increases. Thus, under stronger ambiguity, the firm reduces dirty capacity earlier and maintains the controlled state at a lower level. Economically, this reflects a precautionary response to uncertainty about the future drift of the dirty-capacity index, which may capture regulatory pressure, operating conditions, input costs, technological obsolescence, or transition-policy shocks.\ This behavior is consistent with related results for singular stochastic control problems under model uncertainty; see, for instance, Theorem 5.1 in \cite{CohenHening}.

Moreover, Figure \ref{fig: equilibrium price vs. espilon} shows that the equilibrium aggregate index $\theta^{\epsilon}$ decreases as ambiguity increases. This is consistent with the decrease of the reflection boundary: stronger ambiguity induces more aggressive reductions of dirty capacity, which shifts the stationary distribution toward lower values and therefore lowers the aggregate index. As $\epsilon \downarrow 0$, the numerical results are consistent with convergence toward the risk-neutral benchmark. A rigorous analysis of the continuity of the mean-field equilibrium with respect to the ambiguity parameter is left for future research; see Section \ref{sec:outlook}.

\begin{figure}
     \centering
     \begin{subfigure}[b]{0.4\textwidth}
         \centering
         \includegraphics[width=\textwidth]{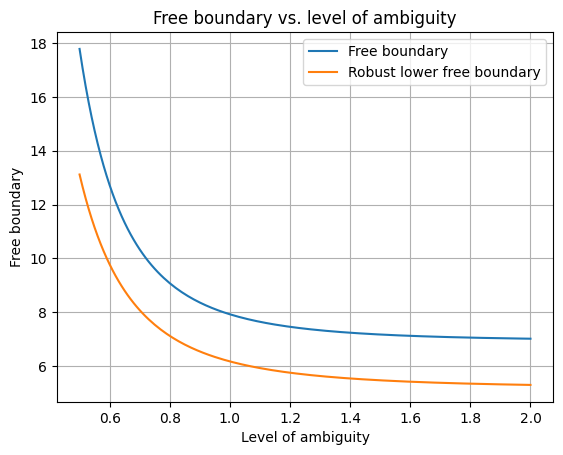}
         \caption{$\beta_{\epsilon}(\theta^{\epsilon})$ vs. $\epsilon$.}
         \label{fig: free-boundary vs. epsilon}
     \end{subfigure}
     \begin{subfigure}[b]{0.4\textwidth}
         \centering
         \includegraphics[width=\textwidth]{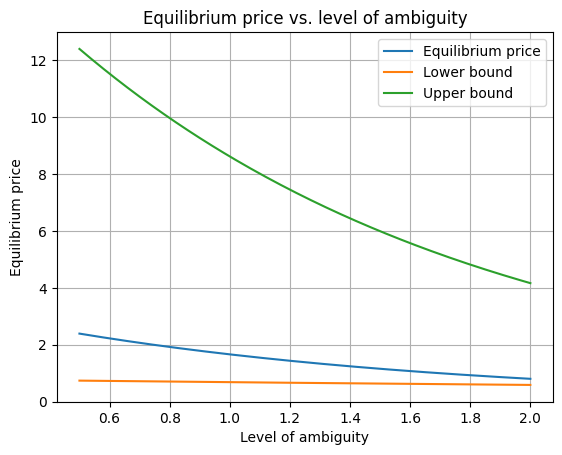}
         \caption{$\theta^{\epsilon}$ vs. $\epsilon$.}
         \label{fig: equilibrium price vs. espilon}
     \end{subfigure}
        \caption{Comparative statics of equilibrium w.r.t. level of ambiguity $\epsilon$.}
        \label{comparative statics wrt epsilon}
\end{figure}

\subsubsection{Level of Volatility $\sigma$}

The volatility parameter $\sigma$ measures the intensity of exogenous fluctuations affecting the dirty-capacity index under the reference model $\mathbb{P}$. As illustrated in Figure \ref{free boundary vs volatility}, an increase in volatility lowers the equilibrium reflection boundary. Thus, higher volatility leads firms to reduce dirty capacity earlier and to maintain the controlled state at a lower level. This reflects a precautionary response to greater uncertainty in the future evolution of emission-intensive activity.

Figure \ref{price equilibrium vs volatility} shows that the equilibrium aggregate index $\theta^{\epsilon}$ also decreases as volatility increases. This is consistent with the behavior of the reflection boundary: stronger fluctuations induce more aggressive reductions of dirty capacity, shifting the stationary distribution toward lower values and reducing the aggregate dirty-capacity index.
\begin{figure}
     \centering
     \begin{subfigure}[b]{0.4\textwidth}
         \centering
         \includegraphics[width=\textwidth]{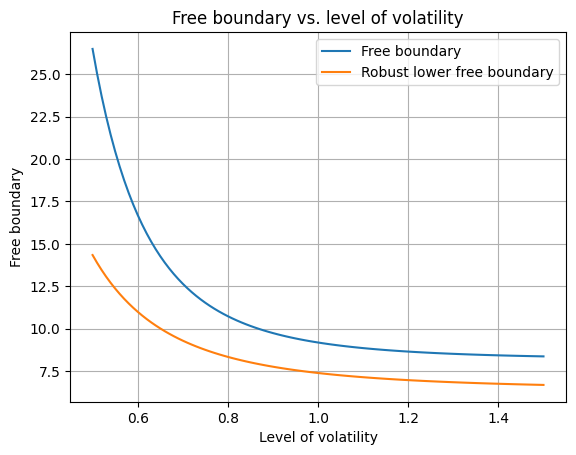}
         \caption{$\beta_{\epsilon}(\theta^{\epsilon})$ vs. $\sigma$.}
         \label{free boundary vs volatility}
     \end{subfigure}
     \begin{subfigure}[b]{0.4\textwidth}
         \centering
         \includegraphics[width=\textwidth]{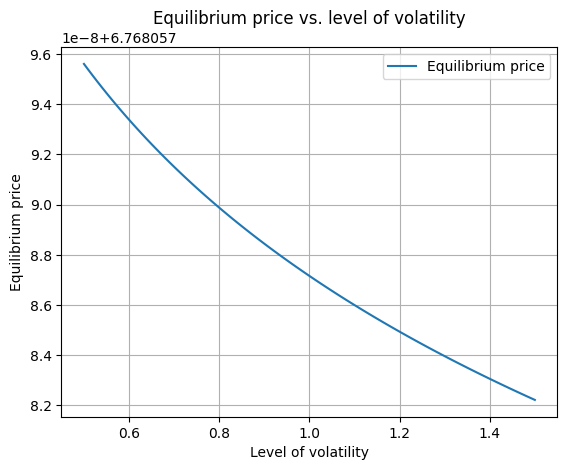}
         \caption{$\theta^{\epsilon}$ vs. $\sigma$.}
         \label{price equilibrium vs volatility}
     \end{subfigure}
        \caption{Comparative statics of equilibria w.r.t. level of volatility $\sigma$.}
        \label{comparative statics wrt volatility}
\end{figure}

\subsubsection{Effect of Ambiguity on the Stationary Distribution}

In continuation of Subsection \ref{subsection: sensitivity wrt ambiguity}, we examine how the stationary distribution of the optimally controlled state changes with the ambiguity parameter. As shown in Figure \ref{stat distributions}, higher levels of ambiguity shift the stationary distribution toward lower values of the dirty-capacity index. This is consistent with the decrease of the reflection boundary documented above: under stronger ambiguity, firms reduce dirty capacity earlier, and the long-run distribution places more mass on lower levels of emission-intensive activity.

\begin{figure}
    \centering
    \begin{subfigure}[b]{0.4\textwidth}
        \includegraphics[width=\textwidth]{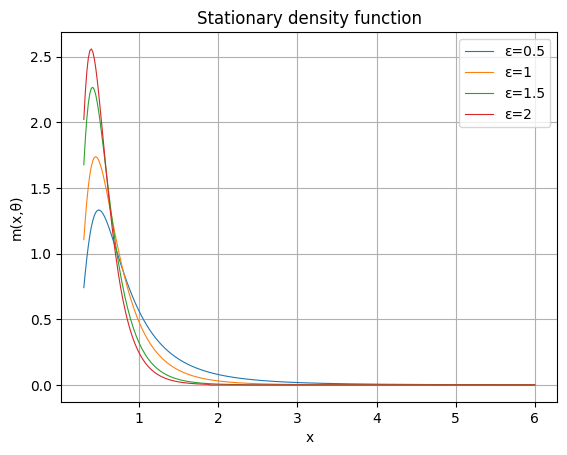}
    \caption{Stationary density function for different levels of ambiguity.}
    \label{stat distributions}
    \end{subfigure}
    \begin{subfigure}[b]{0.5\textwidth}
        \includegraphics[width=\textwidth]{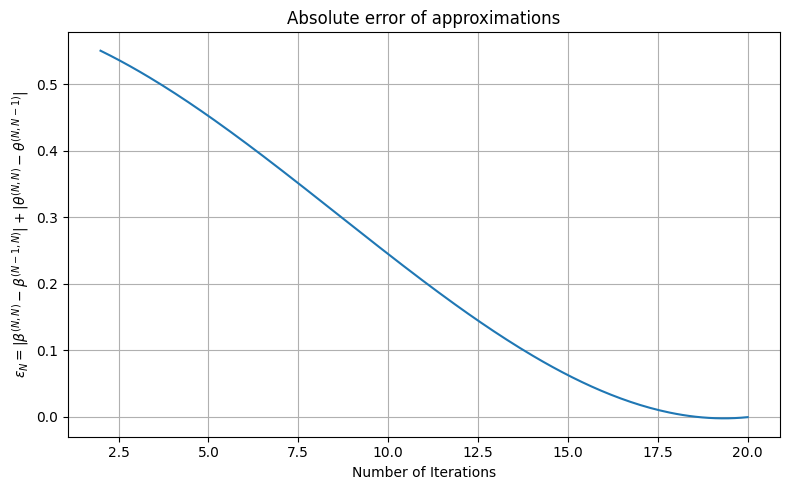}
    \caption{Absolute error of the algorithm for different number of iterations.}
    \label{absolute error}
    \end{subfigure}
\end{figure}

\section{Outlook}
\label{sec:outlook}

In this work, we have studied stationary mean-field games of singular control under model uncertainty, in which a representative agent exerts instantaneous downward control on an It\^{o} diffusion. The agent evaluates a long-run average reward under a worst-case probability measure, while the interaction with the population enters through a scalar functional of the stationary distribution. Under suitable assumptions, we solve the representative agent's robust ergodic singular control problem, characterize the optimal policy through a reflection boundary, construct the stationary distribution under the worst-case measure, and prove existence and uniqueness of a stationary mean-field equilibrium.

The numerical section illustrates the theory through a stylized dirty-capacity reduction benchmark. This example is not intended as a calibrated structural model of the energy sector, but rather as a tractable environment in which the assumptions of the paper can be verified and the equilibrium can be computed. The sensitivity analysis shows how ambiguity and volatility affect the equilibrium reduction boundary, the stationary distribution of the controlled state, and the aggregate dirty-capacity index.

Several natural extensions remain open. A first natural question arises regarding the behavior of the mean-field equilibrium as the level of ambiguity vanishes (the risk-neutral case) or tends to infinity (the degenerate case). More specifically, one might seek to establish the continuity of the mean-field equilibrium with respect to the ambiguity level. This could potentially be achieved by extending existing results for the single-agent problem (cf.\ Theorem 5.1 in \cite{CohenHening}) to our framework and demonstrating that the mean-field operator is locally Lipschitz continuous. The latter requires proving that the free boundary is locally Lipschitz continuous with respect to the mean-field term. However, this step demands analytical techniques distinct from existing methods (see, for instance, \cite{Campi_et.al}). Consequently, we leave this extension for future research.
    
Furthermore, in our framework, we assume from the outset that the mean-field interaction occurs through the stationary distribution. By adapting arguments from \cite{CaoFerrariDianetti}, \cite{ChristensenMordeckiOliu}, and \cite{dianetti2023ergodic}, one can demonstrate two distinct limiting results. The first concerns the connection to a mean-field game with time-dependent interactions. Relying on the ergodicity of the optimally controlled diffusion in the worst-case scenario and Assumption \ref{Ass: Mean-Field Assumptions}-(\ref{ass: Lipschtiz property of pi wrt theta}), one can show that the objective functional with a time-dependent mean-field interaction coincides with that of the stationary distribution. The second limiting result relates to an associated $N$-player game with time-dependent interactions. It is important to note that establishing this connection requires a highly careful analysis; in the worst-case scenario, the state dynamics depend directly on the mean-field interaction—a feature that is absent in the aforementioned works. Consequently, this extension falls beyond the scope of the present paper and is left for future work.

\subsection*{Acknowledgments}
We thank the Associate Editor and the two anonymous reviewers for their valuable comments, which significantly improved the quality of this paper. We also thank Luciano Campi and Asaf Cohen for fruitful discussions and helpful suggestions.\ Moreover, the authors gratefully acknowledge financial support from Deutsche Forschungsgemeinschaft (DFG, German Research Foundation) – Project-ID 317210226 – SFB 1283.


\appendix

\section{Some Technical Results}
\label{section: appendix}
The proof of the next lemma is straightforward.
\begin{lemma}
    \label{Elementary lemma}
   Let $f\in C^{1}((\alpha,\beta))$. Fix $x\in (\alpha,\beta)$ and $\overline{y}:=\sup\{y\in (\alpha,x): f(y)=f(x)\}$ and $\underline{y}:=\inf\{y\in (x,\beta): f(y)=f(x)\}$, if they exist. If $f'(x)> 0$ (resp., $<0$), then $f'(\overline{y})\leq 0$ (resp., $\geq 0$) and $f'(\underline{y})\leq 0$ (resp., $\geq 0$).
\end{lemma}

Recall $\phi_{\cdot}(x,\theta)$ from Section \ref{section: Existence of a solution to First ODE}. We then have the following Lemmata from \cite{CohenHening} (cf. Lemmata 4, 6 and 8).

\begin{lemma}
    \label{lemma: lemma 5 from Cohen}
    For any $\beta>\widehat{x}_{\epsilon}(\theta)$, there exists $y_{1}\in (0,\widehat{x}_{\epsilon}(\theta))$, such that $\phi_{\beta}(x,\theta)\geq c(x),\; x\in (y_{1},\beta]$.
    \end{lemma}

\begin{lemma}
    \label{prop: comparison princ wrt free-boundary}
    Recall $\widehat{x}_{\epsilon}(\theta)$ as in Assumption \ref{Assumption for l}-(\ref{Condition of lambdas der 1}) and let $\widehat{x}^{\epsilon}(\theta)\leq \alpha<\eta$. For  any $x\in (0,\infty)$, one has $\phi_{\alpha}(x,\theta)\leq \phi_{\eta}(x,\theta)$.
\end{lemma}

\begin{lemma}
    \label{lemma: lemma 8 from Cohen}
    For any $\beta\in (\widehat{x}_{\epsilon}(\theta),\beta_{\epsilon}(\theta))$, there exists $y_{2}\in (0,\widehat{x}_{\epsilon}(\theta))$ such that $\phi_{\beta}(x,\theta)\leq c(x)$, for any $x\in (0,y_{2}]$.
\end{lemma}

\begin{proposition}
    \label{prop: uniform boundness of phi}
    For any $\theta\in\mathbb{R}_{+}$ and $\beta\in\mathbb{R}_{+}$ such that $\ell^{\epsilon}(\beta,\theta)\geq 0$, one has
    \begin{equation}
        \int_{x}^{\beta}\phi_{\beta}(y,\theta)dy\leq\begin{cases}
            \pi(\beta,\theta)m^{\mathbb{P}}((0,\beta))\beta=:M(\beta,\theta),\quad x\in (0,\beta) \\
            c(\beta)(x-\beta),\quad x\in [\beta,\infty).
        \end{cases}
    \end{equation}
    Furthermore,
    \begin{equation}
        \lim_{x\downarrow 0}\int_{x}^{\beta}\phi_{\beta}(y,\theta)dy=-\infty.
    \end{equation}
\end{proposition}
\begin{proof}
    \label{Proof of Proposition of uniform boundedness}
    Let arbitrarily fixed $\theta\in\mathbb{R}_{+}$ and $\beta\in \mathbb{R}_{+}$ such that $\ell^{\epsilon}(\beta,\theta)\geq 0$. Note that such a $\beta\in \mathbb{R}_{+}$ exists due to Assumptions~\ref{Assumption for l}-(\ref{Condition of lambdas der 1}) and \ref{Assumption for l}-(\ref{limit assumption}). In the sequel, we denote for simplicity $\phi_{\beta}(\cdot,\theta)$ by $\phi(\cdot,\theta)$. Given that $\phi_{\beta}(x,\theta)=c(x)$ for $x\in [\beta,\infty)$ and $c$ is bounded, it suffices to consider $x\in (0,\beta]$. We then have on $(0,\beta)$
	\begin{align}
		\label{eq: rearranged ODE}
		\frac{d}{dx}&\bigg(\phi(x,\theta)\exp\bigg(-\int_{x}^{\beta}\frac{2b(y)}{\sigma^{2}(y)}dy\bigg)\bigg) \\
        &=\bigg(\phi_{x}(x,\theta)+\frac{2b(x)}{\sigma^{2}(x)}\phi(x,\theta) \bigg)\exp\bigg(-\int_{x}^{\beta}\frac{2b(y)}{\sigma^{2}(y)}dy\bigg) \notag \\
        &=\bigg(\frac{2}{\sigma^{2}(x)}\ell^{\epsilon}(\beta,\theta)-\frac{2\pi(x,\theta)}{\sigma^{2}(x)}+\epsilon\phi^{2}(x,\theta)\bigg)\exp\bigg(-\int_{x}^{\beta}\frac{2b(y)}{\sigma^{2}(y)}dy\bigg) \notag,
	\end{align}
    where in second equation we have used (\ref{Auxiliary Second order ODE}) for $\gamma=0$.
	Then, for $x<\beta$, we integrate (\ref{eq: rearranged ODE}) from $x$ to $\beta$ and exploiting monotonicity of $\pi$ (cf. Assumption \ref{Assumptions for profit fun}-(\ref{Concativity})) and nonnegativity of $\ell^{\epsilon}(\beta,\theta)$ we obtain
    \begin{align}
        \label{eq: rearranged ODE 2}
        \int_{x}^{\beta}\frac{d}{dy}\bigg(\phi(y,\theta)&\exp\bigg(-\int_{y}^{\beta}\frac{2b(z)}{\sigma^{2}(z)}dz\bigg)\bigg)dy\geq \notag \\
        &\quad-\pi(\beta,\theta)\int_{x}^{\beta}\frac{2}{\sigma^{2}(y)} \exp\bigg(-\int_{y}^{\beta}\frac{2b(z)}{\sigma^{2}(z)}dz\bigg)dx.
    \end{align}
    By applying condition $\phi(\beta,\theta)=c(\beta)$, we get
	\begin{align}
		\label{eq: rearranged ODE 3}
		c(\beta)-\phi(x,\theta)&\exp\bigg(-\int_{x}^{\beta}\frac{2b(y)}{\sigma^{2}(y)}dy\bigg)\geq \notag \\
        &\quad-\pi(\beta,\theta)\int_{x}^{\beta}\frac{2}{\sigma^{2}(y)} \exp\bigg(-\int_{y}^{\beta}\frac{2b(z)}{\sigma^{2}(z)}dz\bigg)dy;
	\end{align}
    equivalently,
    \begin{align}
        \label{eq: estimate for phi}
        \phi(x,\theta)&\leq\exp\bigg(\int_{x}^{\beta}\frac{2b(y)}{\sigma^{2}(y)}dy\bigg)\bigg( c(\beta) +\pi(\beta,\theta)\int_{x}^{\beta}\frac{2}{\sigma^{2}(y)} \exp\bigg(-\int_{y}^{\beta}\frac{2b(z)}{\sigma^{2}(z)}dz\bigg)dy\bigg) \notag \\
        &\leq c(\beta)\exp\bigg(\int_{x}^{\beta}\frac{2b(y)}{\sigma^{2}(y)}dy\bigg)\notag \\
        &\quad+\pi(\beta,\theta)\exp\bigg(\int_{x}^{\beta}\frac{2b(y)}{\sigma^{2}(y)}dy\bigg)\int_{x}^{\beta}\frac{2}{\sigma^{2}(y)} \exp\bigg(-\int_{y}^{\beta}\frac{2b(z)}{\sigma^{2}(z)}dz\bigg)dy \notag \\
        &=c(\beta)\exp\bigg(\int_{x}^{\beta}\frac{2b(y)}{\sigma^{2}(y)}dy\bigg)+\pi(\beta,\theta)\int_{x}^{\beta}\frac{2}{\sigma^{2}(y)} \exp\bigg(\int_{x}^{y}\frac{2b(z)}{\sigma^{2}(z)}dz\bigg)dy\notag \\
        &\leq c(\beta)\exp\bigg(\int_{x}^{\beta}\frac{2b(y)}{\sigma^{2}(y)}dy\bigg)+\pi(\beta,\theta)m^{\mathbb{P}}((0,\beta)),
    \end{align}
    where we have used that $S^{\mathbb{P}}_{x}(x)=\exp\big(\int_{x}^{c}\frac{2b(y)}{\sigma^{2}(y)}dy\big)$, and the expression for the speed measure of $X^{0}$ under $\mathbb{P}$, that is, $m^{\mathbb{P}}(x,\alpha)$ as in (\ref{eq: Speed measure under P}). Then, integrating from $x_{0}$ to $\beta_{\epsilon}(\theta)$ (\ref{eq: estimate for phi}) we obtain
    \begin{equation}
        \int_{x_{0}}^{\beta}\phi(y,\theta)dy\leq c(\beta)\int_{x_{0}}^{\beta}S_{x}^{\mathbb{P}}(y)dy+\pi(\beta,\theta)m^{\mathbb{P}}((0,\beta))\beta,\quad x_{0}\in (0,\beta).
    \end{equation}
    Noticing now that $x_{0}\mapsto \int_{x_{0}}^{\beta}S_{x}^{\mathbb{P}}(y)dy=:S^{\mathbb{P}}(x_{0})$ (cf. (\ref{eq: Scale function under P}) and Assumption \ref{Ass: Non explosion}) is such that $S^{\mathbb{P}}(0^{+})=-\infty$ and $S^{\mathbb{P}}(\beta)=0$, $\frac{d}{dx_{0}}S^{\mathbb{P}}(x_{0})<0$, we conclude that 
    \begin{equation}
        \int_{x_{0}}^{\beta}\phi(y,\theta)dy\leq \pi(\beta,\theta)m^{\mathbb{P}}((0,\beta))\beta=:M(\beta,\theta),\quad x_{0}\in (0,\beta).
    \end{equation}
\end{proof}
We are now in the position of proving the continuity of $\beta\mapsto\phi_{\beta}(x,\theta)$. 
\begin{lemma}[Continuity of $\beta\mapsto\phi_{\beta}$]
    \label{Pointwise conv result}
    Recall $\phi_{\beta}$ from Section \ref{section: Existence of a solution to First ODE}. For any fixed $\theta\in\mathbb{R}_{+}$ and fixed $\beta\geq \widehat{x}_{\epsilon}(\theta)$, we have 
    \begin{equation}
        \label{pointwise conv limit}
        \lim_{\delta\downarrow 0}|\phi_{\beta+\delta}(y,\theta)-\phi_{\beta}(y,\theta)|=0,\quad\text{for any }y<\beta.
    \end{equation}
    A similar result holds as $\delta \uparrow 0$.
\end{lemma}

\begin{proof}
    Let $\theta\in\mathbb{R}_{+}$, $y\in (0,\beta)$ and $\delta\in (0,1)$. Given that $\phi_{\beta+\delta}$ and $\phi_{\beta}$ solve (\ref{Auxiliary Second order ODE}) with $\gamma=0$ and $\beta+\delta$ and $\beta$, respectively, for $y\in [x,\beta]$ we have that
    \begin{align}
        \big(\phi_{\beta+\delta}(y,\theta)&-\phi_{\beta}(y,\theta)\big) \\
        &=\big(\phi_{\beta+\delta}(\beta,\theta)-\phi_{\beta}(\beta,\theta)\big)-\int_{y}^{\beta}\big(\phi_{\beta+\delta}(z,\theta)-\phi_{\beta}(z,\theta)\big)_{x}dy \notag \\
        &=\big(\phi_{\beta+\delta}(\beta,\theta)-\phi_{\beta}(\beta,\theta)\big)-\big(\ell^{\epsilon}(\beta+\delta,\theta)-\ell^{\epsilon}(\beta,\theta)\big)\bigg( \int_{y}^{\beta}\frac{2}{\sigma^{2}(z)}dz \bigg) \notag \\
        &\quad-\int_{y}^{\beta} \big(\epsilon(\phi_{\beta+\delta}(z,\theta)+\phi_{\beta}(z,\theta))-\frac{2b(z)}{\sigma^{2}(z)}\big) \big(\phi_{\beta+\delta}(z,\theta)-\phi_{\beta}(z,\theta)\big)dz. \notag
    \end{align}
    From Lemma \ref{prop: comparison princ wrt free-boundary} we know that $\phi_{\beta}(z,\theta)\leq \phi_{\beta+\delta}(z,\theta)\leq \phi_{\beta+1}(z,\theta)$ for any $z\in [y,\beta]$, while from Proposition \ref{prop: uniform boundness of phi} we have that $\int_{x}^{\beta+1}|\phi_{\beta+1}(y,\theta)|dy\leq M(\theta)$, for some $M(\theta)>0$. Hence, we have that
    \begin{align}
        \big| \phi_{\beta+\delta}(y,\theta)&-\phi_{\beta}(y,\theta) \big|\\
        &\leq \big| \phi_{\beta+\delta}(\beta,\theta)-\phi_{\beta}(\beta,\theta) \big| +\big| \ell^{\epsilon}(\beta+\delta,\theta)-\ell^{\epsilon}(\beta,\theta) \big|\bigg( \int_{y}^{\beta}\frac{2}{\sigma^{2}(z)}dy \bigg) \notag \\
        &\quad +\int_{y}^{\beta}\bigg( 2\epsilon|\phi_{\beta+1}(z,\theta)|+\frac{2|b(z)|}{\sigma^{2}(z)}\bigg)\big| \phi_{(\beta+\delta}(z,\theta)-\phi_{\beta}(z,\theta) \big|dz. \notag
    \end{align}
    An application of Grönwall's inequality yields that
    \begin{align}
        \label{eq: Gronwall wrt free boundaries}
        \big| \phi_{\beta+\delta}(y,\theta)&-\phi_{\beta}(y,\theta)\big| \\
        &\leq \bigg( \big| \phi_{\beta+\delta}(\beta,\theta)-\phi_{\beta}(\beta,\theta) \big| +\bigg(\int_{y}^{\beta}\frac{2}{\sigma^{2}(z)}dz\bigg)\big| \ell^{\epsilon}(\beta+\delta,\theta)-\ell^{\epsilon}(\beta,\theta) \big|\bigg) \notag \\
        &\quad \cdot \exp\bigg(2\epsilon M(\theta)+\int_{y}^{\beta}\frac{2|b(z)|}{\sigma^{2}(z)}\bigg)dz\bigg). \notag
    \end{align}
    It remains to show that the right-hand side of (\ref{eq: Gronwall wrt free boundaries}) vanishes as $\delta\downarrow 0$. Let $(\delta_{n})_{n\in \mathbb{N}}$ be an arbitrarily fixed sequence such that $\delta_{n}\downarrow 0$, then introduce $\alpha_{n}:=\phi_{\beta+\delta_{n}}(\beta,\theta)-\phi_{\beta}(\beta,\theta)$. From Lemma \ref{prop: comparison princ wrt free-boundary} one has $\alpha_{n}\geq \alpha_{n+1}$ and $\alpha_{n}\geq 0$ for any $n\in\mathbb{N}$, so that $\lim_{n\uparrow\infty}\alpha_{n}=\inf_{n\in\mathbb{N}}\alpha_{n}=0$. Recalling that $\beta\mapsto \ell^{\epsilon}(\beta,\theta),\; \theta\in\mathbb{R}_{+}$ is continuous (cf. Assumption \ref{Assumption for l}) and taking limits in the right-hand side of (\ref{eq: Gronwall wrt free boundaries}) we obtain
    \begin{align}
        \label{eq: right hand side wrt n}
        \lim_{n\uparrow\infty}\bigg( \big| \phi_{\beta+\delta_{n}}(\beta,\theta)-\phi_{\beta}(\beta,\theta) \big| +\bigg(\int_{y}^{\beta}\frac{2}{\sigma^{2}(z)}dz\bigg)\big| \ell^{\epsilon}(\beta+\delta_{n},\theta)-\ell^{\epsilon}(\beta,\theta) \big|\bigg)=0,
    \end{align}
   which, by (\ref{eq: Gronwall wrt free boundaries}), allows to conclude. The case $\delta\uparrow 0$ can be proved analogously.
\end{proof}

We define the truncated version of $\phi_{\beta}$ as
\begin{equation}
    \label{lemma: truncated phi}
    \overline{\phi}_{\beta}(x,\theta):=\begin{cases}
        \phi_{\beta}(x,\theta),\quad &x> \alpha_{\beta} \\
        c(x),\quad &x\leq\alpha_{\beta}.
    \end{cases}
\end{equation}
We will use the following lemma in the proof of the verification result, Theorem \ref{theorem: Verification result}.\ A similar argument appears in the proof of Proposition 4 in \cite{CohenHening}; however, we present it here in full detail for the sake of completeness.
\begin{lemma}
    \label{lemma: well posedness of truncated problem and limiting behavior}
    Let $\beta\in [\widehat{x}_{\epsilon}(\theta),\beta_{\epsilon}(\theta))$ and $\phi_{\beta}$ satisfy (\ref{Auxiliary Second order ODE}). Then 
    $$\alpha_{\beta}:=\inf\{x>0:\phi_{\beta}(x,\theta)=c(x)\},$$ 
    is such that $\alpha_{\beta}\in (0,\widehat{x}_{\epsilon}(\theta))$ and
    \begin{equation}
        \alpha_{\beta}\to 0^{+},\quad\text{as}\quad \beta\to \beta_{\epsilon}(\theta)^{-}.
    \end{equation}
\end{lemma}
\begin{proof}
    Let $\beta\in (\widehat{x}_{\epsilon}(\theta),\beta_{\epsilon}(\theta))$. Recall that $\phi_{\beta}(x,\theta)$ satisfies \eqref{Auxiliary Second order ODE} with $\gamma=0$ for all $x\in (0,\beta]$; hence by structure of $\ell^{\epsilon}(\beta,\theta)$ (cf. (\ref{function for eigenvalue})) it cannot be equal to $c(x)$ on this interval. Then, invoking Lemmata \ref{lemma: lemma 5 from Cohen} and \ref{lemma: lemma 8 from Cohen}, and using the fact that $\phi_{\beta}(\cdot,\theta)\in C^{1}((0,\beta))$, we deduce the existence of $\alpha_{\beta}\in (0,\widehat{x}_{\epsilon}(\theta))$. It remains to show that $\alpha_{\beta}\to 0^{+}$ as $\beta \to \beta_{\epsilon}(\theta)^{-}$. Taking arbitrary sequence $\beta^{\delta}$ such that $\beta^{\delta}\uparrow\beta_{\epsilon}(\theta)$ as $\delta\to 0$, then from Lemma \ref{Pointwise conv result}, we know that
    \begin{equation}
        \sup_{y<\beta_{\epsilon}(\theta)}\big|\phi_{\beta^{\delta}}(y,\theta)-\phi_{\beta_{\epsilon}(\theta)}(y,\theta)|\to 0,\quad \delta\to 0.
    \end{equation}
    Hence, the truncated version $\overline{\phi}_{\beta^{\delta}}(\cdot,\theta)$ converges as well to $\phi_{\beta_{\epsilon}(\theta)}(\cdot,\theta)$ on $(0,\beta_{\epsilon}(\theta)]$, which implies that $\alpha_{\beta^{\delta}}\to 0^{+}$ as $\delta\to 0$.
\end{proof}

\begin{proposition}[Comparison principle wrt $\theta$]
    \label{Proposition: Comparison Principle wrt theta}
    Recall $\phi_{\cdot}(x,\theta)$ as in Section \ref{section: Existence of a solution to First ODE}. For any $\beta\in\mathbb{R}_{+}$ and any $\theta_{1},\theta_{2}\in\mathbb{R}_{+}$ with $\theta_{1}\leq \theta_{2}$, the following hold:
    \begin{enumerate}
        \item  \label{item: comparison wrt theta}$\phi_{\beta}(x,\theta_{1})\leq \phi_{\beta}(x,\theta_{2})$ for any $x\in\mathbb{R}_{+}$.
        \item \label{item: comparison with robust function}Let $\underline{\phi}_{\beta}\in C^{1}(\mathbb{R}_{+})$ be the unique solution to 
        \begin{align}
            \frac{1}{2}\sigma^{2}(x)\partial_{x}\underline{\phi}_{\beta}(x)+b(x)\underline{\phi}_{\beta}(x)-\frac{\epsilon}{2}\sigma^{2}(x)(\underline{\phi}_{\beta}(x))^{2}=\underline{\ell}^{\epsilon}(\beta),\quad x\leq \beta, \\
            \underline{\phi}_{\beta}(x)=c(x),\quad x\geq \beta, \notag
        \end{align}
        with $\underline{\ell}^{\epsilon}(\beta):=b(\beta)c(\beta)-\frac{1}{2}\sigma^{2}(\beta)(\epsilon c^{2}(\beta)-c_{x}(\beta))$. Then, $\phi_{\beta}(x,\theta)\geq \underline{\phi}_{\beta}(x)$ for any $x\in\mathbb{R}_{+}$.
    \end{enumerate}
\end{proposition}
\begin{proof}
    We only prove (\ref{item: comparison wrt theta}) as the proof of item (\ref{item: comparison with robust function}) will be analogous. Let $\gamma<0$ and denote by $\phi^{\gamma}_{\beta}(x,\theta)$ the classical solution to (\ref{Auxiliary Second order ODE}), and set $\phi_{\beta}(x,\theta):=\phi^{0}_{\beta}(x,\theta)$. We define $\psi^{\gamma}(x):=\phi_{\beta}^{\gamma}(x,\theta_{2})-\phi_{\beta}(x,\theta_{1})$, so that $\psi$ satisfies
    \begin{equation}
    \label{eq: difference in comparison principle result}
    \begin{cases}
        \frac{1}{2}\sigma^{2}(x)(\psi^{\gamma})_{x}(x)+b(x)\psi^{\gamma}(x)-\frac{\epsilon}{2}\sigma^{2}(x)(\psi^{\gamma})^{2}(x)-\epsilon\sigma^{2}\phi^{\gamma}_{\beta}(x,\theta_{2})\psi^{\gamma}(x)\\
        \quad \quad \quad \quad \quad\;\;\;=\big(\pi(\beta,\theta_{2})-\pi(\beta,\theta_{1})\big)-\big(\pi(x,\theta_{2})-\pi(x,\theta_{1})\big)+\gamma, \quad x\in (0,\beta),\\
        \psi^{\gamma}(x)=0,\quad x\in [\beta,\infty).
    \end{cases}
\end{equation}
    For arbitrarily fixed $x\leq \beta$, we rewrite the right-hand side of (\ref{eq: difference in comparison principle result}) as follows
    \begin{align}
        \label{eq: representation of righ-hand side appendix}
        \big( \ell^{\epsilon}(\beta,\theta_{2})-\ell^{\epsilon}(\beta,\theta_{1}) \big)&-\big(\pi(x,\theta_{2})-\pi(x,\theta_{1})\big)+\gamma \notag \\
        &=\big( \pi(\beta,\theta_{2})-\pi(\beta,\theta_{1}) \big)-\big(\pi(x,\theta_{2})-\pi(x,\theta_{1})\big)+\gamma \notag \\
        &=\int_{\theta_{1}}^{\theta_{2}}\pi_{\theta}(\beta,\theta)d\theta-\int_{\theta_{1}}^{\theta_{2}}\pi_{\theta}(x,\theta)d\theta+\gamma=\int_{x}^{\beta}\int_{\theta_{1}}^{\theta_{2}}\pi_{x \theta}(y,\theta)d\theta dy+\gamma,
    \end{align}
    In order to show that $\psi(x):=\psi^{0}(x)\geq 0$ for any $x\leq \beta$ we argue by contradiction. We assume that there exists $x_{1}:=\sup\{ x\in (0,\beta): \psi^{\gamma}(x)=0\}$. Then, plugging $x=\beta$ in (\ref{eq: difference in comparison principle result}) and using the boundary condition $\psi^{\gamma}(\beta)=0$ we obtain $\frac{1}{2}\sigma^{2}(\beta)\psi_{x}^{\gamma}(\beta)=\gamma$, which implies, $\psi_{x}^{\gamma}(\beta)<0$. Therefore, by plugging $x=x_{1}$ in (\ref{eq: difference in comparison principle result}) we have 
    \begin{align}
        \frac{1}{2}\sigma^{2}(x_{1})\psi^{\gamma}_{x}(x_{1})=\int_{x_{1}}^{\beta}\int_{\theta_{1}}^{\theta_{2}}\pi_{x \theta }(y,\theta)d\theta dy+\gamma<0,
    \end{align}
    where last inequality is due to Assumption \ref{Assumptions for profit fun}-(\ref{Larsy-Lions cond}). This implies $\psi_{x}^{\gamma}(x_{1})<0$ and contradicts Lemma \ref{Elementary lemma}. Finally, sending $\gamma\to 0^{-}$, we complete the proof by Lemma \ref{Pertrubed solutions lemma}.
\end{proof}

\begin{proposition}[Comparison principle for singularly controlled SDEs]
    \label{Prop: Comparison Principle for reflected/unreflected SDEs}
    Let $\mathbb{Q}\in \mathcal{P}(\Omega,\mathcal{F})$, $x_{1},x_{2}\in\mathbb{R}_{+}$ and $\theta_{1},\theta_{2}\in\mathbb{R}_{+}$ such that $x_{1}\leq x_{2}$ and $\theta_{1}\leq \theta_{2}$, then the following hold:
    \begin{enumerate}
        \item $X_{t}^{x_{1},\xi^{*}(\theta_{1})}\leq X_{t}^{x_{1},\xi^{*}(\theta_{2})},\; \mathbb{Q}\otimes dt-$a.s.;
        \item let $\psi_{1},\psi_{2}:\mathbb{R}_{+}\to\mathbb{R}$ be locally Lipschitz functions with $\psi_{1}(x)\leq \psi_{2}(x)$ for any $x\in\mathbb{R}_{+}$. Then, if $(X_{t}^{(i)})_{t\geq 0},\;i=1,2$ are strong solutions to
        \begin{equation}
            dX^{(i)}_{t}=\big(b(X^{(i)}_{t})+\psi_{i}(X^{(i)}_{t})\sigma(X^{(i)}_{t})\big)dt+\sigma(X^{(i)}_{t})dW^{\mathbb{Q}}_{t}-d\xi_{t},\quad i=1,2
        \end{equation}
        one has that $X_{t}^{(1)}\leq X_{t}^{(2)},\;\mathbb{Q}\otimes dt-$a.s.
    \end{enumerate}
\end{proposition}
\begin{proof}
    The proof follows by combining Proposition 5.2.18 in \cite{karatzas1991brownian} with Theorem 1.4.1 in \cite{pilipenko}.
\end{proof}

\subsection*{Acknowledgments}
We thank the Associate Editor and the two anonymous reviewers for their valuable comments, which significantly improved the quality of this paper. We also thank Luciano Campi and Asaf Cohen for fruitful discussions and helpful suggestions.\ Moreover, the authors gratefully acknowledge financial support from Deutsche Forschungsgemeinschaft (DFG, German Research Foundation) – Project-ID 317210226 – SFB 1283.

\bibliographystyle{apalike}

\bibliography{bibliography}

\end{document}